\documentclass{article}
\usepackage{mathtools}
\usepackage{graphicx} 
\usepackage{bbm}
\usepackage{stmaryrd,mathrsfs}
\usepackage{lipsum}
\usepackage{bbm}
\usepackage{amsthm}
\usepackage{amsmath,amsfonts,amssymb}
\usepackage{array}
\usepackage{multicol}
\usepackage{empheq}
\usepackage{enumitem}

\theoremstyle{definition}
\newtheorem{theorem}{Theorem}[section]
\newtheorem{proposition}[theorem]{Proposition}
\newtheorem{corollary}[theorem]{Corollary}
\newtheorem{lemma}[theorem]{Lemma}

\newtheorem{remark}[theorem]{Remark}
\newtheorem{definition}[theorem]{Definition}
\newtheorem{exemple}[theorem]{Example}
\newtheorem*{sketch of proof}{Sketch of proof}
\newtheorem{Assumption}{Assumption}

\newtheorem*{Physics Principle}{Generalized Hamilton Principle}

\usepackage{empheq}
\usepackage{caption}
\usepackage{xcolor}
\newcommand{\R}{\mathbb{R}}

\RequirePackage[left=2.5cm,right=2.5cm,top=2.5cm,bottom=3cm]{geometry} 
\usepackage[hidelinks]{hyperref}

\newcommand{\wt}[1]{\widetilde{#1}}

\newcommand{\ol}[1]{\overline{#1}}

\newcommand{\Id}{\mathrm{Id}}

\newcommand{\n}[1]{\left\Vert #1\right\Vert}
\newcommand{\la}{\left\langle}
\newcommand{\ra}{\right\rangle}

\newcommand{\N}{\mathbb{N}}

\newcommand{\mscr}[1]{\textrm{\fontencoding{OMS}\fontfamily{ztmcm}\selectfont#1}
}
\newcommand{\mc}[1]{\mathcal{#1}}
\newcommand{\indic}{\mathbbm{1}}

\newcommand{\pa}[1]{{{\partial}\over{\partial #1}}}
\newcommand{\diff}{d}

\newcommand{\dif}[1]{\frac{\diff}{\diff #1}}
\newcommand{\longto}{\mathop{\longrightarrow}}
\newcommand{\crg}{_{{\mathfrak{g}^*}\!,\mathfrak{g}}}

\newcommand{\GL}{\mathrm{GL}}
\newcommand{\gl}{\mathfrak{gl}}
\newcommand{\SO}{\mathrm{SO}}
\newcommand{\so}{\mathfrak{so}}
\newcommand{\SE}{\mathrm{SE}}
\newcommand{\se}{\mathfrak{se}}

\newcommand{\id}{\mathrm{id}}

\newcommand{\qmatrix}[1]{ \left( \begin{matrix} #1 \end{matrix} \right) }

\def\[#1\]{\begin{align*}#1\end{align*}}
\def\be#1\ee{\begin{align}#1\end{align}}
\def\bea#1\eea{\begin{align}#1\end{align}}
\def\ben#1\een{\begin{align*}#1\end{align*}}

\newcommand{\ad}{\mathrm{ad}}
\newcommand{\Lie}{\mathrm{Lie}}

\begin{document}

\title{Bio-inspired control model for snake locomotion based on Cosserat beam theory}
\author{Eliot Thys, Frédéric Boyer, Yacine Chitour, \\ Petri Kokkonen, Vincent Lebastard, Swann Marx, Clément Moreau}
\date{\today}

\maketitle

\tableofcontents

\section{Introduction}

 This paper addresses the modeling of locomotion in snake-like and other slender bodies, including fish, nematodes, worms, and microorganisms, as well as robotic systems designed to mimic such forms of locomotion. A substantial body of literature has been devoted to this topic, motivated primarily by two considerations. First, snake locomotion remains a fascinating subject of biological research, whose underlying mechanisms are still not fully understood, particularly in the presence of complex interactions with solid environments, \cite{Gray1946}. Second, this type of locomotion provides a major source of inspiration for the design of continuum robots and lies at the heart of the so-called soft robotics paradigm, \cite{BurgnerKahrs2015}. There exists a wide variety of models addressing the locomotion of snakes. Finite-dimensional models typically rely on discrete multibody representations, in which a finite number of rigid vertebrae are connected pairwise through lumped joints \cite{Choseta},\cite{Liljeback}, \cite{Ostrowski96gaitkinematics}, \cite{Boyer2011}. When snakes move on solid ground with full-body contact, they commonly employ a locomotion mode known as lateral undulation. In this mode, the body appears to propagate according to a so-called ``leader-follower'' mechanism, in which each vertebra follows the motion of the head with a certain time delay, \cite{Ostrowski96gaitkinematics},\cite{Choset2009}. However, these discrete models have inherent limitations when compared with real snakes, which possess more than 300 vertebrae embedded in muscles and connective tissues, effectively making them continuum-like bodies. Modeling their locomotion therefore calls for a continuous representation. To address these limitations, Cosserat rod theory has proved to be an effective framework for describing and simulating the locomotion of elongated animals, such as fish swimming \cite{boyer2010poincare} and snake lateral undulation \cite{boyer2011macrocontinuous}. In this framework, several rod kinematic descriptions, including Kirchhoff and Simo--Reissner models, as well as different types of kinematic constraints, have been implemented to study the terrestrial locomotion of slender animals. Despite these advances, these studies have primarily focused on the geometrical and mechanical aspects of locomotion, with no attention paid to its functional aspects.\\

The purpose of the present work is to revisit this approach from a more mathematical perspective. In particular, tools from functional analysis are employed to formulate several locomotion problems within a control-theoretic framework.\\

The paper is organized as follows. We collect in Section~\ref{sec:definitions_and_notations} the main definitions and notations used throughout the paper. We provide in Section~\ref{Section cosserat beam} first a detailed presentation of the configuration space we are working in,
and explain its links with mechanics. We then give a unified definition of dynamical constraints
which covers all the main cases such as kinematic, holonomic and non-holonomic constraints. We illustrate this definition with several examples.
Section~\ref{section:Cosserat-Poincaré Equations} introduces the dynamical equations needed to describe the motion of snakes. We first state a weak form of the Cosserat-Poincar\'e equation with constrains and external
forces obtained as the application of a generalized Hamiltonian principle. In a second step, we rigorously
derive a strong form of the Cosserat-Poincar\'e equation (i.e., variation-free) in the case of distributed Lagragian.
We then specialize this equation to the case where the Lagrangian is left-invariant
and is a difference of kinetic energy and (quadratic) potential energy terms.
Subsequently, we consider the total energy associated with the solutions of left-invariant Cosserat-Poincar\'e equations and compute its times derivative. The last part of the
section is devoted to deriving Cosserat-Poincar\'e equations for several important examples, in the form of systems of partial differential equations with adapted set of unknowns (in particular in the presence of constraints, where the corresponding Lagrange
multiplier comes into the picture). In Section~\ref{sec:control}, we put an emphasis on the notion of actuation, namely classifying these forces according to their physical nature and their role as potential actuators
for locomotion. Then, using these actuation effects, we will derive some controls systems for the robot locomotion of Cosserat-Poincar\'e types.
Finally, in Appendix, we gather on one hand several technical results needed in the main text
(compatibility conditions, inverse form of the structure equation, complete proof of the main theorem \ref{least action thm}, etc.) and on the other hand we provide detailed
computations showing that the Cosserat framework covers the cases of the rigid body and the leader-follower model.

In this work we have presented the Cosserat framework for the applied mathematics and control communities,
putting an emphasis on the derivation of many interesting systems of partial differential equations
representing motions of snakes with or without actuation. Future work must address $(a)$ existence and regularity of of solutions of these nonlinear (semilinear) hyperbolic systems in arbitrary times; $(b)$
asymptotic behavior in case of (let say) dissipative or friction forces; $(c)$ in case of actuation, control issues such as (distributed, boundary) controllability and observability.

\section{Definitions and notations}\label{sec:definitions_and_notations}

Let us introduce notations used throughout the paper.
Let $k\in\N_0$.
For vector spaces (resp. normed spaces) $E,F$,
we write $\mathscr{L}(E;F)$ (resp. $\mathscr{L}_c(E;F)$)
for the set of linear (resp. continuous linear) mappings
from $E$ to $F$.
For normed spaces $E$ and $F$,
write $E^*$ for the (continuous) dual space of $E$
and $L^*:F^*\to E^*$ for the adjoint of $L\in\mathscr{L}_c(E;F)$.



In this paper, we will only consider
real matrix (i.e. linear) Lie groups for simplicity. Adjusting any result presented to general finite dimensional Lie groups should be straightforward.
For a matrix Lie group $G$ acting on the left on a vector space $\mathbb{R}^k$,
the action of an element $g\in G$ on $x\in\R^k$
is denoted $g\cdot x$, or just $gx$
when $G$ consists of real $k\times k$-matrices.
The set of all real $k\times k$-matrices is denoted $M_{\R}(k)$.

The Lie algebra $\mathfrak{g}=\Lie(G)$ of $G$
is equipped with the associated Lie bracket $[\cdot,\cdot]$,
which in our case is just the commutator $[\xi,\eta]=\xi\eta-\eta\xi$ (for $\xi,\eta\in\mathfrak{g}$),
and $\mathfrak{g}$ is also assumed to be equipped with some inner product $\langle\cdot\,,\cdot\rangle$
of choice.
Using the latter, we often identify $\mathfrak{g}$ with its dual
$\mathfrak{g}^*$
and the adjoint $L^*$ of $L\in\mathscr{L}(\mathfrak{g};\mathfrak{g})$ with the transpose $L^T\in\mathscr{L}(\mathfrak{g};\mathfrak{g})$ of $L$
without any further mention.

For any $\xi\in \mathfrak{g}$,
one writes $\ad_{\xi}:\mathfrak{g}\to\mathfrak{g}$
for the linear map defined by
$\ad_{\xi}(\eta):=[\xi,\eta]$ for $\eta\in\mathfrak{g}$.
The adjoint of $\ad_{\xi}$ is a linear map $(\ad_{\xi})^*:\mathfrak{g}^*\to\mathfrak{g}^*$
which, in this text, we shall write as $\ad_{\xi}^*$ for the ease of notation.

The reader should pay attention to the selected sign convention associated with the notation $\ad_{\xi}^*$ in this paper:

\begin{remark}
{\bf Sign convention used for $\ad_{\xi}^*$:}
Usually in the literature (see \cite{kirillov2008introduction}) $\ad_{\xi}^*$ denotes the \emph{coadjoint} representation which is $-(\ad_{\xi})^*$
so our convention here $\ad_{\xi}^*:=(\ad_{\xi})^*$ is the negative of the coadjoint representation.
\end{remark}

For a smooth manifold (with or without boundary) $M$
and a normed space $E$,
we denote by $\mathscr{C}^k(M;E)$ the space of
all $k$-times continuously differentiable functions $M\to E$.
In particular, we only consider derivatives in the sense of Fr\'echet differentiability.

For a matrix Lie group $G\subset M_{\R}(k)$ and the
associated Lie algebra $\mathfrak{g}$, we write for short
\[
C^{k}_G:=\big\{f\in \mathscr{C}^k([0,1];M_{\R}(k))\ \big|\ f([0,1])\subset G\big\}
\quad\textrm{and}\quad
C^{k}_{\mathfrak{g}}:=\mathscr{C}^k([0,1];\mathfrak{g}).
\]
Similarly, for $T>0$ and $l\in\N_0$,
we denote
\[
C^{k,l}_{G}=\mathscr{C}^k([0,T];C^l_G):=\big\{h:[0,T]\to \mathscr{C}^k([0,1];M_{\R}(k))\ \big|\ h([0,T])\subset C^l_G\big\}
\quad\textrm{and}\quad
C^{k,l}_{\mathfrak{g}}:=\mathscr{C}^k([0,T];C^l_{\mathfrak{g}}).
\]


Some particular matrix Lie groups $G$ and their Lie algebras $\mathfrak{g}$
that we are mainly referring to in this paper are the following:
{\bf (a)} $\GL(k)$ is the group of all invertible real $k\times k$-matrices. Its Lie algebra is $\mathfrak{gl}(k)=M_{\R}(k)$, the set of all real $k\times k$-matrices.
All matrix Lie groups are subgroups of $\GL(k)$
and their Lie algebras are Lie subalgebras of $\gl(k)$.
{\bf (b)} $\SO(k)$ is the group of all real orthogonal $k\times k$-matrices with determinant one, i.e. $g\in\GL(k)$ s.t. $g^{-1}=g^T$ and $\det(g)=1$.
Its Lie algebra $\so(k)$ is the set of all real skew-symmetric $k\times k$-matrices, i.e. $\xi\in\mathfrak{gl}(k)$ s.t. $\xi^T=-\xi$.
{\bf (c)} $\SE(k)$ is (the component of identity of) the isometry group of $\R^k$
and it can be identified (isomorphically)
with the matrix group
\begin{equation}
\label{def:SE(k)}
\SE(k)=\left\{g=\begin{pmatrix}
    R & p \\
    0 & 1 \\
\end{pmatrix}\ \middle|\ R\in \SO(k),\ p\in \mathbb{R}^k\right\}
\end{equation}
of $(k+1)\times (k+1)$-matrices.
This matrix group acts naturally on $\R^k$ by
\begin{equation}
g\cdot x:=Rx + p, \quad g\in \SE(k),\ x\in \mathbb{R}^k.
\end{equation}
The associated Lie algebra is
\begin{equation}
\label{def:se(k)}
\mathfrak{se}(k)=\left\{M = \begin{pmatrix}
 A & v \\
 0 & 0
\end{pmatrix} \ \middle| \ A\in \mathfrak{so}(k),\ v\in \mathbb{R}^k \right\}.
\end{equation}

Finally, here are some notations that are relevant
for $\so(k)$ and $\se(k)$ when $k=3$ or $k=2$.
First is the hat operator
\begin{equation}\begin{aligned}\label{eq:hat_operator:1}
\widehat{ }:\mathbb{R}^3&\xrightarrow{} \mathfrak{so}(3)\\
\begin{pmatrix}
x_1\\x_2\\x_3
\end{pmatrix}& \mapsto \begin{pmatrix}
0 & -x_3 &x_2\\
x_3 & 0 & -x_1\\
-x_2 & x_1 & 0
\end{pmatrix}
\end{aligned}
\end{equation}
which yields a linear isomorphism from $\R^3$ onto $\so(3)$
such that
\begin{equation}
   X\times Y = \widehat{X}Y
   \quad \forall X,Y\in \mathbb{R}^3,
\end{equation}
where $\times$ is the cross product.
Analogously, the skew-symmetric matrix
$\mathrm{Q}=\begin{pmatrix}
0 & -1 \\ 1 & 0
\end{pmatrix}\in\so(2)$
yields a linear isomorphism
$\R\to\mathfrak{so}(2)$ by
$\omega\mapsto \omega \mathrm{Q}$.

We may then introduce the maps
   \begin{equation}\label{def : Lie algebra isomorphism}\begin{aligned}
    \mc{I}_3:(\mathbb{R}^6,[\cdot,\cdot]_{\mathbb{R}^6}) &\rightarrow (\mathfrak{se}(3),[\cdot,\cdot]), &  \mc{I}_2:(\mathbb{R}^3,[\cdot,\cdot]_{\mathbb{R}^3}) &\rightarrow (\mathfrak{se}(2),[\cdot,\cdot])\\
   \begin{pmatrix}
       X_1 \\ X_2
   \end{pmatrix} &\mapsto \begin{pmatrix}
       \widehat{X_1} & X_2 \\ 0 & 0
   \end{pmatrix} & \begin{pmatrix}
       \omega \\ X
   \end{pmatrix} &\mapsto \begin{pmatrix}
       \omega\mathrm{Q} & X \\ 0 & 0
   \end{pmatrix}\end{aligned}\end{equation}
which are Lie algebra isomorphisms when the respective
left-hand side spaces are equipped with the Lie brackets
   \begin{equation}
     [\zeta_1,\zeta_2]_{\mathbb{R}^6}=\left[\begin{pmatrix}
         X_1 \\ X_2
     \end{pmatrix},\begin{pmatrix}
         Y_1 \\ Y_2 \end{pmatrix}\right]_{\mathbb{R}^6}:=\begin{pmatrix}
         X_1 \times Y_1 \\ X_2 \times Y_1 + X_1 \times Y_2
     \end{pmatrix}, \quad \forall \zeta_1,  \zeta_2 \in \mathbb{R}^6,
 \end{equation}
and
\begin{equation}
 [\zeta_1,\zeta_2]_{\mathbb{R}^3}=\left[\begin{pmatrix}
     w_1 \\ X_1
 \end{pmatrix},\begin{pmatrix}
     w_2 \\ X_2 \end{pmatrix}\right]_{\mathbb{R}^3}:=\begin{pmatrix}
     0 \\ \omega_1 \mathrm{Q} X_2 - \omega_2 \mathrm{Q} X_1
 \end{pmatrix}, \quad \forall \zeta_1,  \zeta_2 \in \mathbb{R}^3.
 \end{equation}
At last, using these isomorphisms,
we define
   \begin{equation}
    P_3 : = \begin{pmatrix}
        0_{3\times 3 } & I_3
    \end{pmatrix},
    \quad
    P_2 = \begin{pmatrix}
        0_{2\times 1 } & I_2
    \end{pmatrix}
   \end{equation}
to be the projection matrices onto the linear (i.e. translational)
parts of $\mathfrak{se}(3)\simeq \mathbb{R}^6$
and $\mathfrak{se}(2)\simeq \mathbb{R}^3$,  respectively.
In other words, if one writes the elements of $\se(k)$ as $(x,A)\in\R^k\times\so(k)$
then $(P_k\circ\mc{I}_k)(x,A)=x$ for $k=2,3$.

\section{Configuration space and constraints}
\label{Section cosserat beam}

In this section, we define the configuration (or state) space we will be working in, which is a functional space taking values in a matrix Lie group $G$ acting on $\mathbb{R}^k$ where $k$ is a positive integer. We will then consider curves in the configuration space and define the pullback in the Lie algebra of the time and spatial derivatives of these curves. The latter verify
a partial differential equation referred to as the {\it structure equation} which reflects the eventual noncommutativity of $G$ (or equivalently the non-Euclidean nature of $G$). Finally, we will present a general framework for constraints acting on elements of the configuration space (or/and curves in that space).

\subsection{Configuration space}\label{ss-section-config}

In this paper, the configuration space is defined as the function space $C^2_G$, whose elements will be referred to as either \emph{states or configurations} (sometimes called Cosserat beams or rods, \cite{Cosserat1909}). For a fixed time horizon $T>0$, the elements of $C^{2,2}_{G}$ are referred to as \emph{movements} or \emph{motions}. We call \emph{pose at time $t\in [0,T]$} of a movement, the value in $C^{2}_G$ of that movement at $t\in [0,T]$. The space
$C^{2,2}_{G}$
will be our natural candidate for describing motions of snakes
over the time interval $[0,T]$, as explained in the sequel.

\begin{remark}\label{rem-statespace}
Instead of the continuous framework, one could have considered the Hilbertian one, namely replacing $\mathscr{C}^k$'s above by the Sobolev spaces $H^k$'s.
Even though the later choice has some advantages (see below when we will consider several dual spaces and dual operators), we stick to our choice because of two reasons: simplicity of the presentation (partial differential equations defined everywhere instead of a.e.) and keeping track of the dual operations to understand the meaning of each equation better.
\end{remark}

\begin{remark}\label{rem:imbrication-Versus-2Var}
The above spaces have been introduced because
we adopt here a time evolution point of view for describing the motion of snakes, meaning that we aim later at characterizing
this motion over a time interval $[0,T]$ as elements of $C^{2,2}_G$. It is important to notice that $C^{2,2}_G$ is (continuously) included
in $X^2_G:=\mathscr{C}^2([0,T]\times [0,1],G)$ and the inclusion is strict. In Appendix~\ref{sec:app:C^k(C^l)}, we exactly explain
that strict inclusion, namely we prove that $C^{2,2}_G$ is in one-to-one correspondence with the Banach subspace
$D^{2,2}_G$ of
$X^2_G$ made of the functions admitting the six continuous partial derivatives of order four  with two indices of derivations equal to the time (and hence the two others equal to the section $x$). With no further notice, we will often identify $C^{2,2}_G$ with $D^{2,2}_G$.
For an element of $X^2_G$, its value at time $(t,x)\in [0,T]\times [0,1]$ is called its  \emph{pose  at time $t$ and section $x$}.
\end{remark}

\begin{remark}\label{rem:mechanics}
We next provide a point of view for configurations and movements (i.e., elements of $C^{2}_G$ and $C^{2,2}_G$ respectively) issued from mechanics. Although many of the results that follow hold for any Lie group, the mechanical objets to which they apply are Cosserat rods \cite{Cosserat1909}, for which $G$ is equal to $SE(k)$, the Lie group of affine isometries of $\mathbb{R}^k$, cf. Definition~\ref{def:SE(k)}, which is also equal to the semi-direct product $\R^{k}\rtimes \SO(k)$. Starting from the general framework of continuous media mechanics \cite{MarsdenHughes}, we define two spaces. The first is the ambient Euclidian space $\mathcal{E}\cong \R^k$ equipped with a fixed orthonormal frame $\mathcal{F}_s=(o,e_1,\dots,e_k)$ named the ``spatial frame''. The coordinates of points of $\mathcal{E}$ in $\mathcal{F}_s=(o,e_1,\dots,e_k)$, or spatial coordinates, are the usual ``Eulerian variables''. The second space is the continuous medium itself, or ``body'', defined as a compact set $\mathcal{B}\subset\R^k$ of material points or ``particles''. These particles are labeled by their ``Lagrangian coordinates'' $(X^1,X^2,\dots,X^k)$ in a frame $\mathcal{F}_m=(O,E_1,E_2,\dots,E_k)$ attached to $\mathcal{B}$ that we name the material frame. In practice, these two spaces are (non canonically) identified by taking $\mathcal{F}_m=\mathcal{F}_s$ \cite{MerodioRosato2012}. Any configuration of $\mathcal{B}$ is defined by an invertible smooth map $\varphi(\cdot): X\in \mathcal{B}\mapsto \varphi(X)\in \mathcal{E}$ that preserves orientation, named a ``placement'', where $\varphi(X)$ are the coordinates of the position of the $X$-particle in $\mathcal{F}_s$ \cite{Truesdell1991}. In the Cosserat rod model \cite{Antman76}, \cite{Simo1988}, $\mathcal{B}$ is modeled as a continuous set of rigid cross-sections staked along one material dimension, say along $E_1$, that we identify with the material line of cross-section centroids, or ``material centroidal line'', and we have $\mathcal{B}=\cup_{X^1\in[0,1]}\mathcal{S}_{X^1}$, where $\mathcal{S}_{X^1}$ is the material cross-section of label $X^1\in [0,1]$. With this choice, any particle $X\in \mathcal{B}$, is labeled by its position $X=X^1 E_1+ X_\perp$, with $X_\perp=X^2E_2+X^3E_3+\dots+X^kE_k$ the material coordinates of the particle in the $X^1$-cross-section. The cross-sections being rigid, the position $\varphi(X)\in \mathcal{E}$ of any $X\in\mathcal{B}$, is defined in the spatial frame by $\varphi(X)=p(X^1)+R(X^1)X_\perp$, where $p(X^1)\in \R^k$ is the position in $\mathcal{F}_s$ of the centroid of the $X^1$-cross-section to which $X$ belongs, and $R(X^1)\in \SO(k)$ is the orientation in $\mathcal{F}_s$ of a spatial basis $(b_1,b_2,\dots,b_k)(X^1)=(R(X^1)E_1,R(X^1)E_2,\dots,R(X^1)E_k)$ materially attached to the cross-section and named the ``cross-sectional basis''. This parameterization shows that any rod configuration can be identified with a $X^1$-
parameterized curve in $\SE(k)$ \cite{BoyerPrimault},\cite{boyer2017poincare} here represented by homogeneous matrices (\ref{def:SE(k)}) or ``poses'', $g\in \SE(k)$, of cross-sectional frames $(p,b_1,b_2,\dots,b_k)(X^1)$ in $\mathcal{F}_s$, see Figure~\ref{fig:config}. Therefore, the rod
configuration space is naturally defined as the functional space $\mathcal{C}=\{g(\cdot): X^1\in[0,1]\mapsto g(X^1)\in \SE(k)\}$.
As regards movements, i.e. elements of $C^{2,2}_G$, we consider time-dependent placements of the form $\varphi(\cdot,\cdot): (X^1,t)\in [0,1]\times \mathbb{R}^{+}\mapsto\varphi(X,t)=p(X^1,t)+R(X^1,t)X_\perp$, i.e.
movements define a space of $(X^1,t)$-parameterized surfaces in $\SE(k)$: $\mathcal{M}=\{(g(\cdot,\cdot): (X^1,t)\in[0,1]\times \R_+\mapsto g(X^1,t)\in \SE(k)\}$.
\end{remark}

\begin{remark}
\label{remark : g0 fred}
Among all configurations accessible to the rod, we select one of them which is stress-free and refer to it as the reference configuration, denoted by $g_0$. This reference configuration will provide the basis for defining the strain measures in the context of (hyper) elasticity we consider here. In all the following, any field (i.e., any function) evaluated along this particular configuration will be indexed (subscripted) by a $0$ and often referred to as the \emph{reference field}.
\end{remark}

\begin{remark}\label{rem:1DCosserat}
In the above definitions, it is possible to replace the material centroid line by a $q$-dimensional manifold $M^q$ of the material space $\mathcal{B}$ along the dimensions of which are stacked ``small'' rigid bodies named "micro-bodies" or ``micro-structures'' that generalize the cross-sections of a rod. These media are named $q$-dimensional Cosserat media \cite{Antman76}, \cite{boyer2017poincare}. Using this terminology, in the present paper we only deal with one dimensional Cosserat media or Cosserat rods (i.e., configurations).
\end{remark}

\begin{figure}[h]
\begin{minipage}{0.9\textwidth}
\centering
    \includegraphics[scale=0.3]{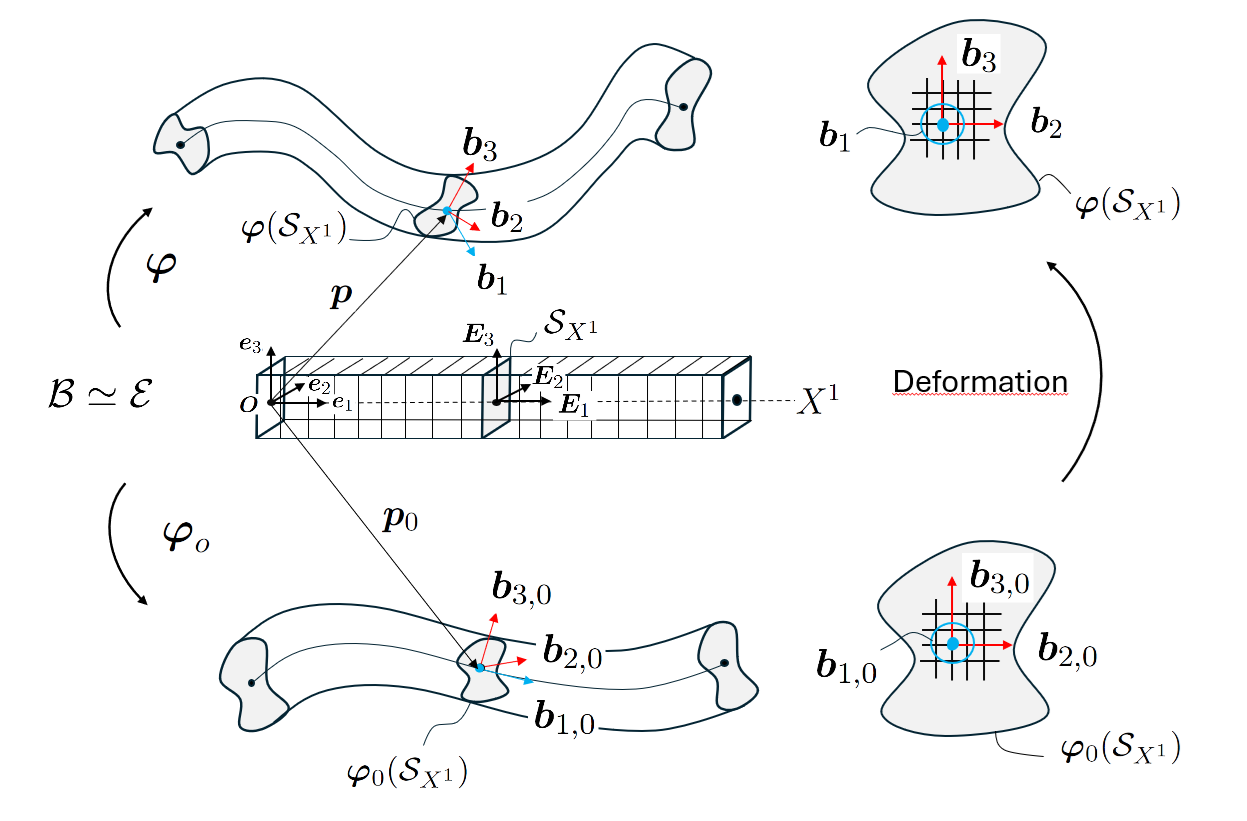}
   \caption{Representation of the material continuous medium $\mathcal{B}$ of a Cosserat rod with two configurations $\varphi$ and $\varphi_0$.}
   \label{fig:config}
\end{minipage}
\end{figure}

We next introduce time-twists and space-twists of a movement.

\begin{definition}[time twist and space twist]\label{def:twists:1}
For $g\in C^{2,2}_G$, we define its \emph{time twist} $\eta_g$ and \emph{space twist} $\xi_g$ as the elements of $C^{1,2}_{\mathfrak{g}}$ and $C^{2,1}_{\mathfrak{g}}$ given respectively by
\begin{equation}\label{eq:VT}
    \eta_g := g^{-1}\partial_tg
\end{equation}
and
\begin{equation}\label{eq:ST}
    \xi_g := g^{-1}\partial_xg.
\end{equation}
\end{definition}
\begin{remark}
    Studying the dynamic of a movement using $\eta_g$ and $\xi_g$ instead of $\partial_tg$ and $\partial_xg$ will be simpler as $\eta_g$ and $\xi_g$ take values in the vector space $\mathfrak{g}$ while $\partial_tg$ and $\partial_xg$ take values in the tangent bundle $TG$. To use the Lie algebra structure of $\mathfrak{g}$, we will often consider functions $f : TG \rightarrow \mathbb{R}$ that admit a left reduction, i.e. there exists a function $f_r : \mathfrak{g} \rightarrow \mathbb{R}$ such that $f(h,h_1)=f_r(h^{-1}h_1)$ for all $(h,h_1)\in TG$.
\end{remark}

The geometric structure of the Lie group $G$ appears in the following PDE satisfied by $\eta_g$ and $\xi_g$ and whose proof is immediate.
\begin{proposition}\label{prop:structure}
For $g\in C^{2,2}_G$, the corresponding
time twist $\eta_g\in  C^{1,2}_\mathfrak{g}$ and space twist twist $\xi_g  \in C^{2,1}_\mathfrak{g}$ satisfy
the \emph{structure equation} given by
\begin{equation}\label{eq:structure}
\partial_x\eta_g-\partial_t\xi_g=[\eta_g,\xi_g]
\quad \textrm{ on } [0,T]\times [0,1].
\end{equation}
\end{proposition}

\begin{remark}\label{re:reconstruction:1}
Conversely, given a function $\eta \in C^{1,2}_\mathfrak{g}$ and $\xi \in  C^{2,1}_\mathfrak{g},$ the existence of a movement $g\in C^{2,2}_G$ such that $\eta=\eta_g$ and $\xi=\xi_g$ is only possible if $(\eta,\xi)$ satisfy \eqref{eq:structure}. In the general setting with a (non-matrix) Lie group $G$, Proposition~\ref{prop:structure} can be seen as a consequence of the Maurer-Cartan structure equation and the converse result can be obtained using the fundamental theorem of calculus,
see \cite[Chapter 3, Theorem 6.1]{sharpe2000differential}
or \cite[Lemma 4.3]{rodriguez2020boundaryfeedbackstabilizationintrinsic}.
\end{remark}

\begin{remark}\label{rem:structure-minimal}
One can see that \eqref{eq:structure} only requires $g\in C^{1,1}_G$ (through the identification $C^{1,1}_G\simeq D^{1,1}_G$) to be defined.\end{remark}

\begin{remark}\label{ex:formese(k)} In the case $G=\SE(3)$, we use \eqref{def : Lie algebra isomorphism} to
express the twists \eqref{eq:VT}-\eqref{eq:ST} as vectors of $\mathbb{R}^6$. For $g\in C^{2,2}_{\SE(3)}$ we start from \eqref{def:SE(k)} to deduce that the position of the centerline $p$ belongs to
$\in C^{2,2}_{\mathbb{R}^3}$ and the rotation matrix representing the deformation of the cross-sections $R$ belongs to $\in C^{2,2}_{\SO(3)}$.
Then using \eqref{def : Lie algebra isomorphism} we can identify the twists as follows
\begin{equation}
\label{def:twistSE(k)}
    \eta_g \simeq \begin{pmatrix}
        \Omega \\ V
    \end{pmatrix}, \quad \xi_g \simeq \begin{pmatrix}
        K \\ \Gamma
    \end{pmatrix},
\end{equation}
where $\Omega$ and $K$ are called the \emph{angular space rate} and \emph{angular space rate} respectively and verify $\widehat{\Omega} = R^T\partial_t R$ and
$\widehat{K} = R^T\partial_x R$; $V=R^T\partial_t p$ and $\Gamma=R^T\partial_x p$ are the \emph{linear velocity} and the \emph{linear space rate} respectively. For the rest of the paper, the twists $\eta_g$ and $\xi_g$ will be considered by a slight abuse of notation as elements of either $\mathfrak{se}(3)$ or $\mathbb{R}^{6}$.

Similar representation of \eqref{def:twistSE(k)} holds for $\mathfrak{se}(2)$ using the isomorphism \eqref{def : Lie algebra isomorphism}, with $\Omega,K\in \mathbb{R}$ and $V,\Gamma\in \mathbb{R}^2$.
\end{remark}

\subsection{Modeling of constraints}

In this section, we introduce constraints to describe physical systems with restricted movements. We consider a framework encompassing both holonomic
and nonholonomic constraints, cf. \cite{bloch2025infinite}.

\begin{definition}\label{def:constraint}
A \emph{constraint} is a pair of functions $(A,a)$ with \begin{equation}
\label{def:constraint1}
    A : [0,T]\times C^2_G \rightarrow \mathscr{L}_c\big(C^2_\mathfrak{g};E\big),
    \quad
a :[0,T] \times C^2_G \rightarrow E,
\end{equation} where $(E,\lVert\cdot\rVert )$ is a Banach space called \emph{constraint space}, and $(A,a)$ satisfy
$a(t,g)\in \operatorname{Im}(A(t,g))$ for all $(t,h)\in [0,T]\times C^2_G$.
A movement $g\in C^{2,2}_G$ satisfies the constraint $(A,a)$ at time $t\in[0,T]$ if
    \begin{equation}
    \label{eq : constrainte}
        A\big(t, g(t,\cdot)\big)\eta_g(t,\cdot) + a\big(t,g(t,\cdot)\big)=0,
    \end{equation}
and $g\in C^{2,2}_G$ satisfies the constraint $(A,a)$ if is satisfies it at every time $t\in[0,T]$.
\end{definition}

We next provide several classical types of constraints.
\begin{definition}
\begin{itemize}
\item[(i)] A constraint $(A,a)$ is  \emph{time independent} if it does not depend on time $t\in[0,T]$.

\item[(ii)] A constraint $(A,a)$ is \emph{left-invariant} if there exists  a reduced constraint
$(A_r,a_r) : [0,T]\times C^2_G\rightarrow \mathscr{L}_c\big(C^2_\mathfrak{g};E\big)\times E$
such that
    \begin{equation}\label{eq:constraint-left-invariant}
    \begin{aligned}
        &A(t,h)=A_r(t,\xi_h), \quad a(t,h)=a_r(t,\xi_h) \quad \forall (t,h) \in [0,T]\times C^2_G.
        \end{aligned}
    \end{equation}
    \item [(iii)] A constraint $(A,a)$ is \emph{holonomic} if it is integrable, i.e there exists a differentiable function $\Phi : [0,T] \times C^2_G \rightarrow E$ such that
    \begin{equation}\label{eq:holonomic}
        \frac{d}{dt}\Phi\big(t,g(t,\cdot)\big)=A\big(t,g(t,\cdot)\big)\eta_g(t,\cdot)+a\big(t,g(t,\cdot)\big)\quad
        \forall (t,g) \in [0,T]\times C^{2,2}_G.
    \end{equation}
    A constraint which is not holonomic is called \emph{nonholonomic}.
\end{itemize}
\end{definition}

\begin{remark}
\label{remark : holo or not holo fred}
Constraints of the form (\ref{eq : constrainte}) and ${\Phi\big(t,g(t,\cdot)\big)=0}$ are called time-dependent bilateral constraints of kinematic and geometric form, respectively (\cite{Bloch, BoyerPorezMauny2018, Murray2017}).
In the first form, the prevented degrees of freedom are
parameterized by configuration variables (positions), while in the second form, they are parameterized with velocities. While non-holonomic constraints can only be set in kinematic form, (\ref{eq:holonomic}) shows
that holonomic constraints can be set in both geometric and kinematic form. Replacing holonomic constraints originally defined in geometric form, by constraints in kinematic
form in a dynamic formulation is always possible under the condition that the initial configuration of the system satisfies the
original geometric constraints (so that the system will remain on the level set fixed by original geometric constraints)\footnote{
As we shall see later, the dynamic formulation then takes the form of a system of partial differential
equations deduced from one of the principles of dynamics (in our case a least action principle), supplemented with a set of algebraic equations realized by the constraints. Such a system is a differential-
algebraic system and shifting holonomic constraints from their geometric to their kinematic form corresponds to decreasing the index of this differential-algebraic system of one unit, where, roughly speaking, the index is the number of derivations of the equations of formulation, required by the removal of constraints.
}.
Based on these considerations, we will consider constraints, holonomic or not, in the general form (\ref{eq : constrainte}) in order to derive a unique dynamic formulation holding in any case.
\end{remark}

\begin{remark}
\label{remark : int ext fred}
Constraints of the form (\ref{eq : constrainte}) can also be classified as either external or internal. External constraints restrict the relative degrees of freedom between one or more cross-sections of a rod and the external world, typically other bodies or obstacles. A typical example is a rod attached at one of its ends to a rigid body through a revolute joint or a rod clamped in a rigid wall. Internal constraints, in contrast, restrict the relative degrees of freedom between cross-sections belonging to the same rod \cite{Antman1984},\cite{boyer2022extended}. Depending on the cross-sections involved, internal constraints may be further classified as local or non-local. Non-local constraints establish kinematic relationships between distinct, potentially distant, cross-sections of the rod, as in the case of a rod connected at its both ends in a loop. Local constraints, on the other hand, act between infinitesimally close cross-sections and therefore impose kinematic restrictions at the local level. In the former case (non-local) the constraints concern the time-twists $\eta_g$ and are captured by the part ``$A\big(t, g(t,\cdot)\big)\eta_g(t,\cdot)$'' of (\ref{eq : constrainte}), while in the second they concern the space-twists $\xi_g$ through the part ``$a\big(t,g(t,\cdot)\big)$'' of (\ref{eq : constrainte}).
\end{remark}


We next provide several examples of constraints.

\begin{exemple}[Reference trajectory]\label{reference trajectory}
Consider \( g_r \in C^{2,2}_G \) referred to as a reference trajectory.
We say that a movement \( g \in C^{2,2}_G \) follows the reference trajectory $g_r$ on $[0,T]$ if it is equal to it, i.e.,
\begin{equation}
\label{eq : reference trajectory 1}
    g(t,x) - g_r(t,x)=0, \quad \forall(t,x)\in [0,T]\times [0,1].
\end{equation}
By time-differentiating \eqref{eq : reference trajectory 1} and left-multiplying by \( g^{-1}=g_r^{-1} \), we get
\begin{equation}
    \label{eq:prescribed constraint}
      \eta_g(t,x)-\eta_r(t,x)=0,
      \quad \forall(t,x)\in [0,T]\times [0,1],
\end{equation}
where $\eta_r(t,x)=g_r^{-1}(t,x)\partial_tg_r(t,x)$.
Hence if $g$ follows the reference trajectory $g_r$, it also satisfies the so-called reference trajectory
constraint $(A,a)$ defined by
\begin{gather}
\label{eq:reference_trajectory}
     A(t,h)\eta := \eta, \quad
     a(t,h) := -\eta_r
\end{gather}
for all $(t,h,\eta)\in [0,T]\times C^{2}_G \times C^2_{\mathfrak{g}}$ and with the constraint space $(C^2_{\mathfrak{g}},\lVert \cdot \rVert_{C^2})$.
Remark that, for all $(t,h)\in[0,T]\times C_G^{2}$, the constraint
satisfies $a(t,g)\in \operatorname{Im}A(t,h)$ as $A(t,h)$ is surjective.
\end{exemple}

\begin{exemple}[Reference shape and rigid body constraint]\label{Rigid body constraint}
Consider a reference \emph{shape} $\xi_r\in C^{2,1}_{G}$. We say that the motion $g\in C^{2,2}_G$ obeys the reference shape constraint if
\begin{equation}
    \xi_g(t,x) - \xi_r(t,x)=0 \quad \forall(t,x)\in[0,T]\times[0,1]
\end{equation}
where $\xi_g$ is the space twist~\eqref{eq:ST} associated with $g$.
Differentiating with respect to time $t$ and using the structure equation~\eqref{eq:structure},
we deduce that $g$ if follows the reference shape $\xi_r$, it satisfies the so-called reference shape constraint $(A,a)$ defined by     \begin{equation}\label{eq:reference shape}
    A(t,h)\eta :=   \operatorname{ad}_{\xi_h}\eta +\partial_x\eta , \quad
    a(t,h) := -\partial_t\xi_r,
\end{equation}
 for all $(t,h,\eta)\in [0,T]\times C^{2}_G \times C^2_{\mathfrak{g}}$.

 For $k\in \mathbb{N}$, $G=\SE(k)$ and for a reference shape $\xi_r$ that is constant in time (i.e. $\partial_t \xi_r=0$), a movement $g$ subject to this constraint is \emph{rigid} in the sense that the relative position of the particles of the associated beam, do not depend on time. We prove this assertion in Appendix~\ref{lemma :Equivalent rigid body}.
 In this case, the constraint $(A,a)$ is also called a rigid body constraint:
\begin{equation}\label{eq:rigid_body_constraint_time}
    A(t,h)\eta =\operatorname{ad}_{\xi_h} \eta+\partial_x\eta, \quad a(t,h) = 0
\end{equation}
for $(t,h,\eta)\in [0,T]\times C^{2}_G \times C^2_{\mathfrak{g}}$. Note that, for both the reference shape constraint~\eqref{eq:reference shape} and the rigid body constraint~\eqref{eq:rigid_body_constraint_time}, the constraint space is $(C^1_{\mathfrak{se}(k)},\lVert \cdot \rVert_{C^1})$ and $a(t,h)\in \operatorname{Im}A(t,h)$ as $A(t,h)$ is surjective for all $(t,h)\in[0,T]\times C^2_G$.

\end{exemple}

For the sequel, we introduce a specific configuration $g_0 \in C^2_G$, which can represent the pose of a movement at rest. We refer to it as the pose at rest.

\begin{definition}\label{Cosserat beam at rest}
   Let $\xi_0\in C^1_\mathfrak{g}$ be the reference space-twist field associated to the reference configuration $g_0\in C^2_G$ defined in Remark~\ref{remark : g0 fred}
    \begin{equation}\label{eq:def:ST0:a}
        \xi_0(x) := g_0^{-1}(x)\,\partial_x g_0(x).
    \end{equation}
\end{definition}

\begin{remark}
    For $G=\SE(3)$ or $\SE(2)$, as explained in Remark~\ref{ex:formese(k)}, the space twist at rest is represented as the vector
    \begin{equation}
    \label{def:ST0}
        \xi_0=\begin{pmatrix}
            K_0 \\ \Gamma_0
        \end{pmatrix},
    \end{equation}
where $K_0$ and $\Gamma_0$ are called the reference rotational and translational rate respectively.
\end{remark}

\begin{exemple}[Kirchhoff beam]\label{ex:kirchhoff}
A moving beam represented by a movement $g\in C^{2,2}_{\SE(3)}$ is said to be a Kirchhoff beam if its stretching component $\Gamma$ is equal to the stretching component at rest $\Gamma_0$, defined respectively in \eqref{def:twistSE(k)} and \eqref{def:ST0}, i.e.,
 \begin{equation}
     \Gamma(t,x)=\Gamma_0(x), \quad \forall(t,x)\in[0,T]\times[0,1].
 \end{equation}
Introducing the projection onto the stretching component $P_3:=\begin{pmatrix}
    0_{3\times3} & I_3
\end{pmatrix}$,
we define the function \begin{equation}\label{Exemple constrainte 1}
    \Phi(\xi): = P_3(\xi - \xi_0)
\end{equation}
which vanishes when $g\in C^{2,2}_{\SE(3)}$ is a Kirchhoff beam.
Differentiating in time the equation $
    \Phi(\xi_g) = 0,
$ and using the structure equation \eqref{eq:structure}, we obtain that any Kirchhoff beam $g\in C^{2,2}_{\SE(3)}$ satisfies the Kirchhoff constraint $(A,a)$ defined by
\begin{equation}
\label{eq:Kirchhoff constraint}
     A(t,h)\eta
: =  P_3\operatorname{ad}_{\xi_h} \eta+P_3\partial_x\eta,
\quad a(t,h) := 0,
\end{equation}
for all $(t,h,\eta) \in [0,T] \times C^2_{\SE(3)} \times C^2_{\mathfrak{se}(3)}$ and with the constraint space $(C^1_{\mathbb{R}^3},\lVert \cdot \rVert_{C^1})$. This constraint is time invariant, left invariant and holonomic as it derives from the function $\Phi(t,g(t,\cdot))=P_3\xi_g(t,\cdot)$. 

\end{exemple}
\begin{exemple}[Follow the leader in 2D]\label{Follow-the-leader}
Elongated animals such as snakes exhibit characteristic modes of locomotion on the ground.
One of them consists in lateral undulation, in which the body moves along its backbone.
This characteristic movement is referred to as \emph{follow-the-leader motion}, cf. \cite{boyer2011macrocontinuous}. Let us model these elongated animals as movements  \( g \in C^{2,2}_{\SE(2)} \) with a backbone's position \( p\in C^{2,2}_{\mathbb{R}^2} \). To move along its backbone, the snake needs to have no lateral velocity: the position of the snake $p\in C^{2,2}_{\SE(2)}$ satisfies that $\partial_t p(t,x)$ is collinear with $\partial_x p(t,x)$ for all $(t,x)\in[0,T]\times[0,1]$.
In particular, as proved Lemma \ref{lemma:Follow-the-leader} in Appendix \ref{Computation Follow the leader}, the functions $\partial_tp$ and $\partial_x p$ are collinear if and only if they satisfy $(P_2\xi_g)^T \mathrm{Q} P_2\eta_g=0$,
where the matrices $\mathrm{Q}$ and $P_2$ were introduced in Section \ref{sec:definitions_and_notations} and the twists $\eta_g$, $\xi_g$ are defined in \eqref{def:twistSE(k)}.
Hence, the movement of a snake $g$ obeys a follow-the-leader motion if and only if it satisfies the constraint $(A,a)$ with
\begin{equation}
\label{eq:Follow-The-Leader Constraint}
   A(t,h)\eta
: =(P_2\xi_h)^T \mathrm{Q} P_2\eta,
\quad a(t,h) := 0
\end{equation}
for all $(t,h,\eta)\in [0,T]\times C^2_{\SE(2)} \times C^2_{\mathfrak{se}(2)}$ and with the constraint space $(C^1_\mathbb{R},\lVert\cdot\rVert_{C^1})$.
\end{exemple}

\begin{exemple}[Embedded joint]\label{Embedding joint} A movement $g\in C^{2,2}_G$ may be embedded in a joint that forces the movement to be in the prescribed form $\widetilde{g}\in C^2_G$ at a section $x_*\in (0,1)$. In this case, the movement $g$ must satisfy the equation
\begin{equation}
    g(t,x_*)=\widetilde{g}(t),\quad \forall t\in[0,T].
\end{equation}
By time differentiating the above equation and multiplying by $g(t,x_*)^{-1}=\widetilde{g}(t)^{-1}$, we obtain the equation
\begin{equation}
    g^{-1}\partial_t g(t,x_*)-\widetilde{g}^{-1}\partial_t \widetilde{g}(t)=0 \quad \forall t\in[0,T].
\end{equation}
Thus, a movement $g$ with a prescribed form $\widetilde{g}$ at section $x_*$ satisfies the embedding joint constraint $(A,a)$ defined by
\begin{equation}\label{eq:embedding joint}
   A(t,h)\eta : =\eta(x_*), \quad  a(t,h):= \widetilde{g}(t)^{-1}\partial_t\widetilde{g}(t)
\end{equation}
for all $(t,h,\eta)\in [0,T]\times C^{2}_G \times C^2_\mathfrak{g}, $ and with the constraint space $(\mathfrak{g},\lVert\cdot\rVert)$. Note that, for all $(t,h)\in[0,T]\times C_G^{2}$, the constraint
satisfy $a(t,g)\in \operatorname{Im}A(t,h)$ as $A(t,h)$ is surjective.
One can also describe the above constraint
as a punctual constraint, at $x=x_*$. It can easily be extended to the case of a punctual constraint at several points.
\end{exemple}

\section{Cosserat-Poincaré Equations}\label{section:Cosserat-Poincaré Equations}

In this section, we derive Cosserat-Poincar\'e equations which are supposed to be verified by the motion of flexible beams, cf. \cite{BoyerPrimault,boyer2010poincare,boyer2017poincare}.
These equations are solutions of a variational problem and, in order to define it, we first introduce the concepts of
Lagrangian and forces. Then, by adding assumptions on the Lagrangian and the forces in
Section \ref{left-invariant Lagrangian section}, we reformulate the variational problem into a strong form and obtain the Cosserat-Poincar\'e equations as a system of
nonlinear hyperbolic PDEs.
In Section \ref{Section Energy estimate}, we study the  mechanical energy associated with the variational problem. We finally give in the Section \ref{Illustrative examples} a list of examples using different forces and constraints. In each of these examples, we aim at characterizing the appropriate unknowns functions defining the dynamics together with the corresponding system of PDEs verified by these unknowns.

\subsection{Variational approach}\label{section:variational}

Throughout this paper, a \emph{Lagrangian} $L$ is a function of the form \begin{equation}\label{def:Lagrangian}L\in \mathscr{C}^0([0,T]\times C^2_G\times C^2_\mathfrak{g};\mathbb{R})\end{equation}
that induces an \textbf{action} $S$ defined by
    \begin{equation}
    \label{def:action}
        \begin{aligned}
            S : C^{2,2}_G& \rightarrow \mathbb{R};
             \quad
             g \mapsto \int_0^T L\big(t,g(t,\cdot),\eta_g(t,\cdot) \big)dt.
        \end{aligned}
    \end{equation}
Last, we call \textbf{forces} functions of the form
\begin{equation}
\label{def:force} F : [0,T]\times C^2_G \times C^2_\mathfrak{g} \rightarrow \big(C^1_\mathfrak{g}\big)^*,\end{equation}
where $\big(C^1_\mathfrak{g}\big)^*$ stands for the dual of the Banach space $(C^1_\mathfrak{g})$
(see Section \ref{sec:dualCk} in Appendix).

\begin{definition}
 A Lagrangian $L$ (resp. a force $F$) is \textit{left-invariant} if
there exists a reduced Lagrangian $L_r:[0,T]\times C^1_\mathfrak{g}\times C^2_\mathfrak{g}\rightarrow \mathbb{R}$ (resp. reduced force $F_r:[0,T]\times C^1_\mathfrak{g}\times C^2_\mathfrak{g}\rightarrow (C^1_\mathfrak{g})'$) such that
\begin{equation}
    L(t,h,\eta)=L_r(t,\xi_h,\eta) \quad \bigl(\text{resp. }
    F(t,h,\eta)=F_r(t,\xi_h,\eta)\bigr),
\end{equation}
for all $(t,h,\eta) \in [0,T]\times C^2_G \times C^2_\mathfrak{g}$.
\end{definition}

\begin{definition}
\begin{itemize}
 \item[(i)] A force $F$ is \textit{distributed} if there exists $\mathcal{F}:[0,T]\times[0,1]\times G\times\mathfrak{g}^2\rightarrow\mathfrak{g}$ such that
    \begin{equation}\label{eq: viscous force}
        \langle F(t,h,\eta),v \rangle
        = \int_0^1 \langle\mathcal{F}(t,x,h(x),\xi_h(x),\eta(x)),v(x) \rangle dx,
        \quad \forall v\in C^1_\mathfrak{g},
        \ \forall (t,h,\eta)\in[0,T]\times C^2_G\times C^2_\mathfrak{g}.
    \end{equation}
    \item[(ii)] A force $F$ is \textit{punctual} if there exists  $x_0\in[0,1]$ and a function $F_p : [0,T]\times G \times \mathfrak{g}^2 \rightarrow \mathfrak{g}$ such that
     \begin{equation}\label{eq:punctual force}
        \langle F(t,h,\eta),v \rangle
        = \langle F_p(t,h(x_0),\xi_h(x_0),\eta(x_0)),v(x_0)\rangle, \quad \forall v\in C^1_\mathfrak{g}\ \forall (t,h,\eta)\in[0,T]\times C^2_G\times C^2_\mathfrak{g}.
    \end{equation}
\end{itemize}
\end{definition}

\begin{exemple}
A force produced by an actuator at a boundary point $x_b\in\{0,1\}$ of the configuration can be modeled as a punctual force $F$ for which there exists $F_b : [0,T]\times C^2_G \rightarrow \mathfrak{g}$ such that
 \begin{equation}
        \langle F(t,h,\eta),v \rangle
        := \langle F_b(t,h(x_b)),v(x_b)\rangle \quad \forall v\in C^1_\mathfrak{g}\ (t,h,\eta)\in[0,T]\times C^2_G\times C^2_\mathfrak{g}.
    \end{equation}
\end{exemple}

\begin{exemple}
The \emph{viscous reaction force} $F$ can be modeled in fluids with low Reynolds numbers as a force depending linearly on the time  twist $\eta$ for which there exists $\mu\in \mathscr{L}(\mathfrak{g},\mathfrak{g})$ such that
\begin{equation}
\label{def:viscous force}
    \langle F(t,h,\eta),v \rangle
        := \int_0^1 \langle - \mu \eta(x),v(x) \rangle dx
        \quad \forall v\in C^1_\mathfrak{g}\ \forall
        (t,h,\eta)\in[0,T]\times C^2_G\times C^2_\mathfrak{g},
\end{equation}
and $\mu$ satisfies $\langle \mu \eta, \eta \rangle \geq 0$ for all $\eta\in \mathfrak{g}$. The force \eqref{def:viscous force} is time independent, left-invariant and distributed.
\end{exemple}

To characterize the dynamics of a $g\in C^{2,2}_G$, we rely on the generalized Hamilton principle (cf. \cite{meirovitch2010methods}) which is of variational nature and goes as follows in the case of this paper.

\begin{definition}\label{def:Weak form}
Let us consider a Lagrangian $L\in \mathscr{C}^0([0,T]\times C^2_G\times C^2_\mathfrak{g};\mathbb{R})$ with the induced action $S: C^{2,2}_G\rightarrow \mathbb{R}$, a force $F:[0,T]\times C^2_G \times C^2_\mathfrak{g} \rightarrow (C^1_\mathfrak{g})^*$ and a constraint $(A,a):[0,T]\times C^2_G \to \mathscr{L}_c\big(C^1_{\mathfrak{g}};E\big)\times E$ with a constraint space $E$. A movement $g\in C^{2,2}_G$
is a said to  be \emph{solution to the weak form of Cosserat-Poincaré equation} if it satisfies
\begin{equation}\label{eq:Hamilton equation}
\left\{
\begin{aligned}
&\mathcal{D}S(g)(v)
=-\int_0^T \langle F\big(t,g(t,\cdot),\eta_g(t,\cdot)\big),v(t,\cdot)\rangle\, dt+  \int_0^T\langle A\big(t,g(t,\cdot)\big)^*\lambda(t,\cdot)),v(t,\cdot)\rangle\, dt
\quad & &\hspace{-3mm}\forall v\in C^1_0([0,T];C^1_\mathfrak{g}) \\
&A\big(t,g(t,\cdot)\big)\eta_g(t,\cdot)+a\big(t,g(t,\cdot)\big)=0 & &\hspace{-3mm}\forall t\in[0,T],
\end{aligned}
\right.
\end{equation}
where
\begin{itemize}
\item $\lambda \in C^1_{E^*}$ is the Lagrange multiplier induced by the constraint $(A,a)$,
\item $C^1_0([0,T],C^1_{\mathfrak{g}})=\{v\in C^{1,1}_{\mathfrak{g}}\ |\ v(0,\cdot)=v(T,\cdot)=0\}$,
\item $\mathcal{D}S(g)(v):=\frac{d}{dh}S(\gamma(h))\big|_{h=0}$,
where $\gamma$ is a $C^1$-curve
in $C^{2,2}_{G}$
such that $\gamma(0)=g$ and $\gamma'(0)=gv$.
\end{itemize}
\end{definition}

\begin{remark} Equation~\eqref{eq:Hamilton equation} can be seen as a extension of the Lagrange-d'Alembert principle. Indeed, it is shown in \cite{bloch2025infinite} that a movement $g$ submitted to the constraint $(A,a)$ must satisfy \begin{equation}
    \mathcal{D}S(g)(v)
=-\int_0^T \langle F\big(t,g(t,\cdot),\eta_g(t,\cdot)\big),v(t,\cdot)\rangle\, dt,
\end{equation} for all $v\in C^1_0([0,T];C^1_\mathfrak{g})$ satisfying $A(t,h) v(t,\cdot) =0$. Then the existence of Lagrange multiplier $\lambda$ follows if $\operatorname{Im} A(t,g)$ is weak-* closed for all $(t,g)\in [0,T] \times C^2_G$ as it is the case for the constraints \eqref{eq:reference_trajectory}, \eqref{eq:rigid_body_constraint_time}, \eqref{eq:Kirchhoff constraint}, \eqref{eq:embedding joint}, see \cite{bloch2025infinite} and \cite{jarchow2012locally} for more  details.
\end{remark}

\begin{remark}
\label{remark : least action hol not hol fred}
In this paper, we aim to produce a general framework that can be applied to both holonomic and non-holonomic constraints. When constraints are holonomic, the configuration space of the constrained rod is a submanifold of its unconstrained configuration space, named admissible submanifold, where by ``admissible'' we mean ``compatible with constraints''. In this case, one can define trajectories between any two arbitrary points of this space, and the least action principle that compare the value of the action functional along such trajectories, can be applied. Technically, one can applied an augmented Lagrangian approach by adding to the unconstrained Lagrangian, a ``constraint potential energy'', whose variation defines the virtual work of some reaction generalized forces ensuring the preservation of constraints. In the case, of non-holonomic constraints, there is no such submanifold, but only a field of admissible hyper-planes defined in each point of the unconstrained configuration manifold, i.e, an admissible non-integrable distribution (in the meaning of Frobenius theorem). Therefore, when applying least action principle, the trajectories tied between any two configurations are not defined globally, but locally by the fact that in each point of the trajectory, the velocity needs to belong to the kernel of kinematic constraints (see Eq. \eqref{eq : constrainte}). In this geometric context, the constraint forces that preserve non-holonomic constraints do not derive from a potential energy globally defined on the candidate trajectories, as in the holonomic case, but can only be defined locally in each point as belonging to the dual of admissible velocities. In short, while the constraint forces associated to holonomic constraints are conservative, those associated to non-holonomic constraints are not, and need to be entered in a ``virtual work contribution''. This can be done by replacing the (integral) least action principle by (differential) D'Alembert's principle of virtual works, or alternatively, by invoking a least action principle extended to non-conservative external forces as that proposed in \cite{meirovitch2010methods}. We here adopt this later approach, which allows to keep benefits of least action principle while addressing both holonomic and non-holonomic constraints in an unique framework.
\end{remark}

\begin{remark}
\label{rem:reaction force}
The covector $\lambda$ represent reaction forces (and couples) exerted in the ambient Euclidean space (the physical space) $\mathcal{E}\simeq \R^{k}$, generally expressed in some local frames adapted to the contact.
Therefore they are generalized functions taking values in the dual space $\mathcal{E}^{\star}$. Based on this first remark, $\lambda$ has a concrete (physical) meaning. In contrast, $A(t,g)^*\lambda$ represents the pullback of the same reaction forces to the dual space of the configuration space here and is what we name the generalized reaction forces, i.e., the forces imposed
by the constraints in $\mathcal{E}$, ``seen from our configuration space''. In our case, these generalized forces are wrenches (\cite{Murray2017}), i.e., duals of twists, named ``reaction wrenches''. \end{remark}

We next derive a strong form of the Cosserat-Poincar\'e equation (i.e., getting rid of  variations $v$ appearing in \eqref{eq:Hamilton equation}), under the following three assumptions on the
Lagrangian, the forces and the constraints.

\begin{Assumption}\label{Ass : lagrangian}
    The Lagrangian $L$
    is \textit{distributed}, i.e.,
    \begin{equation}
        L\big(t,h,\eta\big)=\int_0^1 \mathcal{L}(t,x,h(x),\xi_h(x),\eta(x))\,dx \quad \forall (t,x,h,\eta) \in [0,T]\times[0,1]\times C^2_G \times C^2_\mathfrak{g},
    \end{equation}
    where $\mathcal{L}\in \mathscr{C}^1([0,T]\times      [0,1]\times G \times \mathfrak{g}^2;\mathbb{R})$ is called the {\it Lagrangian density}.
\end{Assumption}

\begin{Assumption}\label{Ass : force}
    The forces $F : [0,T]\times C^2_G \times C^2_\mathfrak{g} \rightarrow \big(C^1_\mathfrak{g}\big)^* $ has the following form
    \begin{equation}\begin{aligned}\label{eq:Ass : force:1}
    \langle F\big(t,h,\eta\big),v\rangle = &\int_0^T \langle \mathcal{F}\big(t,x,h(x),\xi_h(x),\xi'_h(x),\eta(x)\big),v(x)\rangle dx  \\
    &+\langle F_0\big(t,h(0),\xi_h(0),\eta(0)\big),v(0)\rangle+\langle F_1\big(t,h(1),\xi_h(1),\eta(1)\big),v(1)\rangle\end{aligned}\end{equation}
    for all $(t,h,\eta)\in [0,T]\times C^2_G \times C^2_\mathfrak{g}$, $v\in C^1_\mathfrak{g}$, with  \begin{equation}
    \mathcal{F}\in \mathscr{C}^0([0,T]\times[0,1]\times G\times \mathfrak{g}^3;\mathfrak{g}), \quad  F_0,F_1\in \mathscr{C}^2([0,T]\times G\times \mathfrak{g}^2;\mathfrak{g})\end{equation}
    called \emph{distributed part} and \emph{punctual parts} of the force $F$ respectively.
\end{Assumption}

\begin{Assumption}\label{Ass : constraint}
The constraint space $E$ is equal to the Banach space $(C^k_{\mathbb{R}^l},\lVert \cdot \rVert_{C^k})$ for some integer $k\geq 1$. Moreover, the constraint $(A,a):[0,T]\times C^2_\mathfrak{g} \rightarrow \mathscr{L}_c\big(C^2_\mathfrak{g};C^k_{\mathbb{R}^l}\big)\times C^k_{\mathbb{R}^l}$ has the following form
    \begin{equation}\label{eq:Ass : constraint:1}
        A(t,h)\eta
        := \mathcal{C}^1(t,h,\xi_h)\eta + \mathcal{C}^2(t,h,\xi_h)\eta', \quad  a(t,h)   
=\widetilde{a}\big(t,h,\xi_h\big)
    \end{equation}
 for $(t,h,\eta)\in [0,T]\times C^2_G \times C^2_\mathfrak{g}$, $l\in \mathbb{N}$, $k\in \{1,2\}$ and with \begin{equation}\label{def : a,C1,C2} \mathcal{C}^1,\mathcal{C}^2\in \mathscr{C}^1([0,T]\times G \times \mathfrak{g}; \mathscr{L}(\mathfrak{g},\mathbb{R}^l)), \quad \widetilde{a}\in \mathscr{C}^1([0,T]\times G \times \mathfrak{g};\mathbb{R}^l).\end{equation}
Here the constraint $(A,a)$ is said to be \emph{distributed} with density constraint function $\mathcal{C}\in \mathscr{C}^1([0,T]\times G\times \mathfrak{g}^3; \mathbb{R}^l)$ defined as
 \begin{equation}\label{eq:density-constraint}
     \mathcal{C}(t,h,\xi_h,\eta,\eta'):=\widetilde{a}\big(t,h,\xi_h\big)+\mathcal{C}^1(t,h,\xi_h)\eta + \mathcal{C}^2(t,h,\xi_h)\eta'.
 \end{equation}
\end{Assumption}

\begin{Assumption}\label{ass:injecting_lambda:1}
The constraint space will depend on the regularity of $\mathcal{C}^1$ and $\mathcal{C}^2$.
On the other hand, we will need some regularity for the Lagrange multiplier $\lambda(t,\cdot)$ as a function of $x\in [0,1]$ for every $t\in [0,T]$. For that purpose, we consider the smooth injection $i :C^1_{\mathbb{R}^l}\ni f \mapsto \int_0^1 \langle f(x), \cdot \rangle\, dx\in (C^1_{\mathbb{R}^l})^*$. We then simply assume from now on  that $\lambda \in \mathscr{C}^{1,1}_{\mathbb{R}^l}$ and identify $\lambda(t,\cdot)$ with $i(\lambda(t,\cdot))$ for every $t\in [0,T]$.
\end{Assumption}

Under the above assumptions, we obtain the strong form of the Cosserat Poincaré equation presented below.
We use $\mathcal{L}_i$'s to denote the partial derivatives (differentials) of the Lagrangian density $\mathcal{L}$ with respect to the $i$-th variables.

 \begin{theorem}
 \label{least action thm} Consider a Lagrangian $L$, a force $F$ and a constraint $(A,a)$ subject to Assumptions \ref{Ass : lagrangian}, \ref{Ass : force} and \ref{Ass : constraint}.
 Then a solution $g\in C^{2,2}_G$ of the weak form of Cosserat-Poincaré equation \eqref{eq:Hamilton equation}
 solves the \emph{strong form of the Cosserat-Poincaré equation
 which is the partial differential equation}
    \begin{equation}
        \begin{aligned} \label{equation Cosserat dyn 1}
       \frac{d}{dt}\mathcal{L}_5+\frac{d}{dx}\mathcal{L}_4
       -g^*\mathcal{L}_3 - \operatorname{ad}_{\eta_{g}}^*\mathcal{L}_5-\operatorname{ad}_{\xi_{g}}^*\mathcal{L}_4-\mathcal{F}   +\mathcal{C}^{1*}\lambda - \partial_x\big(\mathcal{C}^{2*}\lambda\big)=0 \quad \text{ in } (0,T)\times(0,1)
    \end{aligned} \end{equation}
    with the constraint

    \begin{equation}
        \label{Constraint Cosserat Poincaré}
         \widetilde{a} + \mathcal{C}^1\eta_g+ \mathcal{C}^2\partial_x \eta_g=0 \quad \text{ in } (0,T)\times(0,1)
    \end{equation}
    and the boundary conditions for all $t\in[0,T]$ given by
    \begin{equation}\begin{aligned} \label{equation Cosserat boun 1}\mathcal{L}_3(t,1,g(t,1),\xi_g(t,1),\eta_g(t,1))&=-F_1(t,g(t,1),\eta_g(t,1),\xi_g(t,1))+\mathcal{C}^2(t,g(t,1),\xi_g(t,L))^*\lambda(t,1),  \\
    \mathcal{L}_3(t,0,g(t,0),\xi_g(t,0),\eta_g(t,0))&=F_0(t,g(t,0),\eta_g(t,0),\xi_g(t,0))+\mathcal{C}^2(t,g(t,0),\xi_g(t,0))^*\lambda(t,0).\\
    \end{aligned}\end{equation}

\end{theorem}
\begin{proof}The proof of the theorem is deferred to Appendix~\ref{app:least action thm}.
\end{proof}
\begin{remark}
The boundary conditions \eqref{equation Cosserat boun 1} depend on the choice of the function space for the variation $v$.
For instance, if instead of the space $C^1_0([0,T];C^1_\mathfrak{g})$
we had used e.g. $C_0^1([0,T];C^1_0([0,1];\mathfrak{g}))$
then the boundary conditions \eqref{equation Cosserat boun 1} are not required.
Conversely, considering a function space larger than $C^1([0,T];C^1_0([0,1];\mathfrak{g}))$ imposes additional equations into the Cosserat-Poincaré equation. For instance, considering the  function space $C^{1,1}_\mathfrak{g}$ adds to \eqref{equation Cosserat boun 1} the new equations
    \begin{equation}
        \mathcal{L}_4(0,x,g(0,x),\xi_{g}(0,x),\eta_{g}(0,x))=0, \quad \mathcal{L}_4(T,x,g(T,x),\xi_{g}(T,x),\eta_{g}(T,x))=0 \quad \forall x\in[0,1].
    \end{equation}
 \end{remark}
 \begin{remark}
If the constraint and the force contain punctual forces and constraints at interior points $x_1,x_2...,x_N\in\,(0,1)$, the equations \eqref{equation Cosserat dyn 1}--\eqref{equation Cosserat boun 1} will have to be written on the intervals $(x_i,x_{i+1})$ taking $x_0=0$ and $x_{N+1}=1$.
In this case, the punctual forces and constraints together with the continuity of $g$ at each points $x_i$ will yield the boundary conditions for the partial differential equation (PDE) on the interval $(x_i,x_{i+1})$.
\end{remark}

\begin{remark}[Constraint free equation]
The case $l=0$ corresponds the case where there are no constraints.
The Cosserat-Poincaré equation thus simplifies to the \emph{unconstrained Cosserat-Poincaré equation}
which is \eqref{equation Cosserat dyn 1} with $\mc{C}^1=\mc{C}^2=\lambda=0$
along with the boundary conditions \eqref{equation Cosserat boun 1}.
\end{remark}

\subsection{Left-invariant Lagrangian}
\label{left-invariant Lagrangian section}

In order to provide an explicit form to the Lagrangian density (see Assumption~\ref{Ass : lagrangian}), we introduce the notions of kinetic and potential energy densities. Given
$I : [0,1] \mapsto \mathscr{L}(\mathfrak{g},\mathfrak{g})$ where $I(x)$ is symmetric positive definite for all $x\in[0,1]$,
the \emph{$I$-kinetic energy density} is defined as
\begin{equation}\label{def:kinematic energy}
    \mathcal{T}_I : [0,1]\times \mathfrak{g} \rightarrow \mathbb{R}, \quad (x,\eta) \mapsto \frac{1}{2}\langle I(x)\eta,\eta\rangle.
\end{equation}
Similarly, given $H : [0,1] \mapsto \mathscr{L}(\mathfrak{g},\mathfrak{g})$ where $H(x)$ is symmetric positive definite for all $x\in[0,1]$, and referred to as the Hookean reduced stiffness matrix, and $\xi_0$, the reference space-twis  defined in \eqref{eq:def:ST0:a}, the \emph{$(H,\xi_0)$-potential energy density $\mathcal{U}_{H,\xi_0}(\xi)$} is defined as
\begin{equation}\label{def:potential energy}
    \mathcal{U}_{H,\xi_0} : [0,1]\times \mathfrak{g} \rightarrow \mathbb{R}, \quad (x,\xi)\mapsto \frac{1}{2}\langle  H(x)(\xi-\xi_0(x)),\xi-\xi_0(x)\rangle.
\end{equation}
For later use, we call \emph{stress} the function
\begin{equation}\label{def:stress}
    \Lambda: = H(\xi - \xi_0).
\end{equation}

\begin{remark}
\label{remark : Loi constitutive fred}
The above constitutive law is the most commonly used in practical applications. It assumes that the material is elastic and subject to small space twists. However, it can be said hyper-elastic since it is compatible with finite transformations of $SE(k)$ and especially those induced by the large deformations of the rod. More generally, in (\ref{def:potential energy}) one can replace the internal quadratic energy density by more general nonlinear functions of space twists $\mathcal{U}_{H,\xi_0}(\xi)$, and define the stress field along the rod, as the strain energy conjugate:
\begin{equation}\label{def:stress conjugate}
    \Lambda: = \frac{\partial \mathcal{U}_{H,\xi_0}}{\partial \xi}(\xi).
\end{equation}
\\
Although this does not call into question all the subsequent developments, the choice of such a potential is not an easy task and requires reconsidering the rod as a three-dimensional medium. Much more often, the above Hookean law is enriched in a heuristic way by adding to the restoring elastic terms, some Kelvin-Voigt dissipation terms \cite{Damping2013}. In this case, (\ref{def:stress}) is changed into:
\begin{equation}\label{def:stress + dissipation}
    \Lambda: = H(\xi - \xi_0)+D\dot{\xi}.
\end{equation}
where $D$ is a dissipation constant matrix often considered in the form $D=\mu H$, with $\mu$ an empirical damping factor.
\end{remark}

Last, let us fix the distributed Lagrangian
\begin{equation}\label{def: Cas particulier Lagrangien }
    L(t,h,\eta):=\int_0^1 \mathcal{L}(t,x,h(x),\xi_h(x),\eta(x)) dx, \quad \forall (t,h,\eta)\in[0,1]\times C^2_G \times C^2_\mathfrak{g},
\end{equation}
where \begin{equation}\label{def: Cas particulier Lagrangien density}
     \mathcal{L}(t,x,h,\xi,\eta):=\mathcal{T}_I(\eta)-\mathcal{U}_{H,\xi_0}(\xi), \quad \forall(t,h,\xi,\eta)\in [0,T]\times G \times \mathfrak{g}^2.
\end{equation}

Note that the Lagrangian \eqref{def: Cas particulier Lagrangien density} is time independent and left-invariant\footnote{The left-invariance property is a consequence of the symmetries of the ambient space (homogeneity and isotropy). In more details, it results of the fact that once expressed in the cross-sectional frames, neither kinetic energy, nor potential energy depend on the position and orientation of an external observer attached to the ambient space (the spatial frame). In solid mechanics, this symmetry property of space is often referred to as "material objectivity" and roughly reflects the fact that matter is indifferent to space.}. Let us now apply Theorem \ref{least action thm} with the Lagrangian \eqref{def: Cas particulier Lagrangien  }.

\begin{corollary}
\label{Equation left-invariant}
    Consider a force $F$ satisfying Assumption \ref{Ass : force}, a constraint $(A,a)$ satisfying Assumption \ref{Ass : constraint} and the Lagrangian as in~\eqref{def: Cas particulier Lagrangien }. Denote by $\mathcal{F}$ and $(F_0,F_1)$ respectively the distributed and punctual part of the force and recall the functions $\widetilde{a},\mathcal{C}^1$ and $\mathcal{C}^2$ induced by the constraint $(A,a)$ defined in \eqref{def : a,C1,C2}.
If a movement $g\in C^{2,2}_G$ and the Lagrange multipliers $\lambda\in C^{1,1}_{\mathbb{R}^l} $ satisfy the equation \eqref{eq:Hamilton equation} with the initial conditions $g^0\in C^2_G$ and $(g^0)^{-1}g^1\in C^2_{\mathfrak{g}}$ for all $x\in[0,1]$, then $(g,\lambda)$ satisfies the following partial differential system for all $(t,x)\in [0,T]\times[0,1]$
\begin{subequations}\label{Cosserat Lagrangian 0.5}
    \begin{empheq}[left=\empheqlbrace]{align}
    &\partial_t(I\eta)-\partial_x\Lambda-\operatorname{ad}_{\eta}^*(I\eta)+\operatorname{ad}_{\xi}^*\Lambda+\mathcal{C}^1(t,g,\xi)^*\lambda - \partial_x(\mathcal{C}^2(t,g,\xi)^*\lambda)=\mathcal{F}(t,x,g,\xi,\partial_x\xi,\eta),\label{Cosserat Lagrangian 0.5 1}\\
    &\widetilde{a}(t,g,\xi)+\mathcal{C}^1(t,g,\xi)\eta + \mathcal{C}^2(t,g,\xi)\partial_x\eta = 0,\label{Cosserat Lagrangian 0.5 2}\\
    &\partial_x g(t,x)=g(t,x)\xi(t,x), \quad \partial_t g(t,x)=g(t,x)\eta(t,x),\label{Cosserat Lagrangian 0.5 3}\\
    &\Lambda(t,0)=-F_0(t,g(t,0),\xi(t,0),\eta(t,0))-\mathcal{C}^2(t,g(t,0),\xi(t,0))^*\lambda(t,0),\label{Cosserat Lagrangian 0.5 4}\\
    &\Lambda(t,1)= F_1(t,g(t,1),\xi(t,1),\eta(t,1))-\mathcal{C}^2(t,g(t,1),\xi(t,1))^*\lambda(t,1),\label{Cosserat Lagrangian 0.5 5}\\
    & g(0,x)=g^0(x), \quad \partial_tg(0,x)=g^1(x) ,\label{Cosserat Lagrangian 0.5 6}
    \end{empheq}
    \end{subequations}
    where the stress $\Lambda$ is defined in \eqref{def:stress}.
\end{corollary}

\begin{proof}
    Considering the expression of the Lagrangian \eqref{def: Cas particulier Lagrangien density}, we have
\begin{align}
    \mathcal{L}_3(t,g(t,x),\xi_g(t,x),\eta_g(t,x))&=0,\\
    \mathcal{L}_4(t,g(t,x),\xi_g(t,x),\eta_g(t,x))&=-H(x)(\xi_g(t,x)-\xi_0(x))=-\Lambda(t,x),\\
    \mathcal{L}_5(t,g(t,x),\xi_g(t,x),\eta_g(t,x))&=I(x)\eta_g(t,x).
\end{align}
    By applying Theorem \ref{least action thm}, we obtain the desired set of equations \eqref{Cosserat Lagrangian 0.5}.
\end{proof}

We say that the Cosserat Poincaré equation is \textit{left-invariant} if the Lagrangian density, the force and the constraint are all left-invariant functions. In that case,
\eqref{Cosserat Lagrangian 0.5} can be simplified by taking the time and space  twists $\eta_g, \xi_g$ together with the Lagrange multiplier $\lambda$ as the unknowns of a reduced Cosserat-Poincaré equation as given next. In order to simplify the notation, we make no distinction between the forces $\mathcal{F},F_0,F_1$ (resp. constraints $\mathcal{C}^1,\mathcal{C}^2,\widetilde{a}$) and the reduced forces (resp. reduced constraints).
\begin{corollary}\label{cor:CPE-left-invariant}
With the hypotheses of Corollary \ref{Equation left-invariant}, also assume the force $F$ and the distributed constraint $(A,a)$ are left invariant. If the initial condition is given by
$\eta(0,\cdot) =\eta^0(\cdot) \in C^1_\mathfrak{g}$ and $\xi(0,\cdot) = \xi^0(\cdot) \in  C^1_\mathfrak{g}$
then $(\eta_g,\xi_g,\lambda)$ satisfies the following partial differential system for all $(t,x)\in[0,T]\times[0,1]$
\begin{subequations}
\label{Cosserat Lagrangian 1}
\begin{empheq}[left=\empheqlbrace]{align}
&\label{eq:Cosserat Lagrangian 1 dyn}\partial_t(I\eta) - \partial_x\Lambda - \operatorname{ad}_{\eta}^*(I\eta) + \operatorname{ad}_{\xi}^*\Lambda+\mathcal{C}^1(t,\xi)^*\lambda - \partial_x(\mathcal{C}^2(t,\xi)^*\lambda) = \mathcal{F}(t,x, \xi,\partial_x\xi, \eta), \\
&\label{eq:Cosserat Lagrangian 1 cons}\widetilde{a}(t,\xi)+\mathcal{C}^1(t,\xi)\eta + \mathcal{C}^2(t,\xi)\partial_x\eta =0, \\
&\label{eq:Cosserat Lagrangian 1 struc}\partial_x \eta - \partial_t \xi = [\eta, \xi], \\
&\label{eq:Cosserat Lagrangian 1 bound 1}\Lambda(t, 0) = -F_0(t, \xi(t, 0), \eta(t, 0))-\mathcal{C}^2(t,\xi(t,0))^*\lambda(t,0), \\
&\label{eq:Cosserat Lagrangian 1 bound 2}\Lambda(t, 1) = F_1(t, \xi(t, 1), \eta(t, 1))-\mathcal{C}^2(t,\xi(t,1))^*\lambda(t,1), \\
&\label{eq:Cosserat Lagrangian 1 ini cond}(\eta(0, x), \xi(0, x)) = (g^0(x)^{-1}g^1(x), g^0( x)^{-1}\partial_x g^0(x)),
\end{empheq}
\end{subequations}
where $\Lambda := H(\xi - \xi_0)$ is defined in equation \eqref{def:stress}. Conversely, a solution $(\eta,\xi,\lambda)$ of \eqref{Cosserat Lagrangian 1} defines, up to a multiplicative constant in $G$, a solution $g\in C^{2,2}_G$ of equation \eqref{Cosserat Lagrangian 0.5}.
\end{corollary}
\begin{proof}
Let $(g,\lambda)$ be the solution of \eqref{Cosserat Lagrangian 0.5}, then
the functions $\eta_g,\xi_g,\lambda$ satisfy, by definition, the equations \eqref{eq:Cosserat Lagrangian 1 dyn}-\eqref{eq:Cosserat Lagrangian 1 cons}-\eqref{eq:Cosserat Lagrangian 1 bound 1}-\eqref{eq:Cosserat Lagrangian 1 bound 2}-\eqref{eq:Cosserat Lagrangian 1 ini cond}. The remaining equation \eqref{eq:Cosserat Lagrangian 1 struc} to be satisfied is the structure equation \eqref{eq:structure} that is automatically verified for twists by Proposition \ref{prop:structure}.

Conversely, the movement $g$ arising from a solution $(g,\lambda)$ of equation~\eqref{Cosserat Lagrangian 0.5} can be reconstructed from the solution $(\eta,\xi,\lambda)$ of equation~\eqref{Cosserat Lagrangian 1} using Remark~\ref{re:reconstruction:1} together with a predetermined configuration $g^*\in G$ at $(0,x_0)$ with some $x_0\in [0,1]$.
\end{proof}

\subsection{Energy of solutions of Cosserat-Poincar\'e equations}
\label{Section Energy estimate}

One can associate with the solutions of Cosserat-Poincar\'e equations a function of the time $t$ called energy, as it is the case for solutions of hyperbolic partial differential equations \cite{bastin2016stability}. In this section, we define this energy function and compute its time derivative along solutions of Cosserat–Poincaré equations.

The \emph{(mechanical) energy $E$} of a movement $g\in C^{2,2}_G$ is defined by
     \begin{equation} \label{mechanical energy def}
     E(t)=\int_0^1 \mathcal{T}_I(\eta_g(t,x)) + \mathcal{U}_{H,\xi_0}(\xi_g(t,x))dx,\quad t\in [0,T].
     \end{equation}


In the next result, we determine the time derivative of $E$ under the assumptions of Corollary \ref{Equation left-invariant}.

\begin{proposition} \label{energy} Let $g\in C^{2,2}_G$ be a solution to the Cosserat-Poincaré equation \eqref{Cosserat Lagrangian 0.5}. Then the time derivative of mechanical energy $E$ of $g$ on $[0,T]$ is equal to
\begin{equation}\label{eq:der-E}\begin{aligned}
\frac{d}{dt} E(t) &= \int_0^1 \langle \mathcal{F}(t,x,g,\xi_g,\partial_x\xi_g,\eta_g), \eta_g \rangle +\langle \lambda , \widetilde{a}(t,g,\xi_g) \rangle,dx\\
&+\langle F_0(t,g(t,0),\xi_g(t,0),\eta_g(t,0)),\eta_g(t,0)\rangle+ \langle F_1(t,g(t,L),\xi_g(t,L),\eta_g(t,1)),\eta_g(t,1)\rangle\rangle.\end{aligned}\end{equation}
\end{proposition}
\begin{proof}In this paragraph, we write $\eta$ and $\xi$ instead of $\eta_g$ and $\xi_g$. Due to the $C^1$-regularity of  $\eta$ and $\xi$, we can interchange integration over $[0,1]$ and time derivative. One gets that
\[
\dif{t} E(t)=\int_0^1\dif{t}\mathcal{T}_I(\eta)\,dx + \int_0^1\dif{t}\mathcal{U}_{H,\xi_0}(\xi)\,dx.
\]
Using the structure equation \eqref{eq:structure}, it holds
   \begin{equation}\begin{aligned}
        \dif{t}\mathcal{U}_{H,\xi_0}(\xi)&=\langle \Lambda,\partial_t\xi \rangle=\langle\Lambda,\partial_x\eta-[\eta,\xi]\rangle.
        \end{aligned}
    \end{equation}
Then using \eqref{Cosserat Lagrangian 0.5}, one gets
\begin{equation}\begin{aligned}\label{cumputation energy}
        \int_0^1\dif{t}\mathcal{U}_{H,\xi_0}(\xi)\,dx&=\int_0^1 \langle \Lambda, \partial_x\eta-[\eta,\xi]\rangle\,
        dx
        =\left[\langle \Lambda,\eta \rangle \right]_0^1 - \int_0^1 \langle \Lambda, [\eta,\xi]\rangle  + \langle \partial_x\Lambda,\eta \rangle\,dx\\
        =& \left[\langle \Lambda, \eta \rangle \right]_0^1 - \int_0^1 \big(\langle  \Lambda,[\eta,\xi]\rangle + \langle I\partial_t\eta - ad^*_\eta I\eta + ad^*_\xi\Lambda - \mathcal{F}(t,g,\xi,\eta),\eta  \rangle\big) \,dx \\
        &-\int_0^1 \langle (\mathcal{C}^1)^*\lambda - \partial_x(\mathcal{C}^2)^*\lambda), \eta\rangle \,dx \\
        =&\left[\langle \Lambda,\eta \rangle \right]_0^1  - \int_0^1 \big(\langle I\partial_t\eta, \eta\rangle - \langle  I\eta,[\eta,\eta ] \rangle - \langle \mathcal{F}(t,g,\eta,\xi),\eta\rangle\big) \,dx\\
        &-\int_0^1 \big(\langle \lambda, \mathcal{C}^1\eta \rangle + \langle \lambda, \mathcal{C}^2\partial_x\eta \rangle\big)\,dx + \left[\langle (\mathcal{C}^2)^*(\lambda),\eta \rangle\right]_0^1  \\
        =&\langle F_0(t,g(t,0),\xi(t,0),\eta(t,0)),\eta(t,0)\rangle + \langle F_1(t,g(t,1),\xi(t,1),\eta(t,1)),\eta(t,1)\rangle  \\
        &+ \int_0^1 \langle \mathcal{F}(t,g,\eta,\xi),\eta\rangle \,dx - \int_0^1  \langle I\partial_t\eta, \eta \rangle \,dx+\int_0^1 \langle \lambda, \widetilde{a}(t,g,\xi) \rangle \,dx
    \end{aligned}\end{equation}
    which after adding the term
$\int_0^1\dif{t}\mathcal{T}_I(\eta)\,dx= \int_0^1 \langle I\partial_t\eta,\eta \rangle\,dx$
    yields \eqref{eq:der-E}.
\end{proof}

\begin{remark}
In the case of holonomic constraints, the term $\widetilde{a}(t,g,\xi)$ represents the time variation of the constraint. From a physical point of view, it is therefore expected that
both $\mathcal{F}$ and $\widetilde{a}$ appear in the formula for the time derivative of $E$, since they contribute as external energies.
\end{remark}

Proposition~\ref{energy} highlights a particular class of forces $F$, called dissipative forces,
that renders the mechanical energy into a
non-increasing function of the time.

\begin{definition}
We say that a distributed force $\mathcal{F}$ (a punctual part $(F_0,F_1)$ respectively) is {\it dissipative} if, for all $g\in C^{2,2}_G$ and
$(t,x)\in [0,T]\times[0,1]$, one has
\begin{equation}
        \label{dissipatif 1}\langle \mathcal{F}(t,x,g,\xi_g(t,x),\eta_g(t,x)),\eta_g(t,x)\rangle \leq 0
    \end{equation}
\begin{equation}
        \label{dissipatif 2}(\langle F_0(t,g,\xi_g(t,x),\eta_g(t,x)),\eta_g(t,x)\rangle \leq 0, \quad \langle F_1(t,g,\xi_g(t,x),\eta_g(t,x)),\eta_g(t,x)\rangle \leq 0 \textrm{ respectively}).
    \end{equation}
\end{definition}

A force $F$ is dissipative if $\mathcal{F}$ and $(F_0,F_1)$ are.
The next corollary is immediate.

\begin{corollary}\label{Cor : energy} Let
$g\in C^{2,2}_G$ be a solution of \eqref{Cosserat Lagrangian 0.5}
without constraints
i.e. $\mathcal{C}^1=\mathcal{C}^2=0$ and $\widetilde{a}=0$
in \eqref{eq:density-constraint}.
Then the following holds true.
    \begin{itemize}
    \item[(i)] If the force $F$ is dissipative, then for any solution $g$ of system \eqref{Cosserat Lagrangian 0.5}
    the associated mechanical energy is non-increasing.
    \item[(ii)] If there are no forces acting on $g$, i.e. $F=0$, then for any solution $g$ of the system \eqref{Cosserat Lagrangian 0.5} the associated mechanical energy mechanical is constant.
\end{itemize}
\end{corollary}

\begin{remark}\label{re:noether:1}
The previous result is not surprising as it corresponds to conservation (resp. dissipation) of the mechanical energy in the case of free motion (resp. under dissipative forces).
In particular, the conservation of energy (ii) can be recovered using Noether's theorem.
Nevertheless, Proposition \ref{energy} and its Corollary \ref{Cor : energy} generalize the energy
result established in \cite[Proposition 2.1]{rodriguez2020boundaryfeedbackstabilizationintrinsic}.

\end{remark}

\subsection{Illustrative examples}
\label{Illustrative examples}

In this section, we derive from Corollary \ref{Equation left-invariant} the equations of motion for
various examples of forces and constraints. The main objective of the examples is to express the Cosserat-Poincaré equations as systems of partial differential equations with appropriate unknowns,
especially in the presence of constraints, where attention is paid to reducing the number of unknowns, often by expressing some components of the time and space  twists $\eta,\xi$ in terms of the Lagrange multiplier $\lambda$.
Note that in this paper, we do not prove the existence and uniqueness of solutions for the PDE systems we obtain.


\subsubsection{Without constraints}

We first consider examples without constraints.
In this subsection, only the first example of Cosserat-Poincaré equation is left-invariant.
\begin{exemple}[Left-invariant free motion]
Let us consider the motion of beam in the case where there is no force and no constraint acting on it. This is a particular instance of Corollary~\ref{cor:CPE-left-invariant}, where one simplifies \eqref{Cosserat Lagrangian 1} with $\mathcal{F}=0$, $\mathcal{C}^1=\mathcal{C}^2=0$ and $\widetilde{a}=0$.
Given initial conditions $(\eta^0,\xi^0)\in (C^1_{\mathfrak{g}})^2$, the time and space  twists $\eta, \xi : [0,T] \times [0,1] \rightarrow \mathfrak{g}$ satisfy the following nonlinear hyperbolic system of partial differential equations:
\begin{subequations}
\label{Cosserat Lagrangian exemple 1}
\begin{empheq}[left=\empheqlbrace]{align}
&I\partial_t\eta - \partial_x\big(H(\xi - \xi_0)\big) - \operatorname{ad}_{\eta}^*(I\eta) + \operatorname{ad}_{\xi}^* H(\xi - \xi_0) = 0
\quad & \forall (t,x)\in[0,T]\times[0,1],\\
&\partial_x\eta - \partial_t\xi = [\eta,\xi]
\quad & \forall (t,x)\in[0,T]\times[0,1],\\
&\xi(t,0)=\xi_0(0), \quad \xi(t,1)=\xi_0(1)
\quad &\forall t\in[0,T],\\
&(\eta(0,x),\xi(0,x))=(\eta^0(x),\xi^0(x))
\quad &\forall x\in[0,1].
\end{empheq}
\end{subequations}

\end{exemple}

\begin{exemple}[With weight] Here we consider the case where
the Lie group $G$ is $\SE(3)$, the force acting on the beam
is the \emph{gravitational force (or weight)} and we have no additional constraints.
This force near the surface of the earth is well approximated by
\begin{equation}
\label{def:gravity force}
    \mathcal{F}(g,\xi,\partial_x\xi,\eta)=\begin{pmatrix}
        0 \\ \rho A_0 R^Ta_\mathbf{g}
    \end{pmatrix},\quad a_\mathbf{g}=-\begin{pmatrix}
        0& 0&\mathbf{g}
    \end{pmatrix}^T,
    \end{equation}
    where $\rho>0$ is the mass density of the beam, $A_0>0$ is the cross-sectional area of the beam, $\mathbf{g}\approx 9.81\, m/s^2$ the gravitational acceleration and $R$ is the rotational component of the movement $g\in C^{2,2}_G$.
The corresponding Cosserat-Poincar\'e equation
is given as a particular instance of Corollary~\ref{Equation left-invariant} with $\mathcal{F}$ defined in \eqref{def:gravity force}
and with $\mathcal{C}^1=\mathcal{C}^2=0$, $\widetilde{a}=0$.
Given initial conditions $(g^0,  (g^0)^{-1}g^1)$ in $C^2_{\SE(3)}\times C^2_{\mathfrak{se}(3)}$, the solution $g\in C^{2,2}_{\SE(3)}$ of the Cosserat-Poincar\'e equation on $[0,T]\times[0,1]$ verifies
\begin{equation}
\begin{aligned}
    I\partial_t(g^{-1}\partial_tg)-\partial_x\big(H((g^{-1}\partial_xg)-\xi_0)\big)-\operatorname{ad}_{g^{-1}\partial_tg}^*I(g^{-1}\partial_tg) +\operatorname{ad}_{(g^{-1}\partial_xg)}^*H((g^{-1}\partial_xg)-\xi_0)=\begin{pmatrix}
        0 \\ \rho A (R^Ta_\mathbf{g})
    \end{pmatrix}
\end{aligned}
\end{equation}
with the boundary conditions $\partial_xg(t,0)=g(t,0)\xi_0(0)$,
$\partial_xg(t,1)=g(t,0)\xi_0(1)$ for all $t\in [0,T]$
with $\xi_0(0),\xi_0(1)\in \mathfrak{se}(3)$ and the initial conditions $g(0,x)=g^0(x)$, $\partial_t g(0,x)=g^1(x)$ for all
$x \in [0,1]$.
\end{exemple}

\subsubsection{With constraints}
We consider in this section some examples of Cosserat beams subject to constraints.
All the cases considered are left-invariant.

\begin{exemple}[Rigid body constraint]\label{ex:Rigid body constraint:2}
We return to Example \ref{Rigid body constraint} and consider a rigid body beam, i.e., an elongated body that cannot deform over time.
It turns out that the Euler-Poincaré equation (see \cite{holm1998eulerpoincareequationssemidirectproducts})
\begin{equation}
\label{eq:dynamic Rigid body ODE}
    \begin{cases}
        \mathcal{M} \partial_t  \eta_b(t) - \operatorname{ad}^*_{\eta_b(t)}\mathcal{M} \eta_b(t) = 0, \quad \forall t\in [0,T]\\
        \eta_b(0)=\eta^0(0)
    \end{cases}
\end{equation}
The dynamics of such rigid bodies can be obtained from the Cosserat-Poincaré equation subject to the rigid body constraint introduced in Example \ref{Rigid body constraint}.
More precisely, to obtain \eqref{eq:dynamic Rigid body ODE}, one applies Corollary~\ref{cor:CPE-left-invariant} with $\mathcal{F}=0$, $\mathcal{C}^1=\operatorname{ad}_\xi,$ $\mathcal{C}^2=I_{\mathrm{d}}$ and $\widetilde{a}=0$,
and defines $\mathcal{M}:= \int_0^1 \operatorname{Ad}_{h(x)}^* I(x) \operatorname{Ad}_{h(x)} dx$,
$h(x):=g(0,x)^{-1}g(0,0)$
and $\eta_b(t):=\eta(t,0)$,
where $\eta\in C^{1,2}_{\mathfrak{g}}$ is the time  twist in \eqref{Cosserat Lagrangian 1}.
The complete derivation is given in  Appendix~\ref{sec:rigid-body}.

\end{exemple}

\begin{exemple}[Kirchhoff beam]\label{ex:Kirchhoff_beam:1}
Let us study the dynamic of a Kirchhoff beam, which is a beam having its linear space rate part (stretching component) at rest for all times, i.e. satisfying the constraint \eqref{eq:Kirchhoff constraint}. We model the beam as a movement $g\in C^{2,2}_{\SE(3)}$ and define the matrix  $P_3 = \begin{pmatrix}
   0_{3 \times 3} & I_3
\end{pmatrix}$ that projects the space twist \eqref{def:twistSE(k)} onto its linear part.
The Kirchhoff beam dynamics, without additional forces, is obtained from Corollary~\ref{cor:CPE-left-invariant} by setting $\mathcal{F}=0$, $\mathcal{C}^1=P_3\operatorname{ad}_\xi,$ $\mathcal{C}^2=P_3$ and $\widetilde{a}=0$ in \eqref{Cosserat Lagrangian 1}.
Under \eqref{eq:Kirchhoff constraint} and given initial conditions $\eta^0, \xi^0$ in $C^1_{\mathfrak{g}}$, this dynamics is governed by the following Cosserat-Poincaré equation
\begin{subequations}
    \label{Cosserat Lagrangian example 2}
    \begin{empheq}[left=\empheqlbrace]{align}
    &I \partial_t \eta - \partial_x \Lambda - \operatorname{ad}_{\eta}^T I \eta + \operatorname{ad}_{\xi}^T  \Lambda  + \operatorname{ad}_{\xi}^T P_3^T \lambda - \partial_x \big(P_3^T \lambda \big) = 0 \quad & \forall (t,x)\in[0,T]\times[0,1], \\
    &\label{constraint Lagrangian example 2} P_3 \partial_x \eta = P_3 [\eta, \xi] \quad & \forall (t,x)\in[0,T]\times[0,1], \\
    &\partial_x \eta - \partial_t \xi = [\eta, \xi] \quad & \forall (t,x)\in[0,T]\times[0,1],
    \label{constraint Lagrangian example 2:structure_eq} \\
    &\Lambda(t,0)= - P_3^T \lambda(t,0), \quad \Lambda(t,1)=- P_3^T \lambda(t,1) \quad &\forall t \in [0,T], \\
    &(\eta(0,x), \xi(0,x)) = (\eta^0(x), \xi^0(x)) \quad & \forall x \in [0,1],
    \end{empheq}
\end{subequations}
where the unknown functions are $\eta, \xi,\lambda$ and we have denoted $\Lambda=H(\xi-\xi_0)$ as usual (see \eqref{def:stress}).
Equation~\eqref{Cosserat Lagrangian example 2} can also be expressed in coordinates as shown in Appendix~\ref{Kirchhoff beam Appendix}.

\end{exemple}

\begin{exemple}[Follow the leader motion] \label{ex:followtheleader}
Snakes generate locomotion by undulating along their backbone.
As they move forward, their body follows the path traced by
their head. Modeling a snake as a planar movement $g\in C^{2,2}_{\SE(2)}$ with a position of the centerline $p\in C^{2,2}_{\mathbb{R}^2}$ and a deformation $R\in C^{2,2}_{\SO(2)}$,
it is classical (cf. \cite{6072271}) to model the snake's motion by the following transport equations
on $[0,T]\times [0,1]$:
\begin{equation}
    \begin{aligned}
    \begin{cases}
    \label{eq:follow the leader position}
        \partial_t p(t,x) - v(t) \partial_x p(t,x) =0,\\
     p(t,1) = p_h(t),  \quad p(0,x)=p^0(x),
    \end{cases}
    \end{aligned}
\end{equation}
and \begin{equation}
    \begin{aligned}
    \begin{cases}
     \label{eq:follow the leader deformation}
        \partial_t R(t,x) - v(t) \partial_x R(t,x) =0,\\
        R(t,1) = R_h(t),  \quad R(0,x)=R^0(x).
    \end{cases}
    \end{aligned}
\end{equation}
In equations~\eqref{eq:follow the leader position} and \eqref{eq:follow the leader deformation}, $(p^0,R^0)$ is the initial pose of the snake, $(p_h,R_h)$ is the pose of the head
and the axial velocity $v=\lVert p_h \rVert$ is assumed to be positive and governed on $[0,T]$ by
\begin{equation}
\begin{aligned}
\label{eq:vitesse controlled}
\begin{cases}
v'(t) = U(t),\\
v(0)=v^0,
\end{cases}
\end{aligned}
\end{equation}
where $v^0>0$ and $U : [0,T]\rightarrow \mathbb{R}$ is a control function. Given a positive axial velocity function $v$,
the deformation $(p,R)$ follows the trajectory traced by the pose of the head in the sense that, for every $(t,x)\in[0,T]\times[0,1]$
the value of $(p(t,x),R(t,x))$ can be expressed as $(p_h(s),R_h(s))$ for some $s\in[0,t]$.
Indeed, for a given $(t,x)\in [0,T]\times [0,1]$,
this $s=s(t,x)$ can be determined from the (common) characteristic direction of the above system of equations
and is given by $s=t-h$, where $h=h(t,x)\in [0,t]$
is the solution of $\int_{t-h}^t v(r)dr=1-x$ (assuming a solution exists).

In order to recover the previous equations from Cosserat-Poincaré equation, we show that the solution $g \in C^{2,2}_{\SE(2)}$ of Corollary~\ref{cor:CPE-left-invariant} subject to the follow-the-leader constraint \eqref{eq:Follow-The-Leader Constraint}, the Kirchhoff constraint \eqref{eq:Kirchhoff constraint}, and a boundary force satisfies equations~\eqref{eq:follow the leader position}-\eqref{eq:vitesse controlled}. In particular, the first component $V_1$ (in $\R^2$) of the linear velocity at the head $x=1$ coincides with the axial velocity $v$, and the functions $R_h$, $p_h$ and $U$ are boundary conditions at the head. The derivation proceeds in two steps with all details given in Appendix~\ref{Computation Follow the leader}.

\end{exemple}

\section{Modeling actuation and control}\label{sec:control}

In this section, we present the forces acting on a soft robot modeled as a movement $g\in C^{2,2}_{\SE(3)}$. In particular, we classify these forces according to their physical nature and their role as potential actuators
for locomotion. Then, using these actuation effects, we will derive some controls systems for the robot locomotion of Cosserat-Poincar\'e types.

We next distinguish two types of forces: \emph{the uncontrolled forces} which act on the robot but cannot be specified by an operator and \emph{the actuation forces} which can be prescribed by an operator of the robot.

\subsection{External effects}
\label{External effect}


In what follows, we present a list of relevant forces for locomotion problems, starting with the uncontrolled forces.
\begin{itemize}
 \item The effect of a current in the water, the wind or the gravitation can be modeled by a vector field $\vec{F}:\mathbb{R}^3 \rightarrow \mathbb{R}^3$. 
 The corresponding \emph{force exerted by
 $\vec{F}$} can be expressed as
    \begin{equation}
    \label{def:force current}
    \mathcal{F}(g)=\begin{pmatrix}
        0 \\ R^T\Vec{F}
    \end{pmatrix},
    \quad g=(p,R)\in\SE(3)
\end{equation}
 with $R\in\SO(3)$ being the deformation (rotation) of the cross-sections associated to $g$. In particular, we recover the expression for weight \eqref{def:gravity force} by taking $\vec{F}=\rho A_0 a_\mathbf{g}$ with
 $a_\mathbf{g}=-\begin{pmatrix}
        0& 0&\mathbf{g}
    \end{pmatrix}^T$, $\mathbf{g}$ being the gravitational acceleration,
 $\rho>0$ the mass density of the beam, $A_0>0$ the area of the cross-section. Note that (nonzero) forces exerted by uniform fields are not left-invariant as they depend on the rotation $R$.
\begin{remark}
A force $\mathcal{F}$ exerted by a constant vector field $\Vec{F}$ derives from the potential energy $E_{\Vec{F}}(g)= - \langle \Vec{F},p\rangle$, i.e. it satisfies
  $\mathcal{F}(g) = - g^*\nabla E_{\Vec{F}}(g)$,
where $p\in\R^3$ is the position (translation) associated to $g$, $\nabla$ is the gradient with respect to $g$ and $g^*\nabla E_{\Vec{F}}(g)$ is the element of $\mathfrak{se}(3)$ that satisfies $( g^*\nabla E_{\Vec{F}}(g))^Th = \operatorname{tr}(\nabla E_{\Vec{F}}(g)^T gh)$ for all $h\in \mathfrak{se}(3)$.
Note that $g^*\nabla E_{\Vec{F}}(g)\in \mathbb{R}^6$, where $g,\nabla E_{\Vec{F}},h\in\mathbb{R}^{4\times 4}$. This slight abuse of notation is explained in Remark~\ref{ex:formese(k)}.
Forces deriving from a potential energy are called conservative forces and play an important role in dynamics as they conserve the following energy function:
\begin{equation}\label{energy extented}
    E(t)=\int_0^1 \mathcal{T}_I(\eta_g(t,x)) + \mathcal{U}_{H,\xi_0}(\xi_g(t,x)) + E_{\Vec{F}}(g(t,x))dx.
\end{equation} 
 \end{remark}


\item Due to the nature of the environment, the robot can often hit immovable obstacles on the ground. Assuming that the obstacle is a set $M \subset \mathbb{R}^3$ admitting an outward unit normal vector $n(x)$ at each point $x\in \partial M$ of the boundary, the \emph{obstacle force of reaction}, taken as a regularize repulsive potential~\cite{1087247},  is defined as
\begin{equation}
    F(g,\xi,\partial_x\xi,\eta) = \begin{pmatrix}

   0 \\ \frac{n(p)}{(\operatorname{dist}(p,M) + \varepsilon)^\gamma}  \end{pmatrix},
\end{equation}
where $p$ is the position associated to $g$, $\gamma>0$, $\varepsilon$ is a small positive constant and $\operatorname{dist}(p,M) = \operatorname{inf}\{\lVert p-x\rVert \mid x\in M\}$
is the distance of $p$ from $M$.
Note that the obstacle force can also be modeled as a reaction force of a unilateral constraint~\cite{10494907}.

\item The robot immersed in a fluid or moving on the ground is submitted to \emph{friction forces} defined as
    \begin{equation}
        \label{def : reaction force}
        \mathcal{F}(g,\xi,\partial_x\xi,\eta) = f(\eta),
    \end{equation}
where $f : \mathfrak{se}(3) \rightarrow \mathfrak{se}(3)$ satisfies $\langle f(\eta), \eta \rangle \leq 0$ for all $\eta \in \mathfrak{se}(3)$. In particular, \emph{viscous friction forces}~\cite{Cox_1970} and \emph{dry friction forces}~\cite{brogliato1996nonsmooth} can respectively be modeled as
\begin{equation}
f(\eta) = -\mu_1 \eta \quad\textrm{ and }\quad
f(\eta) = -\mu_2 \frac{ \eta}{ \lVert \eta \rVert },
\end{equation}
where $\mu_1$ and $\mu_2$ are diagonalizable matrices with positive diagonal entries.
\end{itemize}

\subsection{Actuation effects}
\label{actuation effect}

Let us now present some classical actuation effects.
\begin{itemize}
\item Assuming that a distribution of magnetic passive dipoles is attached to the robot, a prescribed magnetic field can be used to exert forces on it and induce locomotion. For a distribution of magnetic dipoles $M : [0,1] \rightarrow \mathbb{R}^3$ expressed in the cross-sectional frame and a spatially uniform magnetic field $b : [0,T] \rightarrow \mathbb{R}^3$, the \emph{magnetic force} is given as
\begin{equation}
    \mathcal{F}(g)=\begin{pmatrix}
        M \times R^T b \\ 0
    \end{pmatrix}
\end{equation}
where $R$ is the rotational part of $g\in C^{2,2}_{\SE(3)}$ (see \eqref{def:SE(k)}). Note that this force is purely rotational in contrast to the weight and current forces, which only have translational components.
    \item Snakes often limit the contact points with the ground.  It is therefore relevant to study locomotion with localized actuation forces at a finite number of contact points. Taking $N$ contact points $(x_i)_{1,..,N}\in [0,1]$, we approximate the \emph{forces of contact} as
    \begin{equation}
        \mathcal{F}_i(t)= \mathbbm{1}_{[x_i-\varepsilon,x_i+\varepsilon]}(x)f_i(t),
    \end{equation}
    where $f_i:[0,T] \rightarrow \mathfrak{se}(3)$ can be prescribed and $\varepsilon>0$ is a small positive constant such that $[x_i-\varepsilon,x_i+\varepsilon] \subset (0,1)$ for all $i\in \{ 1,...,N\}$.
    \item One can also consider \emph{boundary forces} at the extremities $x=0$ and $x=1$ of the robot, defined as
    \begin{equation}
    \label{def:boundary force}
        F_0(t) = f_0(t), \quad  F_1(t) = f_1(t),
    \end{equation}
    where $f_0,f_1:[0,T]\rightarrow \mathfrak{se}(3)$.


\begin{remark}[Internal force]\label{rem:internal force}
In equations~\eqref{Cosserat Lagrangian 0.5} and ~\eqref{Cosserat Lagrangian 1}, the stress $\Lambda(t,x,\xi(t,x))=H\big(\xi(t,x)-\xi_0(x)\big)$ can be replaced by the expression
\begin{equation}
    \widetilde{\Lambda}(t,x)=\Lambda(t,x,\xi(t,x)) + \Lambda_{\mathrm{int}}(t,x,\xi(t,x))
\end{equation}
where $\Lambda_{\mathrm{int}} \in \mathscr{C}^2([0,T]\times [0,1]\times \mathfrak{g};\mathfrak{g})$ models a force exerted at section $x$ by the sections $[x,1]$ on the sections $[0,x]$. This substitution is possible as we can consider in equation~\eqref{Cosserat Lagrangian 0.5} the so-called \emph{internal force $F_{\mathrm{int}}$ induced by $\Lambda_{\mathrm{int}}$} defined by
\begin{equation}
\begin{aligned}
    \langle F_{\mathrm{int}}(t,h,\eta) , v \rangle =& \int_0^1 \langle -\Lambda_{\mathrm{int}}(t,x,\xi_h(x)), v'(x)\rangle dx - \int_0^1\langle \operatorname{ad}_{\xi_h(x)}^*\Lambda_{\mathrm{int}}(t,x,\xi_h(x)), v(x)\rangle dx\\
    =& \int_0^1 \langle \frac{d}{dx} \Lambda_{\mathrm{int}}(t,x,\xi_h(x))-\operatorname{ad}_{\xi_h(x)}^*\Lambda_{\mathrm{int}}(t,x,\xi_h(x)), v(x)\rangle dx \\
    &- \big[\langle \Lambda_{\mathrm{int}}(t,x,\xi_h(x)),v(x) \rangle\big]_0^1
    \end{aligned}
\end{equation} for all $(t,h,\eta)\in [0,T]\times C^2_G \times C^2_\mathfrak{g}$, $v\in C^1_\mathfrak{g}$, or equivalently using the form of Assumption~\ref{Ass : force} with the distributed part
\begin{equation}
\label{def:internal force dens}
    \mathcal{F}_{\mathrm{int}}(t,x,\xi,\partial_x\xi)=
    \partial_x \Lambda_{\mathrm{int}}(t,x,\xi) + D_\xi \Lambda_{\mathrm{int}}(t,x,\xi) \cdot \partial_x \xi - \operatorname{ad}_\xi^* \Lambda_{\mathrm{int}}(t,x,\xi) , \end{equation}
    and the punctual parts
    \begin{equation}
    \label{def:internal force bound}
   F_0(t,\xi)=\Lambda_{\mathrm{int}}(t,0,\xi), \quad F_1(t,\xi)=-\Lambda_{\mathrm{int}}(t,1,\xi).
\end{equation}
\end{remark}
\begin{remark} The internal forces~\eqref{def:internal force dens}-\eqref{def:internal force bound} as defined previously can also be interpreted as reaction forces (see Remark~\ref{rem:reaction force}) of the reference space twist constraint~\eqref{eq:reference shape}.
This is demonstrated in Lemma~\ref{lemma: equivalence inter react forces}.
\end{remark}

\end{itemize}


In this context, we present the tendons model :

\begin{itemize}
\item  Some bio-inspired robots are equipped with tendons to model the internal dynamic of animal such as
snakes (\cite{canali2022design},\cite{rao2021model}). In particular, tendons have been used in static regime to control continuous robots \cite{tummers2023cosserat}. Thus it is relevant to introduce the tendon force as an actuation effect on the shape of the robot.
Given $N$ tendons indexed by $i\in \{1,\dots,N\}$  passing through the sections at coordinates $D_i : [0,1] \rightarrow \mathbb{R}^3$,
expressed in the cross-sectional frame (see Remark \ref{rem:mechanics})
and pulled by the tension $T_i : [0,T] \rightarrow \mathbb{R}^+$, the \emph{tendon force} is an internal force~\eqref{def:internal force dens}-\eqref{def:internal force bound} induced by
\begin{equation}
 \label{def:internal tendon force}
\Lambda_{\mathrm{int}} (t,x,\xi)
= \sum_{i=1}^{N}
\frac{T_i(t)}{\left\lVert \Gamma + K \times D_i(x) + D'_i(x) \right\rVert}
\begin{pmatrix}
    D_i(x) \times \big( \Gamma + K \times D_i(x) + D'_i(x) \big)\\
 \big( \Gamma + K\times D_i(x) + D_i'(x)\big)
\end{pmatrix},
\end{equation}
where $\xi=\begin{pmatrix}
      K \\ \Gamma
    \end{pmatrix}$.
 Hence the tendon force is left invariant and has a complicated structure. To further simplify the presentation, we assume that the tendons cross the sections nearly at their center, i.e. $D_i \approx 0$ for all $i\in \{1,\dots,N\}$. In that case
 $\Lambda_{\mathrm{int}}= \begin{pmatrix}
         0 \\ u_{act},
     \end{pmatrix}$,
     where $u_{act} : [0,T]\times [0,1] \rightarrow \mathbb{R}^3$ is an input function. In particular, if the tendons are only equipped in $N$ subsections $A_i\subset [0,1]$ of the robot, we consider the input function as a piecewise constant function $u_{act}(t,x)= \sum_{i=0}^N u_{act,i}(t)\mathbbm{1}_{A_i}(x)$.
\end{itemize}

\subsection{Towards control systems viewpoint}\label{sec:Towards control systems viewpoint}

Having introduced several input forces in Section~\ref{actuation effect}, we present here control problems for the locomotion of soft robots. For instance, one can study the ability to move the robot from one arbitrary configuration to any other  using actuation forces (also called control inputs).
Note that there exist only a few general control results applying to the Cosserat-Poincar\'e equations we are considering, cf. \cite{rodriguez2020boundaryfeedbackstabilizationintrinsic}. Moreover,
only local controllability results have been established, cf. \cite{strohmeyer2018networks} for a planar robot (modeled by a movement $g\in C^{2,2}_{\SE(2)}$) using boundary controls \eqref{def:boundary force}.
In this section, we obtain rather simple controllability results in the case of full actuation of the distributed forces, using an internal force combined with a boundary force.

Let us consider the control system deriving from the Cosserat-Poincaré  equation~\eqref{Cosserat Lagrangian 0.5} on $[0,T\times [0,1]$\begin{subequations}
\label{eq:control system}
    \begin{empheq}[left=\empheqlbrace]{align}
    \label{eq:control system dyn}&\partial_t(I\eta)-\partial_x\Lambda-\operatorname{ad}_{\eta}^T(I\eta)+\operatorname{ad}_{\xi}^T\Lambda=\mathcal{F}+u_{\mathrm{d}}, \\
    &\partial_x g=g\xi, \quad \partial_t g=g\eta,\label{eq:control system reconstruction}\\
    &\Lambda(t,0)=-F_0(t) - u_0(t),\label{eq:control system bound 1}\\
    &\Lambda(t,1)= F_1(t) + u_1(t),\label{eq:control system bound 2}\\
    \label{eq:control system initial condition}& g(0,\cdot)=g^0, \quad \eta(0,\cdot)=\eta^0,
    \end{empheq}
    \end{subequations}
    where $\Lambda=H(\xi-\xi_0)$ is the stress defined in \eqref{def:stress}, $g^0 : [0,1] \rightarrow G,\eta^0 :[0,1]\rightarrow \mathfrak{g} $ are the initial conditions, $\mathcal{F} :[0,T]\times[0,1]\times G\times \mathfrak{g}^3 \rightarrow \mathfrak{g}$ and $F_0,F_1 : [0,T] \rightarrow \mathfrak{g}$ are respectively the distributed part and the punctual part of an uncontrolled force $F$, $u_{\mathrm{d}} : [0,T] \rightarrow \mathscr{U}_{\mathrm{d}}$ and $u_0,u_1 : [0,T] \rightarrow \mathscr{U}_b$ are measurable control forces with admissible control spaces $\mathscr{U}_{\mathrm{d}}$ and $\mathscr{U}_b$. In the following, we assume that $\mathscr{U}_{\mathrm{d}}=L^2([0,1];\mathfrak{g})$ and $\mathscr{U}_b=\mathfrak{g}$.

We consider next a concept of (weak) solution for \eqref{eq:control system} defined on the state space
${\cal{X}}_0$ and its subspace ${\cal{X}}_1$ defined as
$$
{\cal{X}}_0:=H^1((0,1);G)\times L^2([0,1];\mathfrak{g}),\quad
{\cal{X}}_1:=C^2_G\times C^2_\mathfrak{g}.
$$
On the other hand, we also need to consider the
space ${\cal{Y}}$ defined as
$$
{\cal{Y}}=\mathscr{C}^1\big([0,T];L^2((0,1);G)\big) \cap \mathscr{C}^0\big([0,T];H^1((0,1);G)\big).
$$
\begin{definition}
Let $T>0$, $(g^0,\eta^0)\in {\cal{X}}_0$,
$\mathcal{F},u_\mathrm{d}\in L^2([0,T]\times[0,1];\mathfrak{g})$, $F_0,F_1,u_0,u_1 \in L^2([0,T];\mathfrak{g})$. A (weak) solution to the problem~\eqref{eq:control system} is a function $g\in {\cal{Y}}$ verifying $g(0,\cdot)=g^0,$ and
    \begin{equation}
    \label{eq:control system weak form}
        \begin{aligned}
            \hspace{-2mm}\int_0^T \int_0^1 &\langle \Lambda, \partial_x \phi + \operatorname{ad}_{\xi} \phi  \rangle -\langle I\eta , \partial_t \phi + \operatorname{ad}_{\eta} \phi \rangle  - \langle \mathcal{F} + u_\mathrm{d} , \phi \rangle dx dt
            = \int_0^1 \langle I\eta^0(x),\phi(0,x)\rangle- \langle I\eta(T,x),\phi(T,x)\rangle \,  dx\\ +& \int_0^T \langle F_0(t)+u_0(t), \phi(t,0) \rangle + \langle F_1(t)+u_1(t), \phi(t,1) \rangle \, dt\qquad \forall \phi \in C^{1,1}_\mathfrak{g}
        \end{aligned}
    \end{equation}
    with $\Lambda = H(\xi-\xi_0)$, $\xi = \xi_g$ and $\eta=\eta_g$.
    \end{definition}
    \begin{remark}Equation~\eqref{eq:control system weak form} is well-defined for any $g\in {\cal{Y}}$. Indeed, for such $g$'s, it is easy to prove that $\eta_g,\xi_g \in \mathscr{C}^0([0,T]; L^2([0,1]; \mathfrak{g}))$ as $H^1((0,1);G)$ admits a continuous injection into $L^{\infty}([0,1];G)$. Then by injection of the space $\mathscr{C}^0([0,T];L^2([0,1];\mathfrak{g))}$ into $L^2([0,T]\times [0,1];\mathfrak{g})$, we deduce that $\eta_g,\xi_g \in L^2([0,T]\times [0,1];\mathfrak{g})$ and thus equation~\eqref{eq:control system weak form} is well-defined.
    \end{remark}
    \begin{remark}
    Any solution $g\in C^{2,2}_G$ of ~\eqref{eq:control system weak form} is a solution of \eqref{eq:control system} after integrating by parts and hence $g$ is a classical solution of the partial differential equation~\eqref{eq:control system}.
    \end{remark}
    \begin{remark}
    To the best of our knowledge, there exists no result in the literature on the existence of solutions to the problem~\eqref{eq:control system weak form} for any arbitrary time $T>0$. In this paper, well-posedness is not addressed.
    \end{remark}

We want to address controllability issues for the control system~\eqref{eq:control system}. We give next one possible definition of controllability.

\begin{definition}\label{def:controllability}
     Let $T>0$, $\mathcal{F}\in L^2([0,T]\times[0,1];\mathfrak{g})$, $F_0,F_1 \in L^2([0,T];\mathfrak{g})$. The control system~\eqref{eq:control system} is \emph{controllable in ${\cal{X}}_0$ (resp. in ${\cal{X}}_1$) in time $T$} if for every initial condition $(g^0,\eta^0)$ and final condition $(g^1,\eta^1)$ in ${\cal{X}}_0$ (resp. in ${\cal{X}}_1$), there exist controls $u_{\mathrm{d}}\in L^2([0,T]\times[0,1];\mathfrak{g})$ and $u_0,u_1\in  L^2([0,T];\mathfrak{g})$ such that
     the problem~\eqref{eq:control system} has a solution $g$ satisfying
    \begin{equation}\label{eq:control system final condition}g(T,\cdot)=g^1, \quad \eta_g(T,\cdot)=\eta^1.\end{equation}
\end{definition}

\begin{remark}
The previous definition using
both distributed and boundary controls
can of course be modified: one can consider only
internal forces or boundary forces, or even, by restricting the shape of the internal forces to the form given in Section~\ref{actuation effect}, controls such as tendon or magnetic forces. Every such modification yields a different control problem to be addressed.
\end{remark}


The following proposition shows that the control system~\eqref{eq:control system} is controllable for regular initial and final conditions.



\begin{proposition}\label{pro:total controllability avec bord}
Assume that the Lie group $G$ is path-connected.
Let $T>0$, $\mathcal{F}\in \mathscr{C}^0([0,T]\times [0,1];\mathfrak{g})$ and $F_0,F_1\in \mathscr{C}^2([0,T];\mathfrak{g})$.
Then the system~\eqref{eq:control system} is controllable in ${\cal{X}}_1$ in time $T>0$.
\end{proposition}
The proof of the proposition relies on the following lemma.

\begin{lemma}\label{lemma:construction explicite trajectoire avec bord}
   Let $T>0$, an initial condition $(g^0,\eta^0)$  and a final condition $(g^1,\eta^1)$ both in ${\cal{X}}_1$. Then there exists a movement $g_r \in C^{2,2}_G$ such that
   \begin{equation}\label{eq:construction ini and fin cond avec bord}
       (g_r(0,\cdot),\eta_{g_r}(0,\cdot)) = (g^0,\eta^0), \quad      (g_r(T,\cdot),\eta_{g_r}(T,\cdot)) = (g^1,\eta^1). \end{equation}
\end{lemma}

\begin{proof}[Proof of Lemma~\ref{lemma:construction explicite trajectoire avec bord}]\label{proof:control total avec bord}
  Let $T>0$,  $(g^0,\eta^0)$ and $(g^1,\eta^1)$ in
 ${\cal{X}}_1$.
 We first determine a curve $\gamma$ steering $(g^0(0),\eta^0(0))$ to $(g^1(0),\eta^1(0))$ and then take a suitable $\xi_r$ such that the movement $g_r\in C^{2,2}_G$ defined as the solution of
 \begin{equation}
\begin{cases}
\begin{aligned}\label{eq:construction reference trajectory avec bord1}
   & \partial_x g(t,x) = g(t,x) \xi_r(t,x), \\
   & g(t,0)=\gamma(t).
\end{aligned}
\end{cases}
\end{equation}  satisfies the equations~\eqref{eq:construction ini and fin cond avec bord}.

  The Lie group $G$ being path-connected, we can take a curve $\gamma \in \mathscr{C}^2([0,T],G)$ such that
  $(\gamma(0),\gamma'(0))=(g^0(0),g^0(0)\eta^0(0))$ and $(\gamma(T),\gamma'(T))=(g^1(T),g^1(T)\eta^1(T))$.
  Then we try to determine a space twist $\xi_r$ such that it satisfies the conditions
    \begin{equation}\label{eq:condtion construction initiale et finale avec bord}
    \big(\xi_r(0,x),\partial_t \xi_r (0,x)\big) = \big( \xi^0(x),\xi_t^0(x) \big), \quad    \big(\xi_r(T,x),\partial_t \xi_r (T,x)\big) = \big( \xi^1(x),\xi_t^1(x) \big),
    \end{equation}
    where for $i\in \{0,1\}$,
    \begin{equation}\label{def:xi_t avec bord}
    \xi^i:=(g^i)^{-1}(g^i)', \quad \xi^i_t:= \partial_x\eta^i + [\xi^i, \eta^i].\end{equation}
    Consider four polynomials $P_{1}, P_{2}, P_{3}, P_{4}$ such that
\begin{equation}
 P_{1}(0)=1, \quad P_2'(0)=1, \quad P_{3}(T)=1, \quad  P_{4}'(T)=1,
\end{equation}
and all other values and first derivatives vanish at $t=0$ and $t=T$.
One then easily checks that the function
\begin{equation}\label{eq:interpolation}
    \xi_r(t,x) :=  \xi^0(x)P_{1}(t) + \xi^0_t(x)P_{2}(t) + \xi^1(x)P_{3}(t) + \xi^1_t(x) P_{4}(t)
    \end{equation}
    satisfies conditions~\eqref{eq:condtion construction initiale et finale avec bord} and has the regularity $C^{2,1}_{\mathfrak{g}}$. We can now consider the movement $g_r\in C^{2,2}_{G}$ defined as the solution of \eqref{eq:construction reference trajectory avec bord1}.
By construction, the space twist of the movement $g_r$ satisfies $\xi_{g_r}(t,\cdot)=\xi_r(t,\cdot)$ for all times $t\in[0,T]$. In particular, it holds
\begin{equation}\label{eq:xi_g_r:1}
    \xi_{g_r}(0,\cdot)=\xi_r(0,\cdot)=\xi^0,
    \quad
    \xi_{g_r}(T,\cdot)=\xi_r(T,\cdot)=\xi^1
\end{equation}
and
\begin{equation}\label{eq:D_t_xi_g_r:1}
    \partial_t \xi_{g_r}(0,\cdot)
    =\partial_t \xi_r(0,\cdot)=\xi_t^0,
    \quad
    \partial_t \xi_{g_r}(T,\cdot)=\partial_t \xi_r(T,\cdot)=\xi_t^1
\end{equation}
on $[0,1]$.

Because $g_r(0,0)=\gamma(0)=g^0(0)$ and  $g_r(T,0)=\gamma(T)=g^0(T)$, we deduce that the movement $g_r$ satisfies
\begin{equation}
    g_r(0,\cdot)=g^0, \quad g_r(T,\cdot)=g^1
    \quad\textrm{on $[0,1]$}
\end{equation}
as both functions $g_r(0,\cdot)$ and $g^0$
(resp. $g_r(T,\cdot)$ and $g^1$)
satisfy the same ordinary differential equation~\eqref{eq:construction reference trajectory avec bord1} with the same initial (resp. final) conditions:
$\partial_x g_r(0,x)=g_r(0,x)\xi_r(0,x)=g_r(0,x)\xi^0(x)$
and $\partial_x g^0(x)=g^0(x)\xi^0(x)$ (recall \eqref{def:xi_t avec bord})
and similarly for $g_r(T,x)$ and $g^1$.

Last, given the definition~\eqref{def:xi_t avec bord} of $\xi^0_t$ and $\xi^1_t$, and the structure equation
\begin{equation}
    \partial_t \xi_{g_r} = \partial_x \eta_{g_r} + [ \xi_{g_r},\eta_{g_r}],
\end{equation}
we deduce that
\[
\eta_{g_r}(0,\cdot)=\eta^0,
\quad
\eta_{g_r}(T,\cdot)=\eta^1.
\]
Indeed, the functions $\eta_{g_r}(0,\cdot)$ and $\eta^0$ (resp. $\eta_{g_r}(T,\cdot)$ and $\eta^1$)
are both solutions of the same ordinary differential initial (resp. final) value problems
\begin{equation}
\begin{aligned}
&\begin{cases}
       \partial_x \eta(x)=  \xi^0_t(x) - [ \xi^0(x),\eta(x)], \\
       \eta(0)=\eta^0(0),
       \end{cases}
       \text{and resp. } & \begin{cases}
       \partial_x \eta(x)=  \xi^1_t(x) - [ \xi^1(x),\eta(x)], \\
       \eta(T)=\eta^1(T),
       \end{cases}
       \end{aligned}
\end{equation}
because $\xi^0=\xi_{g_r}(0,\cdot)$, $\xi_t^0=\partial_t\xi_{g_r}(0,\cdot)$,  $\xi^1=\xi_{g_r}(T,\cdot)$, $\xi_t^1=\partial_t\xi_{g_r}(T,\cdot)$
by \eqref{eq:xi_g_r:1}-\eqref{eq:D_t_xi_g_r:1}
and because $\eta_{g_r}(0,0)=g(0,0)^{-1}\partial_t g(0,0)=\gamma(0)^{-1}\gamma'(0)=\eta^0(0)$ and $\eta_{g_r}(T,0)=g(T,0)^{-1}\partial_t g(T,0)=\gamma(T)^{-1}\gamma'(T)=\eta^1(T)$.
We conclude that the movement $g_r\in C^{2,2}_{G}$ satisfies~\eqref{eq:construction ini and fin cond avec bord}.
\end{proof}

\begin{proof}[Proof of Proposition~\ref{pro:total controllability avec bord}]
Let $T>0$, $\mathcal{F}\in \mathscr{C}^0([0,T]\times[0,1];\mathfrak{g})$, $F_0,F_1\in \mathscr{C}^2([0,T];\mathfrak{g})$, an initial condition $(g^0,\eta^0)$  and a final conditions $(g^1,\eta^1)$ both in $\mathcal{X}_1$.
Taking $g_r\in \mathscr{C}^2([0,T]\times[0;1];G)$ given by Lemma~\ref{lemma:construction explicite trajectoire avec bord}, we consider the internal control
\begin{equation}
     u_{\mathrm{d}}:=\partial_t(I\eta_{g_r}) - \partial_x(H(\xi_{g_r} - \xi_0)) - \operatorname{ad}_{\eta_{g_r}}^*(I\eta_{g_r}) + \operatorname{ad}_{\xi_{g_r}}^*(H(\xi_{g_r} - \xi_0))-\mathcal{F}(t,x),
\end{equation} and the boundary controls
\begin{equation}
    u_0(t) = -H( \xi_{g_r}(t)-\xi_0(0)) - F_0(t),
    \quad
    u_1(t) =  H( \xi_{g_r}(t)-\xi_0(1)) - F_1(t).
\end{equation}
It is immediate to check that $g=g_r$ is a solution of~\eqref{eq:control system}.
Given that $g_r$ satisfies \eqref{eq:construction ini and fin cond avec bord}, this proves the controllability of the system.
\end{proof}
Following the same method, it is possible, at the cost of compatibility conditions given in Appendix~\ref{def:compatibility condition},  to prove the controllability of the system~\eqref{eq:control system} in ${\cal{X}}_1$ with no boundary controls, i.e., so that  $u_0=u_1=0$.
\begin{proposition}\label{pro:total controllability}
Assume that the Lie group $G$ is path-connected.
Let $T>0$, $\mathcal{F}\in \mathscr{C}^0([0,T]\times[0,1];\mathfrak{g})$ and $F_0,F_1\in \mathscr{C}^2([0,T];\mathfrak{g})$. The system~\eqref{eq:control system} is controllable for regular initial and final conditions without boundary controls.
In other words, for every initial conditions $(g^0,\eta^0)\in \mathcal{X}_1$ satisfying the initial boundary conditions~\eqref{eq:ini compatibility condition} with $u_0=u_1=0$ and $(g^1,\eta^1)\in  \mathcal{X}_1$ satisfying the final compatibility conditions~\eqref{eq:fin compatibility condition} with $u_0=u_1=0$,
there exists a control $u_{\mathrm{d}}\in \mathscr{C}^0([0,T]\times [0,1]; \mathfrak{g})$ such the solution $g$ of equation~\eqref{eq:control system} with $u_0=u_1=0$ is in space $C^{2,2}_G$ and satisfies
     \begin{equation}g(T,\cdot)=g^1, \quad \eta_g(T,\cdot)=\eta^1.\end{equation}
\end{proposition}

The proof of Proposition~\ref{pro:total controllability}, which is essentially similar to the proof of Proposition~\ref{pro:total controllability avec bord},
is given in Appendix~\ref{def:compatibility condition}.

\begin{remark}
In Proposition~\ref{pro:total controllability}, the control $u_\mathrm{d}$ is designed to track a trajectory $g_r$. In particular, this choice of control ensures the existence of a solution to~\eqref{eq:control system weak form}, thereby avoiding the need to rely on a general existence result.
\end{remark}


We next provide another controllability result, this time where the actuation takes the form of an induced internal force.
\begin{corollary}\label{coro : control internal force}
Assume that the Lie group $G$ is path-connected.
Let $T>0$, $\mathcal{F}\in \mathscr{C}^0([0,T]\times[0,1];\mathfrak{g})$ and $F_0,F_1\in \mathscr{C}^2([0,T];\mathfrak{g})$.
Assuming the initial and final conditions $(g^0,\eta^0)$ and $(g^1,\eta^1)$ belong to $\mathcal{X}_1$ and satisfy the compatibility conditions with $u_0=u_1=0$, then there exist an internal force $F_{\mathrm{int}}$ induced by a function $\Lambda\in C^{0,1}_{\mathfrak{g}}$ and $\tilde{u}_1\in \mathscr{C}^0([0,T];\mathfrak{g})$ such that
     $$u_{\mathrm{d}}=\mathcal{F}_{\mathrm{\mathrm{int}}}=\partial_x \Lambda_{\mathrm{int}} - \operatorname{ad}^T_\xi \Lambda_{\mathrm{int}}$$
     and
     $$u_0(t)= \Lambda(t,0), \quad u_1(t) = -\Lambda(t,1) + \tilde{u}_1(t), \quad \forall t\in[0,T]$$
     steer the control system~\eqref{eq:control system} from $(g^0,\eta^0)$ to $(g^1,\eta^1)$.
\end{corollary}

\begin{proof}[Proof of Corollary~\ref{coro : control internal force}]
 Let $\mathcal{F}\in\mathscr{C}^0([0,T]\times[0,1];\mathfrak{g})$, $F_0,F_1\in \mathscr{C}^2([0,T];\mathfrak{g})$, $(g^0,\eta^0)$ and $(g^1,\eta^1)$ in $\mathcal{X}_1$ satisfy the compatibility conditions with $u_0=u_1=0$. Take a control $u\in \mathscr{C}^0([0,T]\times [0,1]; \mathfrak{g})$ which steers the control system~\eqref{eq:control system} with $u_0=u_1=0$ from $(g^0,\eta^0)$ to $(g^1,\eta^1)$
 and whose existence is guaranteed by Proposition \ref{pro:total controllability}.
 Denote by $g_r$ the corresponding solution of equations~\eqref{eq:control system} and write $\xi_r=g_r^{-1}\partial_x g_r$ for its space twist.
 Let $\Lambda_{\mathrm{int}}\in C^{0,1}_{G}$ be the solution of
 \begin{equation}
     \begin{cases}
         \partial_x \Lambda_{\mathrm{int}}(t,x) = \operatorname{ad}^T_{\xi_r}\Lambda_{\mathrm{int}}(t,x) + u(t,x),\\
         \Lambda_{\mathrm{int}}(t,0)=0,
     \end{cases}
 \end{equation}
 and introduce the functions $u_{\mathrm{d}}:=\partial_x\Lambda_{\mathrm{int}} - \operatorname{ad}^T_{\xi_r}\Lambda_{\mathrm{int}}\in \mathscr{C}^0([0,T]\times [0,1];\mathfrak{g})$
 and $\tilde{u}_1:=\Lambda(\cdot,1)\in \mathscr{C}^0([0,T];\mathfrak{g})$.
 Then one deduces that the control $u_\mathrm{d}$ combined with the boundary controls
 $$
 u_0(t):=\Lambda(t,0), \quad u_1(t):=-\Lambda(t,1)+\tilde{u}_1(t),
 $$
 will still steer the control system~\eqref{eq:control system} from $(g^0,\eta^0)$ to $(g^1,\eta^1)$ since $g_r$ remains the solution of equations~\eqref{eq:control system}-\eqref{eq:control system initial condition}.
\end{proof}

\begin{remark}
Proposition~\ref{pro:total controllability} shows that the solution can be controlled using a distributed force $u_{\mathrm{d}}$ only.
This result is not, however, useful in practice for the robot locomotion as there is no practical way to generate an arbitrary distributed force.
On the other hand, Corollary~\ref{coro : control internal force} proves that the control can be interpreted as an internal force combined with a boundary force at one end. Since internal forces model shape deformations, this implies that a shape-controlled robot with a single boundary actuator is controllable.
\end{remark}

\section{Appendix}

\subsection{Compatibility conditions}\label{section:compatibility condition}

\begin{definition}\label{def:compatibility condition}   Let $T>0$, $\mathcal{F}\in \mathscr{C}^0([0,T]\times[0,1];\mathfrak{se}(3))$, $F_0,F_1\in \mathscr{C}^2([0,T];\mathfrak{se}(3))$ and $u_0,u_1\in \mathscr{C}^2([0,T];\mathfrak{se}(3))$.
    We say that an initial condition $(g^0,\eta^0)$ and a final condition $(g^1,\eta^1)$ in $\mathcal{X}_1$ satisfy the \emph{compatibility conditions at the boundary $x_0=0$ and $x_1=1$} if for $i\in \{0,1\}$
   \begin{equation}
   \label{eq:ini compatibility condition}
        \xi_{g^0}(x_i)=\xi_0(x_i) + H(x_i)^{-1}(F_i(0) + u_i(0)), \quad  (\xi_{g^0})_t(x_i)= H(x_i)^{-1} (F_i'(0) + u_i'(0)),
    \end{equation}
    and
   \begin{equation}
      \label{eq:fin compatibility condition}
        \xi_{g^1}(x_i)=\xi_0(x_i) + H(x_i)^{-1}(F_i(T) + u_i(T)), \quad  (\xi_{g^1})_t(x_i)= H(x_i)^{-1} (F_i'(T) + u_i'(T)),
    \end{equation}
    where $H : [0,1] \rightarrow \mathscr{L}(\mathfrak{se}(3))$ is the Hookian matrix defined in \eqref{def:potential energy}, $\xi_0\in \mathscr{C}^2([0,1];\mathfrak{se}(3))$ is the space twist at rest,
    and
\[
\xi_{g^0}(x):=g^0(x)^{-1} \partial_x g^0(x),
\quad
(\xi_{g^0})_t(x):= \partial_x\eta^0(x) + [\xi_{g^0}(x), \eta^0(x)], \\
\xi_{g^1}(x):=g^1(x)^{-1} \partial_x g^1(x),
\quad
(\xi_{g^1})_t(x):= \partial_x\eta^1(x) + [\xi_{g^1}(x), \eta^1(x)],
\]
for $x\in [0,1]$.
    \end{definition}

The proof of Proposition~\ref{pro:total controllability} relies on the following lemma that shows the existence of a movement $g_r\in \mathscr{C}^2([0, T]\times[0,1]; G)$ satisfying the equations~\eqref{eq:control system bound 1};\eqref{eq:control system bound 2};\eqref{eq:control system initial condition}.

\begin{lemma}\label{lemma:construction explicite trajectoire}
   Let $T>0$, $\mathcal{F}\in \mathscr{C}^0([0,T]\times[0,1];\mathfrak{se}(3))$, $F_0,F_1\in \mathscr{C}^2([0,T];\mathfrak{se}(3))$, an initial condition $(g^0,\eta^0)$  and a final conditions $(g^1,\eta^1)$ both in $\mathcal{X}_1$ satisfying respectively the initial and final compatibility conditions of Definition~\ref{def:compatibility condition} with $u_0=u_1=0$. Then there exists a movement $g_r \in C^{2,2}_G$ such that \begin{equation}\label{eq:construction ini and fin cond}
       (g_r(0,\cdot),\eta_r(0,\cdot)) = (g^0,\eta^0), \quad      (g_r(T,\cdot),\eta_r(T,\cdot)) = (g^1,\eta^1), \end{equation}
 with $\eta_r= g_r^{-1}\partial_tg_r $
and \begin{equation}\label{eq:contruction bound}
    \xi_r(\cdot,0)=\xi_0(1) + H(0)^{-1}F_0, \quad \xi_r(\cdot,1)=\xi_0(0) + H(1)^{-1}F_1,
\end{equation}
where $\xi_r=g_r^{-1}\partial_xg_r $ and $\xi_0$ is the reference space twist~\eqref{def:ST0}.
\end{lemma}

\begin{proof}[Proof of Lemma~\ref{lemma:construction explicite trajectoire}]\label{proof:control total}
The proof of this lemma is entirely similar to that of Lemma~\ref{lemma:construction explicite trajectoire avec bord} with one single modification: instead of choosing $\xi_r$ given in \eqref{eq:interpolation}, one chooses it as
\begin{multline*}
\xi_r(t,x):= f_r(t,x) +(1-x)\big(H(x)^{-1}F_0(t)+\xi_0(0) - f_r(t,0)\big) + x\big( H(x)^{-1}F_1(t) + \xi_0(1) - f_r(t,1)\big).
\end{multline*}
where, for all $(t,x)\times [0,T]\times [0,1]$, $f_r(t,x):=\xi^0(x)P_{1}(t) + \xi^0_t(x)P_{2}(t) +  \xi^1(x)P_{3}(t) + \xi^1_t(x) P_{4}(t).$
Indeed, by direct computations, using the compatibility conditions at $x=0$ and $x=1$, it follows that $\xi_r$ satisfies \eqref{eq:construction ini and fin cond} and  \eqref{eq:contruction bound} .
 \end{proof}

Using Lemma~\ref{lemma:construction explicite trajectoire},
one can prove Proposition~\ref{pro:total controllability}
in exactly the same way as Proposition~\ref{pro:total controllability avec bord}
was shown.

\subsection{Proof of Theorem \ref{least action thm}}\label{app:least action thm}

Below, for compact intervals $I,J,K\subset\R$ and a set $S\subset\R^d$, we shall denote
\[
D^{k,l,m}(I\times J\times K;S):=\{f:I\times J\times K\to \R^d\ |\ {}& \textrm{partial derivatives}\ (h,t,x)\mapsto \partial_h^\alpha \partial_t^\beta \partial_x^\gamma f(h,t,x)
\ \textrm{exist and} \\
{}& \textrm{are continuous on $I\times J\times K$ for all $\alpha\leq k,\ \beta\leq l,\ \gamma\leq m$} \\
{}& \textrm{and $f(I\times J\times K)\subset S$}\}
\]
Theorem \ref{th:D^{k,l}:1} obviously generalizes to yield the
normed space isomorphisms
\begin{align}\label{eq:D^{k,l,m}:1}
D^{k,l,m}(I\times J\times K;S)\cong{}& \mathscr{C}^k(I,D^{l,m}(J\times K;S))
\cong \mathscr{C}^k(I,D^{l,m}(J\times K;S))
\cong \mathscr{C}^k(I,\mathscr{C}^l(J,\mathscr{C}^m(K;S)))
\end{align}
and so on.

\begin{lemma}\label{le:variation:1}
For any $g\in C^{2,2}_G$ and $v\in C^{2,2}_{\mathfrak{g}}$,
there exists $\wt{g}\in \mathscr{C}^2([-1,1];C^{2,2}_G)$
such that
\begin{align}\label{eq:le:variation:1:result:1}
\wt{g}(0)=g
\quad\textrm{and}\quad
\wt{g}'(0)=g v.
\end{align}
\end{lemma}

\begin{proof}
Let $g\in C^{2,2}_G$ and $v\in C^{2,2}_{\mathfrak{g}}$.
By Theorem \ref{th:D^{k,l}:1},
we have $g\in D^{2,2}([0,T]\times [0,1];G)$
and $v\in D^{2,2}([0,T]\times [0,1];\mathfrak{g})$

Define
\[
X:=[-1,1]\times [0,T]\times [0,1]
\]
and
\[
\wt{g}(h,t,x):=g(t,x)\exp(hv(t,x)),
\quad (h,t,x)\in X,
\]
where $\exp$ is the usual matrix exponential map.
Note that since we are working in a matrix Lie group $G$,
this exponential map coincides with the exponential map $\exp_G$ of the Lie group $G$.

Clearly $\wt{g}\in D^{2,2,2}(X;G)$
and hence $\wt{g}\in \mathscr{C}^2([-1,1];C^{2,2}_G)$
similarly as in Theorem \ref{th:D^{k,l}:1}
(see \eqref{eq:D^{k,l,m}:1}).
Note that
\[
\wt{g}(0,t,x)=g(t,x)
\quad\forall (t,x)\in [0,T]\times [0,1]
\]
and
\[
\partial_h \wt{g}(h,t,x)|_{h=0}=g(t,x)v(t,x)
\quad\forall (t,x)\in [0,T]\times [0,1].
\]

For clarity, let us write $f^{\wt{g}}$ for $\wt{g}$
as an element of $\mathscr{C}^2([-1,1];C^{2,2}_G)$
i.e. $f^{\wt{g}}(h)(t,x)=\wt{g}(h,t,x)$ for all $(h,t,x)\in X$.
By the above $f^{\wt{g}}(0)=\wt{g}(0,t,x)=g$.
On the other hand, similarly to \eqref{eq:th:D^{k,l}:1:result:2},
it holds
\[
\partial_h (f^{\wt{g}}(h))(t,x)=\partial_h \wt{g}(h,t,x)
\quad \forall (h,t,x)\in X
\]
after identifying on the left-hand side the space
$C^{2,2}_G$ with $D^{2,2}([0,T]\times [-1,1])$
using Theorem \ref{th:D^{k,l}:1}.

Finally, at $h=0$, this yields
\[
\partial_h (f^{\wt{g}}(h))|_{h=0}(t,x)
=g(t,x)v(t,x)
\]
by the above i.e.
\[
(f^{\wt{g}})'(0)=g v.
\]
This completes the proof.
Thus \eqref{eq:le:variation:1:result:1} holds
up to identifying $f^{\wt{g}}$ with $\wt{g}$.
\end{proof}

By Assumption \ref{Ass : lagrangian}, the action $S$ is of the form
\[
S(g)=\int_0^T L(t,g(t,\cdot),\eta_g(t,\cdot)) dt
=\int_{[0,T]\times [0,1]} \mc{L}(t,g(t,x),\xi_g(t,x),\eta_g(t,x))\, d(t,x),
\quad g\in C^{2,2}_G
\]
with $\mc{L}\in \mathscr{C}^1([0,T]\times G\times\mathfrak{g}^2;\R)$.

Suppose that $\wt{g}\in \mathscr{C}^2([-1,1];C^{2,2}_G)$
and note that $\mathscr{C}^2([-1,1]; C^{2,2}_G)\cong D^{2,2,2}([-1,1]\times [0,T]\times [0,1];G)$ similarly as in Theorem \ref{th:D^{k,l}:1}
(see \eqref{eq:D^{k,l,m}:1}).
Let us denote
\[
X:=[-1,1]\times [0,T]\times [0,1].
\]
In particular $\wt{g}\in \mathscr{C}^2(X;G)$
and we may write $\wt{g}(h,t,x)$ for $\wt{g}(h)(t)(x)$.
Similarly, as
$$
(h\mapsto \xi_{\wt{g}(h)})\in \mathscr{C}^2([-1,1];C^{2,1}_{\mathfrak{g}})\cong D^{2,2,1}(X;\mathfrak{g}),\quad
(h\mapsto \eta_{\wt{g}(h)})\in \mathscr{C}^2([-1,1];C^{1,2}_{\mathfrak{g}})\cong D^{2,1,2}(X;\mathfrak{g}),
$$
we may write $\xi_{\wt{g}}(h,t,x)$ for $\xi_{h(h)}(t,x)$
and $\eta_{\wt{g}}(h,t,x)$ for $\eta_{\wt{g}(h)}(t,x)$.

Let
\begin{align}\label{eq:v_{wt{g}}:1}
v_{\wt{g}}(h,t,x):=\wt{g}(h,t,x)^{-1}\partial_h \wt{g}(h,t,x)\in\mathfrak{g}
\end{align}
and note that (see again Theorem \ref{th:D^{k,l}:1})
\[
v_{\wt{g}}\in \mathscr{C}^1([-1,1]; C^{2,2}_{\mathfrak{g}})\cong D^{1,2,2}(X;\mathfrak{g}).
\]

Since $\wt{g}\in D^{2,2,2}(X;G)\subset \mathscr{C}^2(X;G)$,
we have
$\partial_h \partial_x \wt{g}(h,t,x)=\partial_x \partial_h \wt{g}(h,t,x)$
for all $(h,t,x)\in X$
and hence
\[
\wt{g}(h,t,x)^{-1}\partial_h \partial_x \wt{g}(h,t,x)
={}& \wt{g}(h,t,x)^{-1}\partial_x \partial_h \wt{g}(h,t,x)
=\wt{g}(h,t,x)^{-1}\partial_x (\wt{g}(h,t,x) v_{\wt{g}}(h,t,x)) \\
={}& \xi_{\wt{g}}(h,t,x)v_{\wt{g}}(h,t,x)+\partial_x v_{\wt{g}}(h,t,x).
\]
Therefore
\[
\partial_h \xi_{\wt{g}}(h,t,x)
={}& \partial_{h} (\wt{g}(h,t,x)^{-1}\partial_x \wt{g}(h,t,x)) \\
={}& -\wt{g}(h,t,x)^{-1} (\partial_{h} \wt{g}(h,t,x)) \wt{g}(h,t,x)^{-1} \partial_x \wt{g}(h,t,x)
+\wt{g}(h,t,x)^{-1} \partial_h \partial_x \wt{g}(h,t,x)) \\
={}& -v_{\wt{g}}(h,t,x) \xi_{\wt{g}}(h,t,x)
+\xi_{\wt{g}}(h,t,x)v_{\wt{g}}(h,t,x)+\partial_x v_{\wt{g}}(h,t,x)
\]
i.e.
\[
\partial_h \xi_{\wt{g}}(h,t,x)-\partial_x v_{\wt{g}}(h,t,x)
=[\xi_{\wt{g}}(h,t,x),v_{\wt{g}}(h,t,x)]_{\mathfrak{g}}.
\]
Likewise
\[
\partial_h \eta_{\wt{g}}(h,t,x)-\partial_t v_{\wt{g}}(h,t,x)
=[\eta_{\wt{g}}(h,t,x),v_{\wt{g}}(h,t,x)]_{\mathfrak{g}}.
\]

Write
\[
\mc{L}_{\wt{g}}(h,t,x):=\mc{L}(t,x,\wt{g}(h,t,x), \xi_{\wt{g}}(h,t,x), \eta_{\wt{g}}(h,t,x))
\]
and
\[
(\mc{L}_{\wt{g}})_k(h,t,x):=\mc{L}_k(t,x,{\wt{g}}(h,t,x), \xi_{\wt{g}}(h,t,x), \eta_{\wt{g}}(h,t,x)),
\quad k=1,2,3,4,5.
\]
For simplicity, let us write $\la\cdot,\cdot\ra=\la\cdot,\cdot\ra_{\crg}$
for the dual-coupling between $\mathfrak{g}^*$ and $\mathfrak{g}$.
Then we have, suppressing also the argument $(h,t,x)$ for better readability,
\[
\partial_h \mc{L}_{\wt{g}}(h,t,x)
={}& \big(\mc{L}_{\wt{g}}\big)_3 (\wt{g} v_{\wt{g}})+\la (\mc{L}_{\wt{g}})_4, \partial_h \xi_{\wt{g}}\ra+\la (\mc{L}_{\wt{g}})_5, \partial_h \eta_{\wt{g}}\ra \\
={}& \la (\wt{g}^*(\mc{L}_{\wt{g}})_3), v_{\wt{g}}\ra
+\la (\mc{L}_{\wt{g}})_4, \partial_x v_{\wt{g}}
+\ad_{\xi_{\wt{g}}} v_{\wt{g}}\ra
+\la (\mc{L}_{\wt{g}})_5, \partial_t v_{\wt{g}}
+\ad_{\eta_{\wt{g}}} v_{\wt{g}}\ra.
\]
From the assumption $\mc{L}\in \mathscr{C}^1([0,T]\times G\times\mathfrak{g}^2;\R)$
and the fact that $\wt{g},\xi_{\wt{g}}$ and $\eta_{\wt{g}}$ are continuous on $X$,
the functions
$(\mc{L}_{\wt{g}})_k$, $k=3,4,5$,
are also continuous on $X$.
Then because $v_{\wt{g}}$, $\partial_x v_{\wt{g}}$ and $\partial_t v_{\wt{g}}$ are continuous on $X$,
we conclude that $\partial_h \mc{L}_{\wt{g}}$ is continuous on $X$.

In particular, the derivative $\partial_h \mc{L}_{\wt{g}}:X\to\R$ is bounded on the compact set $X$
i.e. there exists a constant $M>0$ s.t. $|\partial_h \mc{L}_{\wt{g}}|\leq M$ on $X$.
Since $[0,1]\times [0,T]$ has finite Lebesgue measure,
the constant function $M$ is $L^1$-integrable on it.
Hence it follows that (e.g. by DCT)
\[
\dif{h} S(\wt{g}(h))=\pa{h} \int_{[0,T]\times [0,1]} \mc{L}_{\wt{g}}(h,t,x)\, d(t,x)
=\int_{[0,T]\times [0,1]} \partial_h \mc{L}_{\wt{g}}(h,t,x)\, d(t,x)
\]
for every $s\in [-1,1]$.

Integration by parts yields
\begin{multline*}
\int_{[0,T]\times [0,1]} \la (\mc{L}_{\wt{g}})_4(h,t,x), \partial_x v_{\wt{g}}(h,t,x)\ra \, d(t,x) \\
=\int_0^T \la(\mc{L}_{\wt{g}})_4(h,t,x), v_{\wt{g}}(h,t,x)\ra\big|_{x=0}^{x=1}\, dt
-\int_{[0,T]\times [0,1]} \la \partial_x (\mc{L}_{\wt{g}})_4(h,t,x), v_{\wt{g}}(h,t,x)\ra d(t,x)
\end{multline*}
and similarly
\begin{multline*}
\int_{[0,T]\times [0,1]} \la (\mc{L}_{\wt{g}})_5(h,t,x),\partial_t v_{\wt{g}}(h,t,x)\ra \, d(t,x) \\
=\int_0^1 \la (\mc{L}_{\wt{g}})_5(h,t,x), v_{\wt{g}}(h,t,x)\ra\big|_{t=0}^{t=T}\, dx
-\int_{[0,T]\times [0,1]} \la\partial_t (\mc{L}_{\wt{g}})_5(h,t,x), v_{\wt{g}}(h,t,x)\ra d(t,x).
\end{multline*}
Therefore, introducing
$\mc{E}_{\wt{g}}:X\to\mathfrak{g}^*$ by
\[
\mc{E}_{\wt{g}}:={}& -\wt{g}^*(\mc{L}_{\wt{g}})_3+\partial_x (\mc{L}_{\wt{g}})_4-(\mc{L}_{\wt{g}})_4\circ \ad_{\xi_{\wt{g}}}+\partial_t (\mc{L}_{\wt{g}})_5-(\mc{L}_{\wt{g}})_5\circ \ad_{\eta_{\wt{g}}} \\
={}& -\wt{g}^*(\mc{L}_{\wt{g}})_3+\partial_x (\mc{L}_{\wt{g}})_4-\ad_{\xi_{\wt{g}}}^*(\mc{L}_{\wt{g}})_4+\partial_t (\mc{L}_{\wt{g}})_5-\ad_{\eta_{\wt{g}}}^* (\mc{L}_{\wt{g}})_5,
\]
we have
\begin{align}\label{eq:diff_of_S:1}
\dif{h} S(\wt{g}(h))
={}& -\int_{[0,T]\times [0,1]} \la\mc{E}_{\wt{g}}(h,t,x), v_{\wt{g}}(h,t,x)\ra\, d(t,x) \\
{}& +\int_0^T \la(\mc{L}_{\wt{g}})_4(h,t,x), v_{\wt{g}}(h,t,x)\ra\big|_{x=0}^{x=1}\, dt
+\int_0^1 \la(\mc{L}_{\wt{g}})_5(h,t,x), v_{\wt{g}}(h,t,x)\ra\big|_{t=0}^{t=T}\, dx
\nonumber
\end{align}

For $g\in C^{2,2}_G$ and $(t,x)\in [0,T]\times [0,1]$, write
\[
j_g(t,x):=(g(t,x),\xi_g(t,x),\eta_g(t,x))
\]
and
\[
\mc{L}_g(t,x):=\mc{L}(t,x,j_g(t,x)),
\quad
(\mc{L}_g)_k(t,x):=\mc{L}_k(t,x,j_g(t,x)),
\quad 1\leq k\leq 5
\]
as well as
\begin{align}\label{eq:mc{E}_g:1}
\mc{E}_g:={}& -g^*(\mc{L}_g)_3+\partial_x (\mc{L}_g)_4-\ad_{\xi_g}^*(\mc{L}_g)_4+\partial_t (\mc{L}_g)_5-\ad_{\eta_g}^* (\mc{L}_g)_5.
\end{align}

Let now $g\in C^{2,2}_G$ and $v\in \mathscr{C}^2_0([0,T],C^2_{\mathfrak{g}})$ be arbitrary
and identify them as elements of $D^{2,2}([0,T]\times [0,1];G)$
and $\{W\in D^{2,2}([0,T]\times [0,1];\mathfrak{g})\ |\ W(0,\cdot)=W(T,\cdot)=0\}$,
respectively, using Theorem \ref{th:D^{k,l}:1}.
Then by Lemma \ref{le:variation:1},
there exists $\wt{g}\in \mathscr{C}^2([-1,1];C^{2,2}_G)$
such that $\wt{g}(0)=g$ and $\wt{g}'(0)=gv$.
Then $v_{\wt{g}}$ given by \eqref{eq:v_{wt{g}}:1}
satisfies
\[
v_{\wt{g}}(0,t,x)=\wt{g}(0)(t,x)^{-1}\,\wt{g}'(0)(t,x)
=g^{-1} (gv)=v
\]
and hence \eqref{eq:diff_of_S:1} at $h=0$ becomes
\[
\mc{D} S(g)(v)=\dif{h} S(\wt{g}(h))\big|_{h=0}
={}& -\int_{[0,T]\times [0,1]} \la\mc{E}_g(t,x), v(t,x)\ra\, d(t,x)
+\int_0^T \la (\mc{L}_g)_4(t,x), v(t,x)\ra\big|_{x=0}^{x=1}\, dt \\
{}& +\int_0^1 \la (\mc{L}_g)_5(t,x), v(t,x)\ra\big|_{t=0}^{t=T}\, dx
\]
which after recalling that $v(0,\cdot)=v(T,\cdot)=0$ becomes
\begin{align}\label{eq:diff_of_S:2}
\mc{D} S(g)(v)=-\int_{[0,T]\times [0,1]} \la\mc{E}_g(t,x), v(t,x)\ra\, d(t,x)
+\int_0^T \la (\mc{L}_g)_4(t,x),v(t,x)\ra\big|_{x=0}^{x=1}\, dt.
\end{align}

Since by Assumption \ref{Ass : force} the force $F$
is of the form \eqref{eq:Ass : force:1},
the first term $T_F(g)$ on the right-hand side of \eqref{eq:Hamilton equation} is
\begin{align}\label{eq:T_F(g):1}
T_F(g)={}& -\int_{[0,T]\times [0,1]} \la \mc{F}(t,x,j_g(t,x)), v(t,x)\ra d(t,x)
-\int_0^T \la F(t,x,j_g(t,x)), v(t,x)\ra\big|_{x=0}^{x=1} dt
\end{align}
with
\[
F(t,x,g,\eta,\xi):=\begin{cases}
-F_0(t,0,g,\xi,\eta) \\
\phantom{+}F_1(t,1,g,\xi,\eta)
\end{cases}
\]

At last, recalling Assumption \ref{Ass : constraint} on the constraint map $A$
and Assumption \ref{ass:injecting_lambda:1} on the Lagrange multiplier $\lambda$,
the last term $T_C(g)$ on the right-hand side of
\eqref{eq:Hamilton equation} has the form
\begin{align}\label{eq:T_C(g):1}
T_C(g)
={}& \int_0^T \la \lambda(t,\cdot), A(t,g(t,\cdot))v(t,\cdot)\ra_{(\mathscr{C}^1)^*,\mathscr{C}^1} dt \\
={}& \int_{[0,T]\times [0,1]} \big(\la \lambda(t,x),\mc{C}^1(t,x,j_g(t,x))v(t,x)+\mc{C}^2(t,x,j_g(t,x))\partial_x v(t,x)\ra\big) d(t,x) \nonumber \\
={}& \int_{[0,T]\times [0,1]} \la\mc{C}^1(t,x,j_g(t,x))^*\lambda(t,x))-\partial_x\big(\mc{C}^2(t,x,j_g(t,x))^*\lambda(t,x)\big),v(t,x)\ra d(t,x) \nonumber \\
{}& +\int_0^T \la \mc{C}^2(t,x,j_g(t,x))^*\lambda(t,x),v(t,x)\ra\big|_{x=0}^{x=1}\ dt.
\nonumber
\end{align}

We are ready to complete the proof.
Suppose that $g\in C^{2,2}_G$ satisfies
the weak form of Cosserat-Poincaré equation \eqref{eq:Hamilton equation}.
Writing
\[
\mc{P}_g:=\mc{E}_g-\mc{F}+(\mc{C}^1)^*\lambda-\partial_x\big((\mc{C}^2)^*\lambda\big).
\]
and using \eqref{eq:diff_of_S:2}, \eqref{eq:T_F(g):1} and \eqref{eq:T_C(g):1}, we find that
\begin{multline*}
-\int_{[0,T]\times [0,1]} \la\mc{P}_g(t,x,j_g(t,x)),v(t,x)\ra d(t,x)
+\int_0^T \la (\mc{L}_g)_4(t,x),v(t,x)\ra\big|_{x=0}^{x=1} dt \\
=\int_0^T \Big(-\la F(t,x,j_g(t,x)), v(t,x)\ra\big|_{x=0}^{x=1} + \la \mc{C}^2(t,x,j_g(t,x)),v(t,x)\ra\big|_{x=0}^{x=1}\Big) dt
\end{multline*}
must hold for all $v\in \mathscr{C}^2_0([0,T],C^2_{\mathfrak{g}})$.
Therefore
\[
\mc{P}_g(t,x,j_g(t,x))=0
\]
for all $(t,x)\in [0,T]\times [0,1]$
which gives \eqref{equation Cosserat dyn 1} (recall \eqref{eq:mc{E}_g:1})
and
\[
{}& (\mc{L}_g)_4(t,0)=F_0(t,0,j_g(t,0))+\mc{C}^2(t,0,j_g(t,0))^*\lambda(t,x) \\
{}& (\mc{L}_g)_4(t,1)=-F_1(t,1,j_g(t,1))+\mc{C}^2(t,1,j_g(t,1))^*\lambda(t,x)
\]
for all $t\in [0,T]$
which is just \eqref{equation Cosserat boun 1}.

\subsection{Vakonomic approach}\label{Annex : Vakonomic method}
The weak form of the Cosserat-Poincaré equation~\eqref{eq:Hamilton equation} derived from the Lagrange-d'Alembert principle can handle both holonomic and non-holonomic constraints. However, for \emph{holonomic constraints}, it is possible to recover equation~\eqref{eq:Hamilton equation} using the \emph{vakonomic approach}. More precisely, taking the notations introduced in Definition~\ref{def:Weak form} and considering the function $\Phi : [0,T]\times C^2_G \rightarrow E $
which is the integrated form of the holonomic constraint $(A,a)$ defined in equation~\eqref{eq:holonomic},
equation~\eqref{eq:Hamilton equation} is equivalent to the variational problem
\begin{equation} \label{def:vakonomic approach}
\mathcal{D}\tilde{S}(g,\lambda)(v)=-\int_0^T \langle F\big(t,g(t,\cdot),\eta_g(t,\cdot)\big),v(t,\cdot)\rangle\, dt \quad \forall v\in C^1_0([0,T];C^1_\mathfrak{g}) ,
\end{equation}
where $F$ is a force and $\tilde{S}: [0,T]\times C_G^2 \times C^1_{E^*}\rightarrow \mathbb{R}$ is the augmented action defined as
\begin{equation}
    \tilde{S}(t,g,\lambda)= \int_0^T L(t,g(t,\cdot),\eta_g(t,\cdot)) + \langle \lambda, \Phi(t,g(t,\cdot))\rangle dt.
\end{equation}
Note that, in the literature, the function $\Phi$ is often referred to as the constraint, meaning that a movement $g$ satisfies the constraint $\Phi$ if
$$
\Phi(t,g(t,\cdot))=0\quad \forall t\in[0,T].
$$

However, this definition of constraint is not suitable for non-holonomic constraints depending on the time  twist
$\eta_g$, since the corresponding vakonomic equations of motion generally differ from those obtained from the Lagrange-d'Alembert principle and are therefore not consistent with the mechanical equations of motion; see \cite{lewis1995variational} and Appendix~\ref{Annex : Vakonomic method} for an example.

We next provide an exemple where the vakonomic approach and the Lagrange-d'Alembert principle generally yield different equations for non-holonomic constraints. Consider the normal velocity constraint \begin{equation}\label{def:cons velocity norm}
     \langle \xi_g(t,x), \eta_g(t,x)\rangle=0 \quad \forall (t,x)\in [0,T]\times[0,1].\end{equation}
Under this constraint, the  Lagrange-d'Alembert principle yields the strong form of the Cosserat equation~\eqref{Cosserat Lagrangian 1} on $[0,T]\times [0,1]$,
namely
\begin{empheq}[left=\empheqlbrace]{align}\label{ex:vakonomic_method:cosserat_equation}
&\partial_t(I\eta) - \partial_x\Lambda - \operatorname{ad}_{\eta}^*(I\eta) + \operatorname{ad}_{\xi}^*\Lambda+\lambda\xi = \mathcal{F}(t,x, \xi,\partial_x\xi, \eta), \\
&\langle \xi, \eta\rangle=0, \nonumber \\
&\partial_x \eta - \partial_t \xi = [\eta, \xi], \nonumber \\
&\Lambda(t, 0) = -F_0(t, \xi(t, 0), \eta(t, 0)), \nonumber \\
&\Lambda(t, 1) = F_1(t, \xi(t, 1), \eta(t, 1)), \nonumber \\
&(\eta(0, x), \xi(0, x)) = (g^0(x)^{-1}g^1(x), g^0( x)^{-1}\partial_x g^0(x)), \nonumber
\end{empheq}
where  $\Lambda := H(\xi - \xi_0)$ is defined in equation \eqref{def:stress} and $\lambda : [0,T]\times[0,1] \rightarrow \mathbb{R}$ is the Lagrange multipliers induced by the constraints~\eqref{def:cons velocity norm}.

On the other hand, by taking $\Phi(t,g,\eta)= \langle \xi_g(t,x), \eta_g(t,x)\rangle$, the variational problem deduced from the vakonomic approach~\eqref{def:vakonomic approach} yields the system of equations on $[0,T]\times [0,1]$
\begin{empheq}[left=\empheqlbrace]{align}\label{ex:vakonomic equation}
&\partial_t(I\eta) - \partial_x\Lambda - \operatorname{ad}_{\eta}^*(I\eta) + \operatorname{ad}_{\xi}^*\Lambda+\lambda(-\operatorname{ad}_{\xi}^*\eta-\operatorname{ad}_{\eta}^*\xi + \partial_x \eta + \partial_t \xi)= \mathcal{F}(t,x, \xi,\partial_x\xi, \eta), \\
&\langle \xi, \eta\rangle=0, \nonumber \\
&\partial_x \eta - \partial_t \xi = [\eta, \xi], \nonumber \\
&\Lambda(t, 0) = -F_0(t, \xi(t, 0), \eta(t, 0))+\eta(t,0)\lambda(t,0), \nonumber \\
&\Lambda(t, 1) = F_1(t, \xi(t, 1), \eta(t, 1))+\eta(t,1)\lambda(t,1), \nonumber \\
&(\eta(0, x), \xi(0, x)) = (g^0(x)^{-1}g^1(x), g^0( x)^{-1}\partial_x g^0(x)). \nonumber
\end{empheq}
Clealy equations \eqref{ex:vakonomic equation} and \eqref{ex:vakonomic_method:cosserat_equation} are different.

\subsection{Left-invariant Cosserat-Poincar\'e Equations on $\SE(2)$ and $\SE(3)$ with no force and with constraints expressed in coordinates}\label{app:left_invariant_CP:1}

\subsubsection{The $\SE(2)$ case}\label{ex:SE(2) in coordinate}

This case describes the motion of snake-like objects on a plane. Note that $\SO(2)$ is commutative whereas $\SO(3)$ is not. Hence, this case is an interesting first step to understand the locomotion of beam.
Using the map defined in \eqref{def : Lie algebra isomorphism}, the Lie algebra
$\mathfrak{g}=\mathfrak{se}(2)$ can be identified with $\mathbb{R}^3$ and its dual space $\mathfrak{se}(2)^*$ with $\mathfrak{se}(2)$ using the canonical inner product $\la\cdot,\cdot\ra_{\R^3}$ on $\R^3$.
In this case, for $\zeta = \begin{pmatrix} \Omega \\ V \end{pmatrix}\in \mathbb{R}^3$ with $\Omega\in \mathbb{R}$ and $V\in \mathbb{R}^2$, it is easy to check that
\[ \operatorname{ad}_{\zeta}= \begin{pmatrix}
    0 & 0 \\ -\mathrm{Q} V & \Omega\mathrm{Q}
\end{pmatrix},
\quad  \operatorname{ad}_{\zeta}^*=\ad^T_{\zeta}= \begin{pmatrix}
    0 & V^T\mathrm{Q} \\ 0 & -\Omega\mathrm{Q},
\end{pmatrix}
\]
with the transpose $\ad^T_{\zeta}$
defined using $\la\cdot,\cdot\ra_{\R^3}$.

Recall the notations $\eta = \begin{pmatrix}
    \Omega \\ V
\end{pmatrix}$, $\xi = \begin{pmatrix} K \\\Gamma \end{pmatrix}$, $\xi_0 = \begin{pmatrix} K_0 \\\Gamma_0 \end{pmatrix}$ introduced in \eqref{def:twistSE(k)}
and suppose that the functions $I$ and $H$
(see Section \ref{left-invariant Lagrangian section})
take values in
$\mathscr{L}(\mathbb{R}^3)$ and are of the form
\[
    I=\begin{pmatrix}
        j_1 & 0_{1\times 2} \\0_{2\times 1} & J_2
    \end{pmatrix}, \quad H=\begin{pmatrix}
        h_1 & 0_{1\times 2} \\0_{2\times 1}  & H_2
    \end{pmatrix}.
\]
Let us be given initial conditions $(\Omega^0,V^0,K^0,\Gamma^0)$
and a left-invariant distributed force $\mathcal{F}=\begin{pmatrix}
    \mathcal{F}_\Omega\\ \mathcal{F}_V
\end{pmatrix}$ with
$\mathcal{F}_\Omega$ and $\mathcal{F}_V$ being functions defined on $[0,T]\times \mathfrak{se}(2)^2$
and taking values in $\mathbb{R}$ and $\mathbb{R}^2$, respectively.
Then the Cosserat-Poincar\'e equation \eqref{Cosserat Lagrangian 1} can be written as the following system of partial differential equations on $[0,T]\times [0,1]$ for the  unknowns $(\Omega,V,K,\Gamma)$:
\begin{subequations}
\begin{empheq}[left=\empheqlbrace]{align}
 &j_1\partial_t\Omega - \partial_x\big(h_1(K-K_0)\big) - V^T\mathrm{Q} J_2 V + \Gamma^T \mathrm{Q} H_2(\Gamma - \Gamma_0)=\mathcal{F}_\Omega, \\
 &J_2\partial_t V - \partial_x\big(H_2(\Gamma-\Gamma_0)\big) - \Omega \mathrm{Q} J_2 V + K \mathrm{Q} H_2 (\Gamma - \Gamma_0 ) = \mathcal{F}_V, \nonumber \\
   &\partial_x\Omega - \partial_tK = 0 \nonumber \\
    &\partial_x V - \partial_t \Gamma = \Omega \mathrm{Q} \Gamma - K \mathrm{Q} V, \nonumber \\
    &(K,\Gamma)_{|x=0}= (K_0,\Gamma_0)(0), \nonumber \\
    &(K,\Gamma)_{|x=1}= (K_0,\Gamma_0)(1), \nonumber \\
    &(\Omega,V,K,\Gamma)_{|t=0}=(\Omega^0,V^0,K^0,\Gamma^0),
    \nonumber
    \end{empheq}
\end{subequations}
where $(\Omega^0,V^0,K^0,\Gamma^0)$ are the initial conditions and  $(K_0,\Gamma_0)$ are the space twists at rest.

\subsubsection{The $\SE(3)$ case}\label{ex:SE(3) in coordinate}

This particular case has numerous applications, such as surgical robots \cite{7314984} and the study of the dynamics of bacterial flagella \cite{ko2017modeling}. Using the map defined in \eqref{def : Lie algebra isomorphism}, the Lie algebra $\mathfrak{g}=\mathfrak{se}(3)$ can be identified with $\mathbb{R}^6$ and its dual space $\mathfrak{se}(3)^*$ with $\mathfrak{se}(3)$ using the canonical inner product $\la\cdot,\cdot\ra_{\R^6}$ on $\mathbb{R}^6$.

In this case, for $\zeta = \begin{pmatrix}
    \Omega \\ V
\end{pmatrix}$ with $(\Omega,V) \in (\mathbb{R}^3)^2$, we have
\[
    \operatorname{ad}_\zeta =\begin{pmatrix}
        \widehat{\Omega} & 0 \\ \widehat{V} & \widehat{\Omega}
    \end{pmatrix},
    \quad
    \operatorname{ad}_\zeta^* = \operatorname{ad}_\zeta^T  =\begin{pmatrix}
        -\widehat{\Omega} & -\widehat{V} \\ 0 & -\widehat{\Omega}
    \end{pmatrix}
\]
with the transpose $\ad^T_{\zeta}$
defined using $\la\cdot,\cdot\ra_{\R^6}$.

We use the notations
$\eta = \begin{pmatrix} \Omega \\ V \end{pmatrix}$, $ \xi = \begin{pmatrix}
    K \\ \Gamma \end{pmatrix}$ and $\xi_0=\begin{pmatrix} K_0 \\ \Gamma_0\end{pmatrix}$
introduced in \eqref{def:twistSE(k)} and suppose that the functions $I,H:[0,1]\to\mathscr{L}(\mathbb{R}^6)$ (see Section \ref{left-invariant Lagrangian section})
of $x\in [0,1]$
    are of the form :
    \begin{equation}
        I = \begin{pmatrix}
            J_1 & 0_{3\times 3} \\ 0_{3\times 3} & J_2
        \end{pmatrix}, \quad H=\begin{pmatrix}
            H_1 &  0_{3\times 3} \\  0_{3\times 3} & H_2
        \end{pmatrix}.
    \end{equation}
Let us be given initial conditions $(\Omega^0,V^0,K^0,\Gamma^0)$ and a left-invariant distributed force
$\mathcal{F}=\begin{pmatrix}
    \mathcal{F}_\Omega & \mathcal{F}_v
\end{pmatrix}$
with $\mathcal{F}_\Omega$ and $\mathcal{F}_V$ being functions defined on $[0,T]\times \mathfrak{se}(2)^2$ and taking values in $\mathbb{R}^3$.
Then the Cosserat-Poincar\'e equation~\eqref{Cosserat Lagrangian 1} becomes the following system of partial differential equations on $[0,T]\times [0,1]$ for the unknowns $(\Omega,V,K,\Gamma)$:
\begin{subequations}
\label{matrix equation SE(3)}
\begin{empheq}[left=\empheqlbrace]{align}
\label{Equation SE(3) vectoriel 1} &J_1\partial_t\Omega - \partial_x \big(H_1 (K - K_0)\big) + \widehat{\Omega} J_1\Omega + \widehat{V} J_2V - \widehat{K} H_1 (K - K_0) - \Gamma H_2 (\Gamma - \Gamma_0)=\mathcal{F}_
\Omega, \\
&J_2\partial_t V - \partial_x\big(H_2(\Gamma-\Gamma_0)\big) + \widehat{\Omega} J_2V - \widehat{K} H_2(\Gamma - \Gamma_0)=\mathcal{F}_V ,\nonumber \\
&\partial_x \Omega - \partial_t K = \widehat{\Omega} K, \nonumber \\
&\partial_x V - \partial_t \Gamma = \widehat{V} K + \widehat{\Omega} \Gamma, \nonumber \\
&(K,\Gamma)_{|x=0}= (K_0,\Gamma_0)(0), \nonumber \\
&(K,\Gamma)_{|x=1}= (K_0,\Gamma_0)(1), \nonumber \\
&(\Omega,V,K,\Gamma)_{|t=0}=(\Omega^0,V^0,K^0,\Gamma^0), \nonumber
\end{empheq}
where $(\Omega^0,V^0,K^0,\Gamma^0)$ are the initial conditions and  $(K_0,\Gamma_0)$ are the space twist at rest.

\end{subequations}
One can also write the above equations in $\mathfrak{so}(3)$ instead, by
using the hat-operator \eqref{eq:hat_operator:1}.
Recall that
\[
   \widehat{\widehat{Y} X} =\widehat{Y \times X}= [\widehat{X},\widehat{Y}]_{\mathfrak{so}(3)} \quad \forall (X,Y)\in\mathbb{R}^3,
   \]
where $\times$ is the usual cross product and $[\cdot,\cdot]_{\mathfrak{so}(3)}$ is the Lie bracket of $\mathfrak{so}(3)$.
For instance, applying the hat-operator to the first equation in
\eqref{Equation SE(3) vectoriel 1}
re-expresses it in $\mathfrak{so}(3)$ as
\begin{equation}
\begin{aligned}
    \widehat{\mathcal{F}}_\Omega=& \widetilde{J}_1\partial_t\widehat{\Omega} - \partial_x \big(\widetilde{H}_1 (\widehat{K} - \widehat{\Omega}_0)\big) + [\widehat{\Omega}, \widetilde{J}_1\widehat{\Omega}]_{\mathfrak{so}(3)} + [\widehat{V}, \widetilde{J}_2\widehat{V}]_{\mathfrak{so}(3)} \\
    &- [\widehat{K}, \widetilde{H}_1 (\widehat{K} - \widehat{\Omega}_0)]_{\mathfrak{so}(3)} - [\widehat{\Gamma}, \widetilde{H_2} (\widehat{\Gamma} - \widehat{V}_0)]_{\mathfrak{so}(3)},
    \end{aligned}
\end{equation} where the linear applications $\widetilde{H}_i$ and $\widetilde{J}_i$ are defined for $i\in \{1,2\}$ by
\begin{equation}
    \widetilde{H_i}\widehat{X}=\widehat{H_iX}, \quad \widetilde{J_i}\widehat{X}=\widehat{J_iX} \quad \forall X\in \mathbb{R}^3.
\end{equation}

\subsubsection{Kirchhoff beam (continued)} \label{Kirchhoff beam Appendix}

In order to express Equation~\eqref{Cosserat Lagrangian example 2} in coordinates, we assume that $I=\operatorname{diag}(J_1,J_2)$ and $H=\operatorname{diag}(H_1,H_2)$, where $I_1,I_2,H_1,H_2:[0,1]\to \mathscr{L}(\R^3)$ are symmetric positive definite matrix valued functions of $x\in [0,1]$.
The constraint \eqref{constraint Lagrangian example 2} can be written as
\begin{equation}
\partial_t \Gamma(t,x) = 0\quad \forall (t,x)\in[0,T]\times[0,1],
\end{equation}
where $\Gamma$ has been introduced in \eqref{def:twistSE(k)}.
In other words $\Gamma(t,x) = \Gamma(0,x)$ for all $(t,x) \in [0,T] \times [0,1]$.
Assuming that the linear space rate at $t=0$ is at rest, i.e. $\Gamma(0,x)=\Gamma_0(x)$,
we have $\Gamma(t,x) = \Gamma_0(x) $ for all $(t,x) \in [0,T] \times [0,1]$, and \eqref{Cosserat Lagrangian example 2} can be reduced to the following system of partial differential equations on $[0,T]\times [0,1]$ in the unknowns $(\Omega, V, K, \lambda)$:
\begin{subequations}
    \label{Cosserat Lagrangian reduced example 2}
    \begin{empheq}[left=\empheqlbrace]{align}
    &\nonumber J_1 \partial_t \Omega - \partial_x \big(H_1 (K - K_0)\big) + \widehat{\Omega} J_1 \Omega + \widehat{V} J_2 V - \widehat{K} H_1 (K - K_0) - \widehat{V}_0 \lambda = 0,  \\
    &\nonumber J_2 \partial_t V - \partial_x \lambda + \widehat{\Omega} J_2 V - \widehat{K} \lambda = 0, \\
    &\nonumber \partial_x \Omega - \partial_t K = \widehat{\Omega}_\eta K,  \\
    &\nonumber \partial_x V = \widehat{V} K + \widehat{\Omega} \Gamma_0,  \\
    &\nonumber(K,\lambda)_{|x=0}= (K_0,0)(0), \\
    &\nonumber(K,\lambda)_{|x=1}= (K_0,0)(1) \\
    \nonumber&(\Omega,V,K)_{|t=0}=(\Omega^0,V^0,K^0),
    \end{empheq}
\end{subequations}
where $\lambda : [0,T]\times[0,1] \rightarrow \mathbb{R}^3$ is the Lagrange multiplier associated with the Kirchhoff beam constraint, $(\Omega^0,V^0,K^0)$ are the initial conditions
and $(\Omega_0,V_0)\simeq \eta_0$, $(K_0,\Gamma_0)\simeq \xi_0$ are the time and space twists at rest, respectively.

\subsection{Reference trajectory and shape}

Let $g_r \in C^{2,2}_G$ be a reference trajectory. We first determine a distributed force $\mathcal{F}_r:[0,T]\times[0,1]\rightarrow \mathfrak{g}^*$ such that the solution of the Cosserat-Poincaré equation matches the reference trajectory $g_r$. We then show how distributed forces and (induced) internal forces can also be interpreted as reaction forces coming from reference trajectories and reference shapes.

\subsubsection{Reference trajectory}

Let us be given initial conditions $\eta^0 \in C^1_{\mathfrak{g}}$ and $\xi^0 \in C^2_{\mathfrak{g}}$
and a force $\mathcal{F}_r\in C^{1,1}_\mathfrak{g}$ depending only on $t,x$ (hence left-invariant).
Denote by $(\eta,\xi)$ the twists that are the solution
of the left-invariant the Cosserat-Poincaré equation
\eqref{Cosserat Lagrangian 1}
on $[0,T]\times [0,1]$ with the force $\mathcal{F}_r$ and no constraint
and let $g$ be an associated movement of the Cosserat beam
i.e. $(\eta_g,\xi_g)=(\eta,\xi)$.

Given a reference trajectory $g_r$, we seek a force $\mathcal{F}_r(t,x)$
enforcing exact
trajectory matching at all times, i.e., the movement $g$ of the Cosserat beam should satisfy
\begin{equation}
    g(t,x) = g_r(t,x) \quad \forall (t,x)\in[0,T]\times [0,1].
\end{equation}
To this end, assume $g_r\in C^{2,2}_G$ and define the reference trajectory constraint \eqref{eq:prescribed constraint}
\begin{equation}
    \label{def:reference trajectory}
    \mathcal{C}(t,x,g,\xi,\eta,\partial_x\eta) = \eta_{g_r}(t,x) - \eta,
\end{equation}
where $\eta_{g_r}(t,x) = g_r^{-1} \partial_t g_r(t,x)$.
Assume that the initial configuration of the motion $g$ satisfies $g(0,\cdot)=g_r(0,\cdot)$.
The constraint~\eqref{def:reference trajectory} imposes $\eta=\eta_{g_r}$, which then together with the structure equation \eqref{eq:Cosserat Lagrangian 1 struc}
and the initial condition $g(0,\cdot)=g_r(0,\cdot)$
(hence $\xi(0,\cdot)=\xi_{g_r}(0,\cdot)$)
yields $\xi=\xi_{g_r}$.
Therefore, the Lagrange multiplier associated with the constraint~\eqref{def:reference trajectory} is obtained from
the Cosserat--Poincaré equation~\eqref{Cosserat Lagrangian 1} on $[0,T]\times[0,1]$ with no external force
\begin{subequations}
\label{eq:prescribed_systemwithforce}
    \begin{empheq}[left=\empheqlbrace]{align}
    &\partial_t(I\eta) - \partial_x(H(\xi - \xi_0)) - \operatorname{ad}_{\eta}^*(I\eta) + \operatorname{ad}_{\xi}^*(H(\xi - \xi_0)) = - \lambda, \\
    &\partial_t \xi - \partial_x \eta = [\xi,\eta], \\
    &\eta=\eta_{g_r}, \\
    &\xi(t,0)=\xi_0(0), \quad \xi(t,1)=\xi_0(1), \\
    &(\eta(0,x),\xi(0,x))=(\eta^0(x),\xi^0(x)),
    \end{empheq}
\end{subequations}
as the function $\lambda:[0,T]\times [0,1]\to\mathfrak{g}$ is defined by
\begin{equation}\label{eq:force for reference trajectory}
     \lambda:=-\partial_t(I\eta_{g_r}) + \partial_x(H(\xi_{g_r} - \xi_0)) + \operatorname{ad}_{\eta_{g_r}}^*(I\eta_{g_r}) - \operatorname{ad}_{\xi_{g_r}}^*(H(\xi_{g_r} - \xi_0)).
\end{equation}
Taking the force $\mathcal{F}_r$ as the opposite of the reaction force induced by the constraint~\eqref{def:reference trajectory},
that is $\mathcal{F}_r(t,x):=-\lambda(t,x)$, we deduce that the solution $(\eta,\xi)$ of equation~\eqref{eq:prescribed_systemwithforce} satisfies $(\eta,\xi)=(\eta_{g_r},\xi_{g_r})$.
Finally, since also $g(0,\cdot)=g_r(0,\cdot)$,
the uniqueness of the reconstruction of $g$ from $(\xi,\eta)$
(see Remark \ref{re:reconstruction:1} and the references therein),
allows us to conclude that $g=g_r$,
that is, the movement $g$ of the Cosserat beam matches at all times the reference trajectory $g_r$.

\subsubsection{Forces as reference trajectories and reference shapes}


\begin{lemma}\label{lemma: equivalence inter react forces} Let $T>0$, $(\eta^0,\xi^0)\in C^2_\mathfrak{g}$. In equation~\eqref{Cosserat Lagrangian 0.5}:
\begin{itemize}
    \item[(i)] Any distributed force $\mathcal{F}\in C^{1,1}_{\mathfrak{g}}$ depending only on $t,x$ can be represented as a reaction force (see Remark~\ref{rem:reaction force}) associated to the reference \emph{trajectory} constraint~\eqref{eq:prescribed constraint}
    with appropriately chosen $\eta_r$.

    \item[(ii)] Any internal force $F_{\mathrm{int}}$ induced by a function $\Lambda_{\mathrm{int}}\in \mathscr{C}^2([0,T]\times[0,1]\times\mathfrak{g};\mathfrak{g})$
    (see \eqref{def:internal force dens})
    can be represented as a reaction force (see Remark~\ref{rem:reaction force}) associated to the reference \emph{shape} constraint~\eqref{eq:reference shape}
    with appropriately chosen $\xi_r$.
\end{itemize}
\end{lemma}

\begin{proof}
     Let us first prove (i).
     Let $\mathcal{F}\in C^{1,1}_{\mathfrak{g}}$. Given initial conditions $(\eta^0,\xi^0)\in C^2_\mathfrak{g}$, let us denote $(\eta_r,\xi_r)$ the solution of the Cosserat-Poincaré equation~\eqref{Cosserat Lagrangian 1} with the force $\mathcal{F}$ and no constraints.
Then it is immediate that $(\eta_r,\xi_r,\lambda)$ with $\lambda := -\mathcal{F}$ is a solution to the Cosserat-Poincaré equation~\eqref{Cosserat Lagrangian 1} on $[0,T]\times[0,1]$ subjected to no force and to the constraint~\eqref{eq:prescribed constraint}
i.e.
\begin{subequations}
\label{eq:reference trajectory}
    \begin{empheq}[left=\empheqlbrace]{align}
    &\partial_t(I\eta) - \partial_x\Lambda - \operatorname{ad}_{\eta}^*(I\eta) + \operatorname{ad}_{\xi}^*\Lambda = -\lambda,  \\
    &\eta=\eta_{g_r}, \\
    & \partial_x \eta - \partial_t\xi = [\eta,\xi], \\
    &\Lambda(t,0)=0, \quad \Lambda(t,1)=0, \\
    &(\eta(0,\cdot),\xi(0,\cdot))=(\eta^0,\xi^0).
    \end{empheq}
\end{subequations}
Hence the distributed force $\mathcal{F}=-A^*\lambda=-\lambda$ is a reaction force associated to the reference trajectory constraint as claimed.

To prove (ii), we apply the same method. Let $(\eta^0,\xi^0)\in C^2_G$ and $F_{\mathrm{int}}$ be an internal force induced by $\Lambda_{\mathrm{int}}\in \mathscr{C}^2([0,T]\times[0,1]\times \mathfrak{g};\mathfrak{g})$. We denote $(\eta_r,\xi_r)$ the solution of the Cosserat-Poincaré equation~\eqref{Cosserat Lagrangian 1} with no constraints and subjected to the (internal) force $F_{\mathrm{int}}$.


Let us recall the reference shape constraint~\eqref{eq:reference shape}
\begin{equation}
\label{eq:reference shape2}
    \mathcal{C}(t,\xi,\eta,\partial_x\eta) = \operatorname{ad}_{\xi} \eta+\partial_x \eta -\partial_t\xi_r
\end{equation}
associated to the shape $\xi_r$.
Then $(\eta_r,\xi_r,\lambda)$ with $\lambda(t,x):=\Lambda_{\mathrm{int}}(t,x,\xi_r(t,x))$
is a solution to the Cosserat-Poincaré equation on $[0,T]\times [0,1]$
subjected to no force and with constraint~\eqref{eq:reference shape2} i.e.
\begin{subequations}\label{eq:reference shape equivalence 1}
    \begin{empheq}[left=\empheqlbrace]{align}
    &\partial_t(I\eta)-\partial_x\Lambda-\operatorname{ad}_{\eta}^T(I\eta)+\operatorname{ad}_{\xi}^T\Lambda-\partial_x \lambda + \operatorname{ad}_\xi^T \lambda=0, \\
    &\partial_x \eta - \partial_t\xi = [\eta,\xi],\\
    & \operatorname{ad}_{\xi} \eta+\partial_x \eta-\partial_t\xi_r  = 0, \\
    &\Lambda(t,0)=-\lambda(t,0),\\
    &\Lambda(t,1)= -\lambda(t,1),\\
    &(\eta(0,\cdot),\xi(0,\cdot))=(\eta^0,\xi^0).
    \end{empheq}
    \end{subequations}
Hence the internal distributed force $F_{\mathrm{int}}=A^*\lambda$ is a reaction force associated to the reference shape constraint \eqref{eq:reference shape} $(A,a)$
as claimed.

\end{proof}

\subsection{Rigid body}\label{sec:rigid-body}

We start by deriving a constraint on the time and space twists expressing the fact that the body (beam, rod) is rigid.
We then show how the Euler-Poincaré equation describing the evolution of a rigid body can be recovered from the Cosserat-Poincaré equation associated with the rigid body constraint previously obtained.
\subsubsection[Lemma Rigid body]{On the rigid body constraint}
We have the following lemma.
\begin{lemma}[Rigid body constraint]\label{lemma :Equivalent rigid body}
Let us take a movement $g\in C^{2,2}_{\SE(k)}$ representing a beam such that \begin{equation}\label{eq : Condition minimale rigid body}
    \partial_t\xi_g=0 \quad \forall (t,x)\in[0,T]\times[0,1].
\end{equation}
Consider two point functions $x_1,x_2 : [0,T]\rightarrow \mathbb{R}^k$ of the beam represented by $g$, as explained in Remark~\ref{rem:mechanics}. Then $t\mapsto \|x_1(t) - x_2(t)\|$
is a constant function.
\end{lemma}
\begin{proof} In this proof, we use, without recalling, the notation of Remark~\eqref{rem:mechanics}.
 Let us take two particles $X=(X^1,\dots,X^T),Y=(Y^1,\dots,Y^k)\in \mathbb{R}^k$. Their position in the configuration $\varphi : [0,T]\times \mathbb{R}^k \rightarrow \mathbb{R}^k$ of the beam is given by, for all time $t\in[0,T]$, by
\[\begin{aligned}
\varphi(t,X)=g(t,X^1)\cdot X^\perp,
\end{aligned}\]
where $X^\perp = X^2E_2 + .. + X^k E_k$ and $g(t,X^1)\cdot X^{\perp}=p(t,X^1)+R(t,X^1)X^\perp$, where $p(t,X^1)\in \R^k$ is the position of the cross section $X^1$ and $R(t,X^1)\in \SO(k)$ is the orientation of the cross section. Analogically, using the same notation, we have
\[\begin{aligned}
\varphi(t,Y)=g(t,Y)\cdot Y^\perp.
\end{aligned}\]

According to \eqref{eq : Condition minimale rigid body}, the space twist $\xi_g : [0,1] \rightarrow \mathfrak{se}(k)$ of $g$ only depends on the sections $x$ of the movement,
i.e. $\xi_g(t,x)=\xi_g(x)$, and hence the movement $g$ satisfies the equation $\partial_x g(t,x)=g(t,x)\xi_g(x)$ for every $(t,x)\in[0,T]\times[0,1]$. This can be integrated to get
\begin{equation}\label{eq:decomposition g}
g(t,x)=g(t,0)S(x),
\end{equation}
where $S : [0,1] \rightarrow \SE(k)$ is the resolvent matrix associated with $\xi_g$, i.e. $S$ is the solution of the Cauchy problem given by
\begin{equation*}
    \begin{cases}
        S'(x)=S(x)\xi_g(x) \quad \forall x\in [0,1],\\
        S(0)=\id
    \end{cases}
\end{equation*}
Using \eqref{eq:decomposition g}, for every time $t\in[0,T]$, we have
\begin{equation*}
    \begin{aligned}
        \lVert \varphi(t,X) - \varphi(t,Y) \rVert &=\lVert g(t,X^1)\cdot X^\perp -g(t,Y^1)\cdot Y^\perp\rVert \\
        &= \lVert g(t,0)\cdot \big(S(X^1)\cdot X^\perp\big) - g(t,0)\cdot\big( S(Y^1)\cdot Y^\perp \big) \rVert
        \end{aligned}
\end{equation*}
which, after noticing that $g(t,0)\in \SE(k)$, yields
\[
\lVert \varphi(t,X) - \varphi(t,Y) \rVert  =\lVert S(X^1)\cdot X^\perp - S(Y^1)\cdot Y^\perp \rVert,
\]
where the right-hand side does not depend on time $t$.
This completes the proof.
\end{proof}

\begin{remark} In the proof, we actually show that setting $A(t):=g(t,0)g(0,0)^{-1}\in\SE(k)$ for $t\geq 0$, it holds
\[
\Omega_t
=A(t)\Omega_0,
\]
where $\Omega_t$ was defined in Remark \ref{rem:mechanics}, i.e., $\{\Omega_t\}_{t\geq 0}$ is a rigid body.
\end{remark}

\subsubsection{Derivation of the Cosserat-Poincaré equation under rigid body constraint.}

We start from Corollary~\ref{cor:CPE-left-invariant} with the rigid body constraint \eqref{eq:rigid_body_constraint_time} and no force, i.e. $\mathcal{F}=0$, $\mathcal{C}^1=\operatorname{ad}_\xi,$ $\mathcal{C}^2=I_{\mathrm{d}}$ and $\widetilde{a}=0$.
Given initial conditions $\eta^0 \in C^2_{\mathfrak{se}(k)}$, $\xi^0 \in  C^1_{\mathfrak{se}(k)}$ satisfying the rigid body constraint~\eqref{eq:rigid_body_constraint_time}, the dynamics of the movement is determined by the following system of partial differential equations for the unknowns $(\eta,\xi,\lambda)$ given by, on $[0,T]\times [0,1]$
\begin{subequations}
    \label{Cosserat Lagrangian example 4}
    \begin{empheq}[left=\empheqlbrace]{align}
    &I\partial_t\eta - \partial_x\big(H(\xi - \xi_0)\big) - \operatorname{ad}_{\eta}^*I\eta + \operatorname{ad}_{\xi}^*H(\xi - \xi_0)
    + \operatorname{ad}_{\xi}^*\lambda - \partial_x\lambda = 0, \\
    &\label{constraint Lagrangian example 3}\partial_t \xi =0, \\
    &\partial_x\eta - \partial_t\xi = [\eta, \xi],  \\
    & H(\xi(t,0) -\xi_0(0))=- \lambda(t,0), \quad H(\xi(t,1) -\xi_0(1))=  - \lambda(t,1), \\
    &(\eta(0,x), \xi(0,x)) = (\eta^0(x), \xi^0(x)),
    \end{empheq}
\end{subequations}
where $\eta\in C^{1,2}_{\mathfrak{se}(k)}$ and $\xi\in C^{2,1}_{\mathfrak{se}(k)}$ are, respectively, the time and space  twists and $\lambda:[0,1]\rightarrow \mathfrak{g}^*$ is the reaction force induced by the rigid body constraint \eqref{eq:rigid_body_constraint_time}.
If the space twist initially satisfies $\xi^0(x) = \xi_0(x)$ for all $x \in [0,1]$
(recall that $\xi_0$ is the space twist at rest), using equation~\eqref{constraint Lagrangian example 3} we obtain $\xi(t,x) = \xi_0(x)$ for all $(t,x) \in [0,T] \times [0,1]$.
In this case, \eqref{Cosserat Lagrangian example 4} can be reduced to the following problem on $[0,T]\times [0,1]$ with unknowns $(\eta, \lambda)$
\begin{subequations}
    \label{Cosserat Lagrangian example 4.1}
    \begin{empheq}[left=\empheqlbrace]{align}
    &\label{eq:dynamic Rigid body EDP}I\partial_t\eta - \partial_x\lambda - \operatorname{ad}_{\eta}^*I\eta + \operatorname{ad}_{\xi_0}^*\lambda = 0,  \\
    &\label{eq:structure equation Rigid body EDP}\partial_x\eta = -\ad_{\xi_0}\eta,\\
    &\lambda(t,0) = 0, \quad \lambda(t,1) = 0, \label{eq:rigid_body:lambda_bdry_conditions:1} \\
    &\eta(0,x) = \eta^0(x).
    \end{empheq}
\end{subequations}

Consider, for every $g\in \SE(k)$, the invertible operator
\begin{equation}\begin{aligned}\label{def:operator Ad}\operatorname{Ad}_g :  \mathfrak{se}(k)& \rightarrow \mathfrak{se}(k) \\
\eta &\mapsto g \eta g^{-1}.
\end{aligned}
\end{equation}
It has the following properties:
\begin{itemize}
    \item [(i)]$\partial_x \operatorname{Ad}_h  = \operatorname{Ad}_h\operatorname{ad}_{h^{-1}\partial_x h}$ and $\partial_x \operatorname{Ad}_h  = \operatorname{ad}_{\partial_x h h^{-1}}\operatorname{Ad}_h$ for all $h\in C^{2,2}_{\SE(k)}$ .
    \item [(ii)]$\operatorname{Ad}^*_h\operatorname{ad}^*_{\operatorname{Ad}_h \eta } =\operatorname{ad}^*_\eta \operatorname{Ad}_h^* $ for all $(h,\eta) \in \SE(k)\times \mathfrak{se}(k)$.
\end{itemize}

Let us now take a movement $g\in C^{2,2}_{\SE(k)}$  such that $\eta=\eta_g$ and $\xi=\xi_g$
(such a $g$ exists; see Remark \ref{re:reconstruction:1})
and define the functions
\[
g_b(t):=g(t,0),\quad
h(t,x) := g(t,x)^{-1} g_b(t),\quad \eta_b(t):=\eta_g(t,0)=g_b(t)^{-1}g_b'(t).
\]
Since $\xi=\xi_0$, we deduce the identity $\partial_x h h^{-1} = -\xi_0$ and hence by property (i) above
\begin{equation}
\label{eq:redction rig_body_2}
   \partial_x \operatorname{Ad}_h = -\operatorname{ad}_{\xi_0} \operatorname{Ad}_h.
\end{equation}
Since $h(t,0)=\Id$ and hence $\operatorname{Ad}_{h(t,x)} \eta_b(t)\big|_{x=0}=\operatorname{Ad}_{\Id} \eta_b(t)=\eta_g(t,0)$
for all $t$, it follows that
the unique solution $\eta=\eta_g$ of the ordinary differential equation \eqref{eq:structure equation Rigid body EDP} is given by
\begin{equation}\label{eq:rigid eta selon bound}
\eta_g(t,x) =  \operatorname{Ad}_{h(t,x)} \eta_b(t) \end{equation}
It follows that the function $h=h(t,x)$ does not depend on time $t$ since
\begin{equation}
\begin{aligned}
\partial_t h h^{-1} &= \big( -g^{-1} \partial_t g g^{-1} g_{b} + g^{-1} \partial_t g_{b} \big) \big( g^{-1} g_b\big)^{-1}      \\
&=-g^{-1}\partial_t g + g^{-1} \partial_t g_b (g^{-1} g_b)^{-1}
=-\eta_g + \operatorname{Ad}_{h}\eta_b = 0,
\end{aligned}
\end{equation}
where the last equality follows from \eqref{eq:rigid eta selon bound}.
Therefore $\partial_t h(t,x)=0$ and hence $h(t,x)=h(0,x)$ for all $t,x$.

By multiplying the equation~\eqref{eq:dynamic Rigid body EDP}
from the left by $\operatorname{Ad}_{h}^*$, we obtain that on $[0,T]\times [0,1]$
\[
\operatorname{Ad}_{h}^*I\partial_t\operatorname{Ad}_{h}\eta_b - \operatorname{Ad}_{h}^*\partial_x\lambda- \operatorname{Ad}_{h}^*\operatorname{ad}_{\operatorname{Ad}_{h}\eta_b}^*I\operatorname{Ad}_{h}\eta_b + \operatorname{Ad}_{h}^*\operatorname{ad}_{\xi_0}^*\lambda = 0
\]
which, using the above property (ii) and recalling that $\partial_t h=0$, may be rewritten as
\[
\label{eq:redction rig_body_4}
\operatorname{Ad}_{h}^*I\operatorname{Ad}_{h}\partial_t\eta_b - \operatorname{Ad}_{h}^*\partial_x\lambda - \operatorname{ad}^*_{\eta_b}\operatorname{Ad}_{h}^*I\operatorname{Ad}_{h}\eta_b + \operatorname{Ad}_{h}^*\operatorname{ad}_{\xi_0}^*\lambda = 0.
\]
Integrating this equation w.r.t $x$ and performing an integration by parts on the term $ \operatorname{Ad}_{h}^*\partial_x\lambda$, we obtain
\begin{equation}
\label{eq:redction rig_body_5}
\mathcal{M}\partial_t\eta_b + \int_0^1 (\partial_x \operatorname{Ad}_{h}^*) \lambda dx
-\big[\operatorname{Ad}_{h}^*\lambda \big]_{x=0}^{x=1} - \operatorname{ad}^*_{\eta_b}\mathcal{M}\eta_b
+\int^1_0\operatorname{Ad}_{h}^*\operatorname{ad}_{\xi_0}^*\lambda dx = 0,
\end{equation}
where (recall that $h(t,x)=h(0,x)$ for all $t,x$)
\[
\mathcal{M}:=\int_0^1\operatorname{Ad}_{h(0,x)}^*I(x)\operatorname{Ad}_{h(0,x)} dx
\]
is a constant
symmetric positive definite matrix.
Finally, since $\lambda(t,0)=\lambda(t,1)=0$ by \eqref{eq:rigid_body:lambda_bdry_conditions:1}
and since $\partial_x \operatorname{Ad}^*_{h}=-\operatorname{Ad}_h^* \operatorname{ad}^*_{\xi_0}$ by \eqref{eq:redction rig_body_2},
the terms containing the Lagrange multipliers $\lambda$ cancel out
and we are left with the expected rigid body dynamics (see \cite{marsden1999introduction}-\cite{poincare1901forme})
\begin{equation}\label{eq:app:E-P:rigid_body:1}
\mathcal{M}\partial_t\eta_b -\operatorname{ad}^*_{\eta_b}\mathcal{M}\eta_b = 0
\quad \textrm{on $[0,T]$.}
\end{equation}

\subsection{Follow the leader}\label{Computation Follow the leader}
We first provide a technical result giving an equivalent formulation of the follow-the-leader constraint~\eqref{eq:Follow-The-Leader Constraint}.
We then derive equations that depend solely on the dynamics of the snake's head, starting from the Cosserat-Poincar\'e equation.

\subsubsection[Lemma (Follow-the-leader)]{On the constraints for the Follow the leader model}

\begin{lemma}\label{lemma:Follow-the-leader}
    Let $g\in C^{2,2}_{\SE(2)}$ be movement with  corresponding position $p\in C^{2,2}_{\mathbb{R}^2}$. For every $(t,x)\in [0,T]\times [0,1]$, the following two assertions are equivalent.
    \begin{itemize}
        \item the vectors $\partial_t p(t,x) $ and $\partial_x p(t,x)$ are collinear.
        \item The following equation is satisfied
        \begin{equation} \big(P_2\xi_g(t,x)\big)^T\mathrm{Q}P_2\eta_g(t,x) =0,
        \end{equation}
        where $\xi_g,\eta_g$ are the twists defined in \eqref{def:twistSE(k)}, $\mathrm{Q}$ is the rotation matrix of angle $\frac{\pi}{2}$ and $P_2=\begin{pmatrix}
            0_{2\times1} & I_2
        \end{pmatrix}$.
    \end{itemize}
\end{lemma}

\begin{proof}
The result is immediate after first noticing that $P_2\xi=\Gamma$ and $P_2\eta=V$ and then that collinearity of two vectors $x,y$ in $\mathbb{R}^2$
is equivalent to $x^T \mathrm{Q} y=0$.
\end{proof}

\subsubsection{Derivation of the Cosserat Poincaré equation under follow-the-leader constraint}

The planar Follow the leader model aims at describing the planar motion of a snake moving along the of its backbone: if the snake is moving forward, the motion of the head dictates the trajectory of the entire body, as the rest of the snake follows the path of the head. Within this framework, our objective is to reduce the Cosserat-Poincaré equation to an ordinary differential equation.
To achieve this reduction, we take $g\in C^{2,2}_{\SE(2)}$ representing the snake, with the head and tail respectively located at $x=0$ and $x=1$. We assume that the beam at rest satisfies
$\Gamma_0 = \begin{pmatrix} 1 & 0 \end{pmatrix}^T$ and the inertia and elasticity matrices can be written as $I=\operatorname{diag}(j_1,j_2,j_2)$ and $H=\operatorname{diag}(h_1,h_2,h_3)$, where $j_1,j_2,h_1,h_2,h_3$ are positive constants.
We also merge the Kirchhoff constraint \eqref{eq:Kirchhoff constraint} in dimension $2$, denoted $\mathcal{C}_K$, and the follow-the-leader constraint \eqref{eq:Follow-The-Leader Constraint}, denoted $\mathcal{C}_f$, into a single distributed constraint $\mathcal{C}$ \begin{equation}
\label{eq:Follow the leader constraint + Kirchhoff}
    \mathcal{C}(t,\xi,\eta,\partial_x\eta) =\begin{pmatrix}
\mathcal{C}_K(t,\xi,\eta,\partial_x\eta)\\ \mathcal{C}_{\mathrm{f}}(t,\xi,\eta,\partial_x\eta)
    \end{pmatrix} =\begin{pmatrix}
        P_2\operatorname{ad}_\xi \\ (P_2\xi)^T\mathrm{Q}P_2
    \end{pmatrix}\eta +  \begin{pmatrix}
        P_2 \\ 0_{1 \times 3}
    \end{pmatrix}\partial_x\eta.
\end{equation}
Last, we assume that the snake moves forward along its backbone, meaning the position $p$ satisfies \begin{equation}\label{eq:moving forward}
    \partial_x p(t,\cdot)^T \partial_tp(t,\cdot) >0,\quad \forall t \in [0,T],
\end{equation} and we adapt the variational problem \eqref{eq:Hamilton equation} by replacing \( \mathscr{C}^2_0 \) with the functional space
\begin{equation}
\label{def:new variational space}
\widetilde{\mathscr{C}}^2_0 := \{ v \in C^{2,2}_\mathfrak{g} \mid v(0, \cdot) = v(T, \cdot) = v(\cdot, 0) = 0 \}.
\end{equation}
This variational problem only gives one boundary condition at the head. Thus it is adapted to the assumption \eqref{eq:moving forward} as the boundary condition at the tail no longer needs to be given : the tail follows the trajectory traced by the initial pose and the head.
Let us now suppose that the snake is submitted to a punctual force \(F_1 \in \mathscr{C}^1([0,T];\mathbb{R}) \) acting on the head and the constraint \eqref{eq:Follow the leader constraint + Kirchhoff}.
Given initial conditions $(\eta^0, \xi^0):[0,1]\rightarrow \mathfrak{g}^2$, the twists \( \eta, \xi \in C^{1,1}_{\mathbb{R}^3} \) and the reaction force $ \lambda \in C^{1,1}_{\mathbb{R}^3}$ satisfy the following system of partial differential equations on $[0,T]\times [0,1]$
\begin{subequations}
    \begin{empheq}[left=\empheqlbrace]{align}
    &\label{eq:dynamic follow the leader} I \partial_t \eta - \partial_x \Lambda  - \operatorname{ad}_{\eta}^* I \eta + \operatorname{ad}_{\xi}^* \Lambda  + \mathcal{C}_1^* \lambda - \partial_x \big( \mathcal{C}_2^* \lambda \big) = 0, \\
    &\label{eq :structure follow the leader} \partial_x \eta - \partial_t \xi = [\eta, \xi],  \\
    &\label{constraint Lagrangian exemple 1} \mathcal{C}(t, \xi, \eta, \partial_x \eta) = 0,  \\
    &\label{eq:boundary constraint follow the leader} \Lambda(t,1) = - \mathcal{C}_2^* \lambda(t,1) + F_1(t), \\
    &(\eta(0,x), \xi(0,x)) = (\eta^0(x), \xi^0(x)),
    \end{empheq}
\end{subequations}
with
$$
\Lambda(t,x) =H(x)\big(\xi(t,x)-\xi_0(x)\big),\  \mathcal{C}_1(t,\xi) := \begin{pmatrix}
    P_2\operatorname{ad}_\xi \\ (P_2\xi)^T\mathrm{Q}P_2
\end{pmatrix},\  \mathcal{C}_2(t,\xi) := \begin{pmatrix}
    P_2 \\ 0_{1 \times 3}
\end{pmatrix},\ F_1 = \begin{pmatrix} m_{act} \\ f_{act,1} \\ f_{act,2} \end{pmatrix}.
$$
Using \eqref{constraint Lagrangian exemple 1}, i.e., $\mathcal{C}(t,\xi,\eta,\partial_x\eta)=0$,
implies that the twists $\eta,\xi$ satisfy, for every $(t,x)\in[0,T]\times[0,1]$,
\begin{equation}\label{eq : follow the leader contraint Kirchhoff}
    P_2 \operatorname{ad}_\xi(t,x) \eta(t,x) + P_2 \partial_x\eta = 0,
\end{equation}
and \begin{equation}\label{eq : follow the leader constraint Follow the leader}
    (P_2\xi(t,x))^T \mathrm{Q} P_2 \eta(t,x) =0.
\end{equation}
The structure equation \eqref{eq :structure follow the leader} simplifies \eqref{eq : follow the leader contraint Kirchhoff} into $P_2\partial_t \xi = 0$ or equivalently $\partial_t \Gamma=0$.
If we suppose that the linear space rate at initial condition is at rest, i.e., $\Gamma^0=\Gamma_0$, then \eqref{eq : follow the leader contraint Kirchhoff} can be reduced to
\begin{equation}
\label{eq:Gamma follow the leader}
    \Gamma\equiv \Gamma_0=\begin{pmatrix}
        1\\0
    \end{pmatrix}.
\end{equation}
Moreover, \eqref{eq : follow the leader constraint Follow the leader} ensures that $V$ and $\Gamma$ are collinear, which implies that the second component $V_2$ of $V$ is equal to zero.
In this case, the structure equation \eqref{eq :structure follow the leader} case be written on $[0,T]\times[0,1]$
\begin{subequations}
\label{eq: structure equation follow the leader coord non simplifiée}
    \begin{empheq}[left=\empheqlbrace]{align}
    &\partial_x \Omega = \partial_t K, \\
    &\partial_x V_1 -\partial_t \Gamma_1= K V_2 - \Omega \Gamma_2 , \\
    &\partial_x V_2 - \partial_t \Gamma_2 = \Omega \Gamma_1 - K V_1.
    \end{empheq}
\end{subequations}
where $\Omega$, $\Gamma=\begin{pmatrix}
    \Gamma_1 & \Gamma_2
\end{pmatrix}^T$, $K$ are respectively the angular time and space  \eqref{def:twistSE(k)}. Using equations \eqref{eq : follow the leader constraint Follow the leader} and \eqref{eq:Gamma follow the leader}, the structure equation \eqref{eq: structure equation follow the leader coord non simplifiée} simplifies on $[0,T]\times[0,1]$ to
\begin{subequations}
\label{eq: structure equation follow the leader coord simplifiée}
    \begin{empheq}[left=\empheqlbrace]{align}
    &\partial_x \Omega = \partial_t K, \\
    &\partial_x V_1 = 0,  \\
    &\Omega - K V_1 = 0.
    \label{eq:ftl:1:c}
    \end{empheq}
\end{subequations}
Equation \eqref{eq: structure equation follow the leader coord simplifiée} have been obtained first in \cite{6072271} to study the kinematic of 2D snakes in lateral undulation. In particular, \eqref{eq: structure equation follow the leader coord simplifiée} implies, on  $[0,T] \times [0,1]$, that $V_1(t,x) = V_1(t,1)$ and that
   \( K \) is fully described by the velocity \( V_1(t,1) \), the boundary conditions and initial conditions as it is the solution of the transport equation
\begin{subequations}
\label{eq:transport angular part}
    \begin{empheq}[left=\empheqlbrace]{align}
    &\partial_t K(t,x) - V_1(t,1) \partial_x K(t,x) = 0, \\
    &K(t,1) = K_0(1) + \frac{m_{act}(t)}{h_1}, \quad K(0,x) = K^0(x).
    \label{eq:transport angular part:b}
    \end{empheq}
\end{subequations}
In the above, the boundary condition comes from the projection of the boundary conditions \eqref{eq:boundary constraint follow the leader} on the angular space rate component.
Note that the transport equation \eqref{eq:transport angular part} is well posed in $C^1([0,T]\times[0,1];\mathbb{R})$ as we assume the snake to move forward $V_1=\Gamma^TV=\partial_x p ^T \partial_t p>0$.
Let us now rewrite \eqref{eq:dynamic follow the leader} and \eqref{eq:boundary constraint follow the leader} using the simplifications due to the constraint and the kinematics on $[0,T]\times [0,1]$
\begin{subequations}
\label{eq:dynamic follow the leader coordonné}
    \begin{empheq}[left=\empheqlbrace]{align}
    &\label{eq:dynamic follow the leader coordonné 1} j_1 \partial_t V_1 K + j_1( V_1)^2 \partial_x K - h_1 \partial_x (K - K_0) - \lambda_{2} = 0,  \\
    &\label{eq:dynamic follow the leader coordonné 2} j_2 \partial_t V_1 - \partial_x \lambda_{1} - K \lambda_{2} = 0, \\
    &\label{eq:dynamic follow the leader coordonné 3} \partial_x \lambda_{k,2} - j_2 ( V_1)^2 K - \lambda_3 = 0,\\
    &\lambda_{2}(t,1) = f_{act,2}(t),\quad \lambda_{1}(t,1) = f_{act,1}(t),\label{eq:boundary1}\\
    &K(t,1) = K_0(1) +\frac{m_{act}(t)}{h_1},\label{eq:boundary2}
    \end{empheq}
\end{subequations}
where \( \lambda = \begin{pmatrix} \lambda_{1} & \lambda_{2} & \lambda_3 \end{pmatrix}^T \in \mathbb{R}^3 \). In particular, \eqref{eq:dynamic follow the leader coordonné 1} at section \( x = 1 \) is an ordinary differential equation
\begin{equation}
    j_1 \partial_t V_1(t,1) K(t,1) + j_1 V_1(t,1)^2 \partial_x K(t,1) - h_1 \partial_x (K(t,1) - K_0(1)) - \lambda_{k,2}(t,1) =0.
\end{equation}
Consider now the transport equation \eqref{eq:transport angular part} evaluated at $x=1$. Using \eqref{eq:boundary1} and \eqref{eq:boundary2} and the change of variable $w(t)=j_1V_1(t,1) \big( K_0(1) + \frac{m_{act}(t)}{h_1}\big)$, it reduces to
\begin{equation}
\begin{cases}
\label{eq : velocity}
    w'(t) = U, \\
    w(0)= j_1V^0_1(K_0(1) + \frac{m_{act}(0)}{h_1}),
    \end{cases}
\end{equation}
where $V^0_1$ is the initial velocity $V_1$ and the function
\[
U=-(h_1\partial_xK_0(1)+f_{act,2}(t)) +\frac{j_1m'_{act}(t)(j_1K_0(1)+\frac{j_1}{h_1}m_{act}(t))}{h_1w(t)},
\]
stands for a control input. Hence $V_1=V_1(t,1)$
is completely determined by the actuation. Once it is
is determined, we recover $\lambda_2$ by \eqref{eq:dynamic follow the leader coordonné 1}, then $\lambda_3$ by \eqref{eq:dynamic follow the leader coordonné 3}, and finally $\lambda_1$ by integrating \eqref{eq:dynamic follow the leader coordonné 2} and the boundary condition \eqref{eq:boundary1}.

Recalling the position of the snake $p\in C^{2,2}_{\mathbb{R}^2}$ and the deformation of the snake $R\in C^{2,2}_{\SO(2)}$ satisfy the equations
$\partial_t R=\Omega RQ$, $\partial_t p=RV$,
$\partial_x R=KRQ$ and $\partial_x p=R\Gamma$,
and using \eqref{eq:Gamma follow the leader} and the collinearity of $V$ and $\Gamma_1$, we deduce that
\begin{equation}
\partial_t p(t,x)=R(t,x)V(t,x)=R(t,x)\qmatrix{V_1(t,1) \\ 0}
=V_1(t,1)R(t,x)\Gamma
=V_1(t,1)\partial_x p(t,x),
\end{equation}
and
\begin{equation}
\partial_t R(t,x)
=\Omega(t,x)R(t,x)Q=K(t,x)V_1(t,1)R(t,x)Q
=V_1(t,1)\partial_x R(t,x).
\end{equation}
This concludes the derivation of the dynamics presented in Example~\ref{ex:followtheleader}.

We next show that the head's position $p(\cdot,1)$ and head's deformation $R(\cdot,1)$ can track any regular reference curve in the plane $X : [0,T]\rightarrow\mathbb{R}^2$. To that end, let us define the reference linear velocity by
\begin{equation}
\sigma_r(t):=\Vert X'(t) \Vert,
\end{equation} and the reference angular velocity $\Omega_r$ by
(recall that $\Omega Q=R^{-1}\partial_t R$)
\begin{equation}\label{eq:ref_Omega:1}
\Omega_r(t)\mathrm{Q}:=R_r(t)^{-1}R_r'(t),
\end{equation}
where $R_r(t)=\begin{pmatrix}
    \frac{X'(t)}{\Vert X'(t) \Vert} & Q\frac{X'(t)}{\Vert X'(t) \Vert}
\end{pmatrix}$. Then by choosing as the moment $m_{act}$ and the force $f_{act,2}$ the functions
\begin{equation}
    m_{act}(t):=h_1\big(\frac{\Omega_r(t)}{\sigma_r(t)}-K_0(1)\big),
\end{equation}
and
\begin{equation}
        f_{act,2}(t):= -j_1\partial_t \sigma_r(t)(K_0(1)+\frac{m_{act}(t)}{h_1}) +j_1 \sigma_r(t)\frac{m'_{act}(t)}{h_1}   + \frac{m'_{act}(t)}{\sigma_r(t)} -h_1\partial_x K_0(1) ,
\end{equation}
we have that
$V_1(t,1)=\sigma_r(t)$ i.e.
$V(t,1)=\begin{pmatrix}
    \sigma_r(t) \\ 0
\end{pmatrix}$
and then from \eqref{eq:ftl:1:c} and \eqref{eq:transport angular part:b} that
$\Omega(t,1)=K(t,1)\sigma_r(t)=\Omega_r(t)$ for any time $t\in[0,T]$.
If the initial condition of head's position
and orientation satisfy
\begin{equation}
p(0,1)=X(0), \quad R(0,1)=\begin{pmatrix}
    \frac{X'(0)}{\Vert X'(0) \Vert} & Q\frac{X'(0)}{\Vert X'(0) \Vert}
\end{pmatrix}
\end{equation}
then, as will be shown below, we have the expected property :
\begin{equation}
\label{eq:condition control follow the leader}
    p(t,1)=X(t)
    \quad\textrm{and}\quad
    R(t,1)=R_r(t)
    \quad \forall t \in [0,T].
\end{equation}

First, by definition, the function $R_r$ satisfies $R_r'(t)=\Omega_r(t)R_r(t)Q$ for all $t\in [0,T]$.
On the other hand, since $\Omega(t,1)=\Omega_r(t)$ for all $t$, we have
$\dif{t} R(t,1)=\Omega(t,1)R(t,1)Q=\Omega_r(t)R(t,1)Q$ for all $t$
with $R(0,1)=R_r(0)$.
Thus $R_r$ and $R(\cdot,1)$ solve the same initial value problem
and therefore
\[
R(t,1)=R_r(t)\quad \forall t\in [0,T]
\]
which proves the second identity in \eqref{eq:condition control follow the leader}.
Consequently for $p$ we have
\[
\dif{t} p(t,1)=R(t,1)V(t,1)
=R_r(t)\qmatrix{\sigma_r(t) \\ 0}
=\frac{X'(t)}{\n{X'(t)}} \n{X'(t)}
=X'(t)
\]
for all $t\in [0,T]$
and hence $p(t,1)=X(t)$
because $p(0,1)=X(0)$.
This proves the first identity in \eqref{eq:condition control follow the leader}.

\subsection{Noether's Theorem and Energy Conservation}

Assume that the Lagrangian density $L$, Eq. \eqref{def: Cas particulier Lagrangien density}, does not depend explicitly on time $t$.
That is
\[
\wt{\mc{L}}=\wt{\mc{L}}(x,g,\partial_x g,\partial_t g)
:=\mc{L}(x,g,g^{-1}\partial_x g,g^{-1}\partial_t g).
\]
Assume, moreover, that the Lagrangian is of the form
(see Section \ref{left-invariant Lagrangian section})
\[
\wt{\mc{L}}=\mc{L}
={}& \mathcal{T}_I(\eta) - \mathcal{U}_{H,\xi_0}(\xi)
=\frac{1}{2}\la I\eta,\eta\ra-\frac{1}{2}\la H(\xi-\xi_0),\xi-\xi_0\ra \\
={}& \frac{1}{2}\la Ig^{-1}\partial_t g,g^{-1}\partial_t g\ra-\frac{1}{2}\la H(g^{-1}\partial_x g-\xi_0),g^{-1}\partial_x g-\xi_0\ra
\]
and that we have no constraints and no (external) forces acting on the system.

In notations of \cite{edvardsson:2016},
we have
$x^\mu=(t,x)$, $\phi=g$, symmetry with $\delta x^{\mu}=(1,0)$, $\delta \phi(x^\mu)=0$
(in \cite{edvardsson:2016}: $\mu=0$ corresponds to time $t$
and $\mu=1$ to $x$).
Therefore, in Eq. (4.26) of \cite{edvardsson:2016}
(see also Example 6.1 therein),
we have $\epsilon=\epsilon_1\in\R$,
$\psi_{ik}=0$, $\xi^\mu_1=(1,0)$
and $\lambda^\mu_k=0$,
so that (note: $\phi_{,0}=\partial_t g$)
\[
j^0=j^0_1=\frac{\partial \wt{\mc{L}}}{\partial \phi_{,0}}\phi_{,0}-\wt{\mc{L}}
=\frac{\partial \wt{\mc{L}}}{\partial (\partial_t g)}(\partial_t g)-\wt{\mc{L}}
=\frac{\partial \mc{L}}{\partial \eta}\eta-\mc{L}
=\la I\eta,\eta\ra-\mc{L}
=\mathcal{T}_I(\eta) + \mathcal{U}_{H,\xi_0}(\xi)
\]
while (note: $\phi_{,1}=\partial_x g$)
\[
j^1=j^1_1
=\frac{\partial \wt{\mc{L}}}{\partial \phi_{,1}}\phi_{,1}-\delta^1_0\mc{L}
=\frac{\partial \wt{\mc{L}}}{\partial (\partial_x g)}(\partial_x g)
=\frac{\partial \mc{L}}{\partial \xi}\xi
=-\la \Lambda(t,x),\eta(t,x)\ra
\]
so that
\[
0=\partial_0 j^0+\partial_1 j^1=\partial_t j^0(t,x)+\partial_x j^1(t,x)
\]
and hence
\[
\dif{t} E(t)=\dif{t} \int_0^1 j^0(t,x) dx=-\int_0^1 \dif{x} j^1(t,x) dx
=j^1(t,1)-j^1(t,0)
=0
\]
because $\Lambda(t,0)=\Lambda(t,1)=0$
and hence $j^1(t,0)=j^1(t,1)=0$
since $F^\pm=0$ and $\mc{C}=0$ by assumption.

\subsection{Spaces $\mathscr{C}^k([a,b];\mathscr{C}^l([c,d];B))$}\label{sec:app:C^k(C^l)}

In this section, we will study the structure of the iterated
continuously differentiable spaces
$C^k([a,b];C^l([c,d];B))$ with $B=(B,\n{\cdot})$ a real Banach space with $k,l$ non-negative integers and
$a<b$, $c<d$.
For sets $X,Y$, we write below $Y^X$ for the set of all functions $X\to Y$.

Let $X,Y,Z$ be sets.
For $f\in (Z^Y)^X$ we define
\[
F_f\in Z^{X\times Y};\quad
F_f(x,y):=f(x)(y),\quad x\in X,\ y\in Y.
\]
and for $F\in Z^{X\times Z}$ define
\[
f^F\in (Z^Y)^X;
\quad f^F(x)(y):=F(x,y),\quad x\in X,\ y\in Y.
\]
We also write
\[
F_x\in Z^Y;\quad F_x(y):=F(x,y),\quad x\in X,\ y\in Y
\]
so that $f^F(x)=F(x,\cdot)=F_x$ for $x\in X$.
The maps $f\mapsto F_f$ and $F\mapsto f^F$ are inverses to each other: $f^{F_f}(t)(x)=F_f(t,x)=f(t)(x)$
and $F_{f^F}(t,x)=f^F(t)(x)=F(t,x)$ for every $t\in T$, $x\in X$
and $f\in (Z^Y)^X$, $F\in Z^{X\times Y}$.

Let $\tau$ be the map
\[
\tau:(Z^Y)^X\to (Z^X)^Y;
\quad \tau(f)(y)(x):=f(x)(y),
\quad x\in X,\ y\in Y
\]
Clearly $\tau$ is a bijection from $(Z^Y)^X$ onto $(Z^X)^Y$
and $\tau\circ\tau=\id$ i.e. $\tau^{-1}=\tau$.
Since then $F_{\tau(f)}:Y\times X\to Z$ with
\[
F_{\tau(f)}(y,x)=\tau(f)(y)(x)=f(x)(y)=F_f(x,y),
\]
we also write $\tau$ for the map
\[
\tau:Z^{X\times Y}\to Z^{Y\times X};
\quad
\tau(F)(y,x):=F(x,y).
\]
Then $F_{\tau(f)}=\tau(F_f)$ for every $f\in (Z^Y)^X$
and likewise $\tau(f^F)=f^{\tau(F)}$ for every $F\in Z^{X\times Y}$.

Let $X,Y$ be compact metric spaces and $Z$ a metric space.
Endow $X\times Y$ with the metric $d=d_{X\times Y}$ given by
\[
d((x,y),(x',y')):=d_X(x,x')+d_Y(y,y'),\quad x,x'\in X,\ y,y'\in Y.
\]
Then $X\times Y$ is a compact metric space.
We endow $\mathscr{C}(Y,Z)$ with the metric $D=D_{\mathscr{C}(Y,Z)}$ given by
\[
D(f,g):=\sup_{y\in Y} d_Z(f(y),g(y))
=\max_{y\in Y} d_Z(f(y),g(y)),
\quad f,g\in \mathscr{C}(Y,Z).
\]

Then the following holds
(see e.g. \cite[Proposition A.16]{hatcher:2001} (p. 531)):

\begin{theorem}\label{th:C^0(C^0):1}
If $X,Y$ are \emph{compact} metric spaces and $Z$ is a metric space,
then the map $f\mapsto F_f$ yields
a bijection
\[
\mathscr{C}(X;\mathscr{C}(Y;Z))\cong \mathscr{C}(X\times Y;Z).
\]
\end{theorem}

\begin{remark}\label{re:C^0(C^0):isometry:1}
For every $f,g\in \mathscr{C}(X;\mathscr{C}(Y;Z))$, we have
\[
\max_{(x,y)\in X\times Y} d_Z(F_f(x,y),F_g(x,y))
=\max_{x\in X} \max_{y\in Y} d_Z(f(x)(y),g(x)(y))
=\max_{x\in X} D(f(x),f(y))
\]
and hence the bijection of
Theorem \ref{th:C^0(C^0):1}
becomes an \emph{isometric isomorphism}
if the space $\mscr{C}(X\times Y;Z)$
is equipped with the metric
$d_{\mscr{C}(X\times Y;Z)}(F,G):=\max_{(x,y)\in X\times Y} d_Z(F(x,y),G(x,y))$,
the space $\mscr{C}(Y;Z)$
is equipped with the metric $D$
and the space
$\mscr{C}(X;\mscr{C}(Y;Z))$
is equipped with the metric
$d_{\mscr{C}(X;\mscr{C}(Y;Z))}(f,g)=\max_{x\in X} D(f(x),g(x))$.
\end{remark}

\begin{definition}
For an open set $U\subset\R^n$
and a real Banach space $(E,\n{\cdot}_E)$,
we define $\mathscr{C}^k(\ol{U};E)$
to be the set of all Fr\'echet $\mathscr{C}^k$-functions $F:U\to E$
such that the partial derivatives $\partial^\alpha F:U\to E$
of all orders $|\alpha|\leq k$
extend to continuous maps $\wt{\partial^\alpha F}:\ol{U}\to E$ on the closure $\ol{U}$ of $U$.
As usual, we write $\mathscr{C}^k(\ol{U}):=\mathscr{C}^k(\ol{U};\R)$.
\end{definition}

If the domain $U$ in the previous definition is \emph{bounded}
so that $\ol{U}$ is \emph{compact},
then it is easy to see that $\mathscr{C}^k(\ol{U};E)$
is a Banach space when equipped with the norm
\[
\n{F}_{\mathscr{C}^k}=\n{F}_{\mathscr{C}^k(\ol{U};E)}
:=\sum_{|\alpha|\leq k} \sup_{x\in U} \n{\partial^\alpha F(x)}_E
=\sum_{|\alpha|\leq k} \max_{x\in\ol{U}} \n{\wt{\partial^\alpha F}(x)}_E
=\sum_{|\alpha|\leq k} \n{\partial^\alpha F(x)}_{\mathscr{C}^0(\ol{U};E)}
\]
for $F\in \mathscr{C}^k(\ol{U};E)$.

From now on, we let $T=[a,b]$, $X=[c,d]$ be compact intervals (for simplicity) and
let $(B,\n{\cdot})$ be a real Banach space.
Define the spaces
\begin{align}\label{eq:D^{1,0}:1}
D^{1,0}(T\times X;B):=\{F\in \mathscr{C}(T\times X;B)\ |\ \partial_t F\in \mathscr{C}(T\times X;B)\}
\end{align}
and
\begin{align}\label{eq:D^{0,1}:1}
{}& D^{0,1}(T\times X;B):=\{F\in \mathscr{C}(T\times X;B)\ |\ \partial_x F\in \mathscr{C}(T\times X;B)\}.
\end{align}
More precisely, $F\in D^{1,0}(T\times X;E)$
(resp. $F\in D^{0,1}(T\times X;B)$)
if $F\in \mathscr{C}(T\times X;B)$ and if $\partial_t F$ (resp. $\partial_x F$)
exists, is continuous on $]a,b[\,\times ]c,b[$
and extends continuously onto $X\times T$.

Let $f\in \mathscr{C}^1(T;\mathscr{C}^0(X;B))$ and write $F=f^F$.
Then $f\in \mathscr{C}^0(T;\mathscr{C}^0(X;B))=\mathscr{C}(T\times X;B)$
i.e. $F\in \mathscr{C}(T\times X;B)$ by Theorem \ref{th:C^0(C^0):1}.
Next, we have
\[
\n{\frac{F(t+h,x)-F(t,x)}{h}-(\partial_t f(t))(x)}
={}& \n{\frac{f(t+h)(x)-f(t)(x)}{h}-(\partial_t f(t))(x)} \\
\leq \n{\frac{f(t+h)-f(t)}{h}-\partial_t f(t)}_{\mathscr{C}^0}
\leq {}& \n{\frac{f(t+h)-f(t)}{h}-\partial_t f(t)}_{\mathscr{C}^1}
\longto_{h\to 0} 0
\]
and hence
\[
\exists\, \partial_t F(t,x)=(\partial_t f)(t)(x)
\quad\forall t\in T,\ x\in X.
\]
Moreover since $\partial_t f\in \mathscr{C}^0(T;\mathscr{C}^0(X;B))=\mathscr{C}(T\times X;B)$,
we have ${\partial_t F\in \mathscr{C}(T\times X;B)}$.
Thus $F\in D^{1,0}(T\times X;B)$.

Conversely, let $F\in D^{1,0}(T\times X;B)$
and write $f=f^F$.
Since $F\in \mathscr{C}(T\times X;B)$, we have $f\in \mathscr{C}^0(T;\mathscr{C}^0(X;B))$
by Theorem \ref{th:C^0(C^0):1}.
Next, since $\partial_t F\in \mathscr{C}(T\times X;B)$
and since $T\times X$ is compact,
the map $\partial_t F:T\times X\to B$
is uniformly continuous.
Let $\epsilon>0$. Then there exists $\delta>0$
s.t. $\n{\partial_t F(s,x)-\partial_t F(t,y)}<\epsilon$
whenever $\n{(s,x)-(t,y)}<\delta$.
Thus if $|h|<\delta$,
for any $t\in T$ and $s\in \big[t-|h|,t+|h|\big]\cap T$,
we have $\n{(s,x)-(t,x)}=|s-t|\leq |h|<\delta$.
Therefore, for any $t\in T$ and any $h\neq 0$
such that $|h|<\delta$ and $t+h\in T$, one has
\[
\n{\frac{f(t+h)-f(t)}{h}-\partial_t F(t,\cdot)}_{\mathscr{C}^0}
={}& \sup_{x\in X} \n{\frac{1}{h} \int_t^{t+h} \partial_t F(s,x) ds-\partial_t F(t,x)} \\
\leq {}& \sup_{x\in X} \frac{1}{|h|} \int_{[t-|h|,t+|h|]\cap T} \n{\partial_t F(s,x)-\partial_t F(t,x)} ds \\
\leq {}& \frac{1}{|h|} \int_{t-|h|}^{t+|h|}\epsilon ds
=2\epsilon.
\]
Thus
\[
\exists\, (\partial_t f)(t)=\partial_t F(t,\cdot)\quad\forall t\in T
\]
i.e. $\partial_t f=f^{\partial_t F}$.
Since $\partial_t F\in \mathscr{C}(T\times X;B)=\mathscr{C}^0(T;\mathscr{C}^0(X;B))$,
it then follows from Theorem \ref{th:C^0(C^0):1}
that $\partial_t f=f^{\partial_t F}\in \mathscr{C}^0(T;\mathscr{C}^0(X;B))$.
Hence $f\in \mathscr{C}^1(T;\mathscr{C}^0(X;B))$.

Thus we have shown the following:

\begin{theorem}\label{th:D^{1,0}:1}
The map $f\mapsto F_f$ yields an $\R$-linear isomorphism
\[
\mathscr{C}^1(T;\mathscr{C}^0(X;B))\cong D^{1,0}(T\times X;B).
\]
Moreover, for every $f\in \mathscr{C}^1(T;\mathscr{C}^0(X;B))$, it holds
\[
\partial_t F_f(t,\cdot)=(\partial_t f)(t)
\quad \forall t\in T
\]
i.e. $f^{\partial_t F_f}=\partial_t f$.
\end{theorem}

\begin{remark}
When the space $D^{1,0}(T\times X;B)$ is equipped with the norm
\[
\n{F}_{D^{1,0}}:=\n{F}_{\mathscr{C}(T\times X;B)}+\n{\partial_t F}_{\mathscr{C}(T\times X;B)},
\]
it becomes a Banach space
and the linear isomorphism of Theorem \ref{th:D^{1,0}:1} becomes an \emph{isometric isomorhism}.

Indeed, for any $f\in \mathscr{C}^1(T;\mathscr{C}^0(X;B))$,
the identity $f^{\partial_t F_f}=\partial_t f$
and Remark~\ref{re:C^0(C^0):isometry:1}
imply that
\[
\n{F_f}_{D^{1,0}}
={}& \n{F_f}_{\mathscr{C}(T\times X;B)}+\n{\partial_t F_f}_{\mathscr{C}(T\times X;B)}
=\n{f}_{\mathscr{C}^0(T;\mathscr{C}^0(X;B))}+\n{f^{\partial_t F_f}}_{\mathscr{C}^0(T;\mathscr{C}^0(X;B))} \\
={}& \n{f}_{\mathscr{C}^0(T;\mathscr{C}^0(X;B))}+\n{\partial_t f}_{\mathscr{C}^0(T;\mathscr{C}^0(X;B))}
=\n{f}_{\mathscr{C}^1(T;\mathscr{C}^0(X;B))}.
\]
\end{remark}

Let $f\in \mathscr{C}^0(T;\mathscr{C}^1(X;B))$ and write $F=F_f$.
Since
\[
\n{f(s)-f(t)}_{\mathscr{C}^0}\leq \n{f(s)-f(t)}_{\mathscr{C}^1}\longto_{s\to t} 0,
\]
we have $f\in \mathscr{C}^0(T;\mathscr{C}^0(X;B))=\mathscr{C}(T\times X;B)$
so that $F\in \mathscr{C}(T\times X;B)$.
On the other hand, because $f(t)\in \mathscr{C}^1(X;B)$, one has
\[
\frac{F(t,x+h)-F(t,x)}{h}
={}& \frac{f(t)(x+h)-f(t)(x)}{h}
\longto_{h\to 0} \partial_x(f(t))(x)
\]
and hence
\[
\exists\, \partial_x F(t,x)=\partial_x(f(t))(x)
\quad \forall t\in T,\ x\in X
\]
i.e. $f^{\partial_x F}(t)=(\partial_x F)(t,\cdot)=\partial_x(f(t))$ for all $t\in T$.
Moreover
\[
\n{f^{\partial_x F}(s)-f^{\partial_x F}(t)}_{\mathscr{C}^0}
=\n{\partial_x(f(s))-\partial_x(f(t))}_{\mathscr{C}^0}
\leq \n{f(s)-f(t)}_{\mathscr{C}^1}\longto_{s\to t} 0
\]
so that $f^{\partial_x F}\in \mathscr{C}^0(T;\mathscr{C}^0(X;B))=\mathscr{C}(T\times X;B)$
and therefore $\partial_x F\in \mathscr{C}(T\times X;B)$
by Theorem \ref{th:C^0(C^0):1}.
Thus $F\in D^{0,1}(T\times X;B)$.

Conversely, let $F\in D^{0,1}(T\times X;B)$ and write $f=f^F$
i.e. $f(t)=F(t,\cdot)$ for $t\in T$.
For every $t\in T$, we have $f(t)=F(t,\cdot)\in \mathscr{C}^1(X;B)$
and hence $f$ is a map $T\to \mathscr{C}^1(X;B)$.
In addition
\[
\frac{f(t)(x+h)-f(t)(x)}{h}=\frac{F(t,x+h)-F(t,x)}{h}
\longto_{h\to 0} \partial_x F(t,x)
\]
so that
\[
\exists\, \partial_x(f(t))(x)=\partial_x F(t,x)
\quad\forall t\in T,\ x\in X
\]
i.e. $\partial_x(f(t))=\partial_x F(t,\cdot)=f^{\partial_x F}(t)$ for all $t\in T$.
Finally, since $F,\partial_x F\in \mathscr{C}(T\times X;B)=\mathscr{C}^0(T;\mathscr{C}^0(X;B))$
by Theorem \ref{th:C^0(C^0):1},
we have $f=f^F\in \mathscr{C}^0(T;\mathscr{C}^0(X;B))$
and $f^{\partial_x F}\in \mathscr{C}^0(T;\mathscr{C}^0(X;B))$
and therefore
\[
\n{f(s)-f(t)}_{\mathscr{C}^1}
={}& \n{f(s)-f(t)}_{\mathscr{C}^0}+\n{\partial_x(f(s))-\partial_x(f(t))}_{\mathscr{C}^0} \\
={}& \n{f(s)-f(t)}_{\mathscr{C}^0}+\n{f^{\partial_x F}(s)-f^{\partial_x F}(t)}_{\mathscr{C}^0}
\longto_{s\to t} 0
\]
so that $f\in \mathscr{C}^0(T;\mathscr{C}^1(X;B))$.

We have thus shown the following:

\begin{theorem}\label{th:D^{0,1}:1}
The map $f\mapsto F_f$ yields an $\R$-linear isomorphism
\[
\mathscr{C}^0(T;\mathscr{C}^1(X;B))\cong D^{0,1}(T\times X;B).
\]
Moreover, for every $f\in \mathscr{C}^0(T;\mathscr{C}^1(X;B))$, it holds
\[
\partial_x F_f(t,x)=\partial_x(f(t))(x)
\quad \forall t\in T,\ x\in X
\]
i.e. $f^{\partial_x F}(t)=\partial_x(f(t))$ for all $t\in T$.
\end{theorem}

\begin{remark}
When the space $D^{0,1}(T\times X;B)$ is equipped with the norm
\[
\n{F}_{D^{0,1}}:=\n{F}_{\mathscr{C}(T\times X;B)}+\n{\partial_x F}_{\mathscr{C}(T\times X;B)},
\]
it becomes a Banach space
and the linear isomorphism of Theorem \ref{th:D^{0,1}:1} becomes
a \emph{normed space isomorphism}.

Indeed, for any $f\in \mathscr{C}^0(T;\mathscr{C}^1(X;B))$,
it follows from the identity $f^{\partial_x F}(t)=\partial_x(f(t))$, $t\in T$,
and Remark \ref{re:C^0(C^0):isometry:1} that
\[
\n{F_f}_{D^{0,1}}
={}& \n{F_f}_{\mathscr{C}(T\times X;B)}+\n{\partial_x F_f}_{\mathscr{C}(T\times X;B)}
=\n{f}_{\mathscr{C}^0(T;\mathscr{C}^0(X;B))}+\n{f^{\partial_x F_f}}_{\mathscr{C}^0(T;\mathscr{C}^0(X;B))} \\
={}& \n{f}_{\mathscr{C}^0(T;\mathscr{C}^0(X;B))}+\n{\partial_x (f(\cdot))}_{\mathscr{C}^0(T;\mathscr{C}^0(X;B))}
=\sup_{t\in T} \n{f(t)}_{\mathscr{C}^0(X;B)}
+\sup_{t\in T} \n{\partial_x (f(t))}_{\mathscr{C}^0(X;B)}
\]
while on the other hand
\[
\n{f}_{\mathscr{C}^0(T;\mathscr{C}^1(X;B))}
={}& \sup_{t\in T} \n{f(t)}_{\mathscr{C}^1(X;B)}
=\sup_{t\in T} \big|\n{f(t)}_{\mathscr{C}^0(X;B)}+\n{\partial_x (f(t))}_{\mathscr{C}^0(X;B)}\big| \\
={}& \sup_{t\in T} \big(\n{f(t)}_{\mathscr{C}^0(X;B)}+\n{\partial_x (f(t))}_{\mathscr{C}^0(X;B)}\big).
\]

For any non-empty set $A$ and non-negative functions $g,h:A\to\R_+$,
it holds
\[
\sup_{a\in A} (g(a)+h(a))
\leq \sup_{a\in A} g(a)+\sup_{a\in A} h(a)
\leq 2\sup_{a\in A} (g(a)+h(a)).
\]

Applying this above with $A=T$, $g=\n{f(t)}_{\mathscr{C}^0(X,B)}$
and $h(t)=\n{\partial_x (f(t))}_{\mathscr{C}^0(X,B)}$, $t\in T$,
then yields
\[
\n{f}_{\mathscr{C}^0(T;\mathscr{C}^1(X;B))}
=\n{F_f}_{D^{0,1}}
\leq 2\n{f}_{\mathscr{C}^0(T;\mathscr{C}^1(X;B))}
\]
for every $f\in \mathscr{C}^0(T;\mathscr{C}^1(X;B))$
and thus $f\mapsto F_f$ is indeed an isomorphism
of normed spaces.
\end{remark}

For $k,l=0,1,2,\dots$, we define the spaces
\begin{align}\label{eq:D^{k,l}:1}
D^{k,l}(T\times X;B):=\big\{F\in \mathscr{C}(T\times X;B)\ |\ \partial_x^\alpha \partial_t^\beta F\in \mathscr{C}(T\times X;B)
\ \ \forall\, 0\leq\alpha\leq k,\ 0\leq\beta\leq l\big\}
\end{align}
which coincide with \eqref{eq:D^{1,0}:1} and \eqref{eq:D^{0,1}:1} when $(k,l)=(1,0)$ and $(k,l)=(0,1)$, respectively.
Clearly $\tau$ restricts to an $\R$-linear isomorphism
$D^{k,l}(T\times X;B)\to D^{l,k}(X\times T;B)$.

\begin{exemple}
The space $D^{1,1}$ is
\[
D^{1,1}(T\times X;B)
=\{F\in \mathscr{C}^1(T\times X;B)\ |\ \partial_t \partial_x F=\partial_x \partial_t F\in \mathscr{C}(T\times X;B)\}
\]
\end{exemple}

Before we show a result characterizing the spaces $\mathscr{C}^k(T;\mathscr{C}^l(X;B))$
in terms of the spaces $D^{k+l}(X\times B)$
for general $k,l\geq 0$
(the cases $(k,l)=(1,0)$ and $(k,l)=(0,1)$
being shown in Theorems \ref{th:D^{1,0}:1}
and \ref{th:D^{0,1}:1}),
we need the following:

\begin{proposition}\label{pr:x-reduction:1}
Let $f\in \mathscr{C}^k(T;\mathscr{C}^{l+1}(X;B))$ for $k,l\geq 0$.
Write $R_f$ for $f$ as a map $T\to \mathscr{C}^l(X;B)$
and $g:=\big(t\mapsto \partial_x (f(\cdot))\big)$.
Then $R_f$ and $g$ are in $\mathscr{C}^k(T;\mathscr{C}^l(X;B))$.
Moreover, it holds
\begin{gather}\label{eq:pr:x-reduction:1:result:1}
\partial_t^\alpha R_f(t)(x)=\partial_t^\alpha f(t)(x) \\
\partial_t^\alpha g(t)=\partial_x (\partial_t^\alpha f(t))
\nonumber
\end{gather}
for all $t\in T$, $x\in x$ and $0\leq\alpha\leq k$.
\end{proposition}

\begin{proof}
Let $f\in \mathscr{C}^k(T;\mathscr{C}^{k+1}(X;B))$ and
write $g(t)=\partial_x (f(t))$, $t\in T$,
as in the statement.
For every $t\in T$, we have $f(t)\in \mathscr{C}^{l+1}(X;B)$
and hence $g(t)\in \mathscr{C}^l(X;B)$.
Thus $g$ is a map $T\to \mathscr{C}^l(X;B)$.
Moreover
\[
\n{g(s)-g(t)}_{\mathscr{C}^l}
=\n{\partial_x (f(s))-\partial_x(f(t))}_{\mathscr{C}^l}
\leq \n{f(s)-f(t)}_{\mathscr{C}^{l+1}}\longto_{s\to t} 0
\]
so that $g\in \mathscr{C}^0(T;\mathscr{C}^l(X;B))$.
Also \eqref{eq:pr:x-reduction:1:result:1} holds with $\alpha=0$
by the definition of $g$.
Furthermore
\[
\n{R_f(s)-R_f(t)}_{\mathscr{C}^l}=\n{f(s)-f(t)}_{\mathscr{C}^l}
\leq \n{f(s)-f(t)}_{\mathscr{C}^{l+1}}
\longto_{s\to t } 0
\]
so that $R_f\in \mathscr{C}^0(T;\mathscr{C}^l(X;B))$.

For the sake of induction,
suppose we have shown that
$R_f\in \mathscr{C}^m(T;\mathscr{C}^l(X;B))$ and
$g\in \mathscr{C}^m(T;\mathscr{C}^l(X;B))$
for some $0\leq m\leq k-1$ and that \eqref{eq:pr:x-reduction:1:result:1} holds with $\alpha=m$.
Then for any $t\in T$, we have
\[
{}& \n{\frac{(\partial_t^m g)(t+h)-(\partial_t^m g)(t)}{h}-\partial_x(\partial^{m+1}_t f(t))}_{\mathscr{C}^l(X;B)} \\
={}& \n{\frac{\partial_x ((\partial_t^m f)(t+h))-\partial_x ((\partial_t^m f)(t))}{h}-\partial_x(\partial^{m+1}_t f(t))}_{\mathscr{C}^l(X;B)} \\
\leq {}& \n{\frac{(\partial_t^m f)(t+h)-(\partial_t^m f)(t)}{h}-\partial^{m+1}_t f(t)}_{\mathscr{C}^{l+1}(X;B)}
\longto_{h\to 0} 0
\]
so that $\partial_t^m g$ is differentiable as a map
$T\to \mathscr{C}^l(X;B)$ with
\[
\exists\, \partial_t(\partial_t^m g)(t)=\partial_t^{m+1} g(t)
=\partial_x(\partial^{m+1}_t f(t))
\quad \forall t\in T.
\]
Hence \eqref{eq:pr:x-reduction:1:result:1} holds
with $\alpha=m+1$.
It also follows that
\[
\n{(\partial^{m+1}_t g)(s)-(\partial_t^{m+1} g)(t)}_{\mathscr{C}^l}
={}& \n{\partial_x ((\partial_t^{m+1} f)(s))-\partial_x((\partial_t^{m+1} f)(t))}_{\mathscr{C}^l} \\
\leq {}& \n{(\partial_t f)(s)-(\partial_t f)(t)}_{\mathscr{C}^{l+1}}\longto_{s\to t} 0
\]
and therefore $g\in \mathscr{C}^{m+1}(T;\mathscr{C}^l(X;B))$.

Likewise, for any $t\in T$, we have
\[
{}& \n{\frac{(\partial_t^m R_f)(t+h)-(\partial_t^m R_f)(t)}{h}-\partial^{m+1}_t f(t)}_{\mathscr{C}^l(X;B)} \\
{}& \n{\frac{(\partial_t^m f)(t+h)-(\partial_t^m f)(t)}{h}-\partial^{m+1}_t f(t)}_{\mathscr{C}^l(X;B)} \\
\leq {}& \n{\frac{(\partial_t^m f)(t+h)-(\partial_t^m f)(t)}{h}-\partial^{m+1}_t f(t)}_{\mathscr{C}^{l+1}(X;B)}
\longto_{h\to 0} 0
\]
which shows that $\partial_t^m R_f$ is differentiable
as a map $T\to \mathscr{C}^l(X;B)$
and $(\partial_t^{m+1} R_f)(t))(x)=(\partial^{m+1}_t f(t))(x)$
for every $t\in T$, $x\in X$.
Consequently
\[
\n{(\partial_t^{m+1} R_f)(s)-(\partial_t^{m+1} R_f)(t)}_{\mathscr{C}^l}
=\n{\partial^{m+1}_t f(s)-\partial^{m+1}_t f(t)}_{\mathscr{C}^l}
\leq \n{\partial^{m+1}_t f(s)-\partial^{m+1}_t f(t)}_{\mathscr{C}^{l+1}}
\longto_{s\to t} 0
\]
which implies that $R_f\in \mathscr{C}^{m+1}(T;\mathscr{C}^l(X;B))$.

Thus by induction, we conclude that
$R_f\in \mathscr{C}^k(T;\mathscr{C}^l(X;B))$,
$g\in \mathscr{C}^k(T;\mathscr{C}^l(X;B))$
and that \eqref{eq:pr:x-reduction:1:result:1} holds.
The proof is complete.
\end{proof}

\begin{theorem}\label{th:D^{k,l}:1}
For any $k,l\geq 0$,
the map $f\mapsto F_f$ yields an $\R$-linear isomorphism
\begin{align}\label{eq:th:D^{k,l}:1:result:1}
\mathscr{C}^k(T;\mathscr{C}^l(X;B))\cong D^{k,l}(T\times X;B)
\end{align}
Moreover, for every $f\in \mathscr{C}^k(T;\mathscr{C}^l(X;B))$, it holds
\begin{align}\label{eq:th:D^{k,l}:1:result:2}
\partial_t^\alpha \partial_x^\beta F_f(t,x)
=\partial_x^\beta ((\partial_t^\alpha f)(t))(x)
\quad \forall t\in T,\ x\in X,\ 0\leq\alpha\leq k,\ 0\leq\beta\leq l.
\end{align}
\end{theorem}

\begin{proof}
We do an induction argument w.r.t $k,l\geq 0$.
When $k=l=0$ the result follows from Theorem \ref{th:C^0(C^0):1}
and the fact that \eqref{eq:th:D^{k,l}:1:result:2}
is trivially satisfied by the definition of $F_f$.

The proof is by double induction on $k$ and $l$,
and for this reason the argument is divided into two parts:
(a) for the induction on $k$ and and (b)
for the induction on $l$.

\medskip
\noindent {\bf (a):}
Let $l\geq 0$ be fixed.
\emph{Induction hypothesis:}
Suppose that our result holds for all $0\leq k\leq K$.
Let $f\in \mathscr{C}^{K+1}(T;\mathscr{C}^l(X;B))$.

Then $\partial_t f\in \mathscr{C}^K(T;\mathscr{C}^l(X;B))$
and therefore our induction hypothesis
implies that $F_{\partial_t f}\in D^{K,l}(T\times X;B)$
with
\[
\partial_t^\alpha \partial_x^\beta F_{\partial_t f}(t,x)
=\partial_x^\beta \big(\partial_t^\alpha (\partial_t f)(t)\big)(x)
=\partial_x^\beta (\partial_t^{\alpha+1} f(t))(x)
\]
for all $t\in T$, $x\in X$ and $0\leq\alpha\leq K$, $0\leq\beta\leq l$.

Since $f\in \mathscr{C}^{K+1}(T;\mathscr{C}^l(X;B))$,
we have $f\in \mathscr{C}^1(T;\mathscr{C}^l(X;B))$
and hence $f\in \mathscr{C}^1(T;\mathscr{C}^0(X;B))$
by Proposition \ref{pr:x-reduction:1}.
Then Theorem \ref{th:D^{1,0}:1}
implies that $f^{\partial_t F}=\partial_t f$
i.e. $\partial_t F_f=F_{\partial_t f}$
which is in $D^{K,l}(T\times X;B)$ by the above.
Also because $f\in \mathscr{C}^{K+1}(T;\mathscr{C}^l(X;B))\subset \mathscr{C}^K(T;\mathscr{C}^l(X;B))$,
we have $F_f\in D^{K,l}(T\times X;B)$ by the induction hypothesis.
From these observations, it follows that $F_f\in D^{K+1,l}(T\times X;B)$ as desired.
Moreover, since $\partial_t F_f=F_{\partial_t f}$,
we have by Schwarz's theorem (\cite{wiki:Schwarz_theorem}) and by the above that
for all $t\in T$, $x\in X$ and $0\leq\beta\leq l$,
it holds
\[
\partial_t^{\alpha+1} \partial_x^\beta F_f(t,x)
=\partial_t^\alpha \partial_x^\beta (\partial_t F_f)(t,x)
=\partial_t^\alpha \partial_x^\beta F_{\partial_t f}(t,x)
=\partial_x^\beta (\partial_t^{\alpha+1} f(t))(x)
\]
for all $0\leq\alpha\leq K$
i.e.
$\partial_t^{\alpha'} \partial_x^\beta F_f(t,x)
=\partial_x^\beta (\partial_t^{\alpha'} f)(t)(x)$
holds for all $1\leq\alpha'\leq K+1$.
Finally, since $f\in \mathscr{C}^0(T;\mathscr{C}^l(X;B))$,
this also holds for $\alpha'=0$
by the induction hypothesis.

Therefore, by induction,
the theorem holds for all $(k,l)$ with $k=0,1,2,\dots$
if it holds for $(0,l)$.

\medskip
\noindent {\bf (b):}
\emph{Induction hypothesis:}
Suppose that our result holds for all $0\leq l\leq L$
and all $k\geq 0$.
Let $f\in \mathscr{C}^k(T;\mathscr{C}^{L+1}(X;B))$.

Then by Proposition \ref{pr:x-reduction:1},
we have $g:=\big(t\mapsto \partial_x (f(\cdot))\big)\in \mathscr{C}^k(T;\mathscr{C}^L(X;B))$
with \eqref{eq:pr:x-reduction:1:result:1} holding.
Therefore our induction hypothesis implies that
$F_g\in D^{k,L}(T\times X;B)$
with
\[
\partial_t^\alpha \partial_x^\beta F_g(t,x)
=\partial_x^\beta (\partial_t^\alpha g(t))(x)
=\partial_x^\beta \big(\partial_x (\partial_t^\alpha f(t))\big)(x)
=\partial_x^{\beta+1} (\partial_t^\alpha f(t))(x)
\]
for all $t\in T$, $x\in X$ and $0\leq\alpha\leq k$, $0\leq\beta\leq L$.

Since $f\in \mathscr{C}^k(T;\mathscr{C}^{L+1}(X;B))$,
we have $f\in \mathscr{C}^k(T;\mathscr{C}^1(X;B))$ by Proposition \ref{pr:x-reduction:1}
and hence $f\in \mathscr{C}^0(T;\mathscr{C}^1(X;B))$.
Then Theorem \ref{th:D^{0,1}:1} implies
that $F_g(t,x)=g(t)(x)=\partial_x (f(t))(x)=\partial_x F_f(t,x)$
for every $t,x$
i.e. $F_g=\partial_x F_f$.
Hence by the above,
for every $t\in T$, $x\in X$ and $0\leq\alpha\leq k$, it holds
\[
\partial_t^\alpha \partial_x^{\beta+1} F_f(t,x)=\partial_t^\alpha \partial_x^\beta (\partial_x F_f)(t,x)
=\partial_x^{\beta+1} (\partial_t^\alpha f(t))(x)
\]
for all $0\leq\beta\leq L$
i.e. $\partial_t^\alpha \partial_x^{\beta'} F_f(t,x)=\partial_x^{\beta'} (\partial_t^\alpha f(t))(x)$
holds for all $1\leq\beta'\leq L+1$.
Since $f\in \mathscr{C}^k(T;\mathscr{C}^0(X;B))$ by Proposition \ref{pr:x-reduction:1},
this also holds for $\beta'=0$ by the induction hypothesis.

Finally, since $f\in \mathscr{C}^k(T;\mathscr{C}^0(X;B))$,
it follows from the induction hypothesis that $F_f\in D^{k,0}(T\times X;B)$.
This combined with the above fact that $\partial_x F_f=F_g\in D^{k,L}(T\times X;B)$,
allows us to conclude that $F_f\in D^{k,L+1}(T\times X;B)$
as desired.

Therefore, by induction,
the theorem holds for all $(k,l)$ with $0\leq k\leq K$
and $l=0,1,2,\dots$

\medskip
To conclude, by (a) and (b) the theorem holds for all $k,l\geq 0$.

\end{proof}

\begin{corollary}\label{cor:th:D^{k,l}:1:1}
For any $k,l\geq 0$,
the map $f\mapsto\tau(f)$ is an $\R$-linear isomorphism
\begin{align}\label{eq:cor:th:D^{k,l}:1:1:result:1}
\mathscr{C}^k(T;\mathscr{C}^l(X;B))\cong \mathscr{C}^l(X;\mathscr{C}^k(T;B)).
\end{align}
\end{corollary}

\begin{proof}
Clearly $\tau$ is an $\R$-linear isomorphism $D^{k,l}(T\times X;B)\to D^{l,k}(X\times T;B)$.
If $f\in \mathscr{C}^k(T;\mathscr{C}^l(X;B))$ then $F_f\in D^{k,l}(T\times X;B)$
and so $F_{\tau(f)}=\tau(F_f)\in D^{l,k}(X\times T;B)$
i.e. $\tau(f)\in \mathscr{C}^l(X;\mathscr{C}^k(T;B))$.

Likewise $\tau^{-1}=\tau$ is a bijection $D^{l,k}(X\times T;B)\to D^{k,l}(T\times X;B)$.
Hence if $g\in \mathscr{C}^l(X;\mathscr{C}^k(T;B))$ then
$F_g\in D^{l,k}(X\times T;B)$
and so $F_{\tau^{-1}(g)}=\tau^{-1}(F_g)\in D^{k,l}(T\times X;B)$
i.e. $\tau^{-1}(g)\in \mathscr{C}^k(T;\mathscr{C}^l(X;B))$.
The result follows.
\end{proof}

Let $k,l\geq 0$.
Equipped with the norm
\[
\n{F}_{D^{k,l}}
:=\sum_{\alpha=0}^k \sum_{\beta=0}^l \n{\partial_t^\alpha \partial_x^\beta F}_{\mathscr{C}^0(T\times X;B)}
=\sum_{\alpha=0}^k \sum_{\beta=0}^l \sup_{(t,x)\in T\times X} \n{\partial_t^\alpha \partial_x^\beta F(t,x)},
\]
the space $D^{k,l}(T\times X;B)$ becomes a real Banach space.
Moreover, we have

\begin{theorem}
The $\R$-linear isomorphisms \eqref{eq:th:D^{k,l}:1:result:1} and \eqref{eq:cor:th:D^{k,l}:1:1:result:1}
are normed space isomorphisms.
\end{theorem}

\begin{proof}
For any $f\in \mathscr{C}^k(T;\mathscr{C}^l(X;B))$, we have
(see \eqref{eq:th:D^{k,l}:1:result:2})
\[
\n{f}_{\mathscr{C}^k(T;\mathscr{C}^l(X;B))}
={}& \sum_{\alpha=0}^k \sup_{t\in T} \n{\partial_t^\alpha f(t)}_{\mathscr{C}^l(X;B)}
=\sum_{\alpha=0}^k \sup_{t\in T} \sum_{\beta=0}^l \sup_{x\in X} \n{\partial_x^\beta (\partial_t^\alpha f(t))(x)} \\
={}& \sum_{\alpha=0}^k \sup_{t\in T} \sum_{\beta=0}^l \sup_{x\in X} \n{\partial_t^\alpha \partial_x^\beta F_f(t,x)}.
\]

For any non-empty set $A$ and non-negative
functions $h_1,\dots,h_n:A\to\R_+$, it holds
\[
\sup_{a\in A} \sum_{i=1}^n h_i(a)
\leq \sum_{i=1}^n \sup_{a\in A} h_i(a)
\leq n\sup_{a\in A} \sum_{i=1}^n h_i(a).
\]
Consequently
\[
\n{f}_{\mathscr{C}^k(T;\mathscr{C}^l(X;B))}
\leq {}& \sum_{\alpha=0}^k \sum_{\beta=0}^l \sup_{t\in T} \sup_{x\in X} \n{\partial_t^\alpha \partial_x^\beta F_f(t,x)}
=\sum_{\alpha=0}^k \sum_{\beta=0}^l \sup_{(t,x)\in T\times X} \n{\partial_t^\alpha \partial_x^\beta F_f(t,x)} \\
={}& \sum_{\alpha=0}^k \sum_{\beta=0}^l \sup_{(t,x)\in T\times X} \n{\partial_t^\alpha \partial_x^\beta F_f(t,x)}
=\n{F_f}_{D^{k,l}}
\]
and likewise
\[
\n{F_f}_{D^{k,l}}
=\sum_{\alpha=0}^k \sum_{\beta=0}^l \sup_{t\in T} \sup_{x\in X} \n{\partial_t^\alpha \partial_x^\beta F_f(t,x)}
\leq l\sum_{\alpha=0}^k \sup_{t\in T} \sum_{\beta=0}^l \sup_{x\in X} \n{\partial_t^\alpha \partial_x^\beta F_f(t,x)}
=l\n{f}_{\mathscr{C}^k(T,\mathscr{C}^l(X,B))}.
\]
That is
\[
\n{f}_{\mathscr{C}^k(T;\mathscr{C}^l(X;B))}\leq \n{F_f}_{D^{k,l}}\leq l\n{f}_{\mathscr{C}^k(T;\mathscr{C}^l(X;B))}
\quad\forall f\in \mathscr{C}^k(T;\mathscr{C}^l(X;B))
\]
which shows that \eqref{eq:th:D^{k,l}:1:result:1}
is a normed space isomorphism.

For any $F\in D^{k,l}(T\times X;B)$, we have
\[
\n{\tau(F)}_{D^{l,k}}
=\sum_{\beta=0}^l \sum_{\alpha=0}^k \sup_{(x,t)\in X\times T} \n{\partial_x^\beta \partial_t^\alpha \tau(F)(x,t)}
=\sum_{\alpha=0}^k \sum_{\beta=0}^l \sup_{(t,x)\in T\times X} \n{\partial_t^\alpha \partial_x^\beta F(t,x)}
=\n{F}_{D^{k,l}}
\]
so that $\tau$ is an isometric isomorphism $D^{k,l}\to D^{l,k}$.
Therefore, because \eqref{eq:th:D^{k,l}:1:result:1}
is a normed space isomorphism by the above,
we deduce that $\tau$ as a map $\mathscr{C}^k(T;\mathscr{C}^l(X;B))\to \mathscr{C}^l(X;\mathscr{C}^k(T;B))$
is also a normed space isomorphism.
The proof is complete.
\end{proof}

\subsection{Dual of $\mathscr{C}^k([a,b];V)$}\label{sec:dualCk}

For a finite dimensional $\R$-linear space $V$, $a<b$ and $k\in\N$,
the dual space
$\mathscr{C}^k([a,b];V)^*$ of $\mathscr{C}^k([a,b];V)$ consists of all $\R$-linear maps $\lambda:\mathscr{C}^k([a,b];V)\to\R$ of the form
\begin{align}\label{eq:dual_of_C^k:1}
\lambda(f)=\sum_{j=0}^k \int_{[a,b]} f^{(j)} d\mu_j,
\quad f\in \mathscr{C}^k([a,b];V),
\end{align}
where $\mu_0,\mu_1,\dots,\mu_k$ are $V^*$-valued Radon measures on $[0,1]$.
This result is readily obtained by embedding
$\mathscr{C}^k([a,b];V)$ into $\mathscr{C}([a,b];V)^{k+1}$
via the map
$f\mapsto (f,f',f'',\dots,f^{(k)})$, extending
$\lambda$ to a bounded linear functional $\wt{\lambda}$ on $\mathscr{C}([a,b];V)^{k+1}$ by Hahn-Banach theorem and finally using Riesz representation theorem on $\wt{\lambda}$.
One should note that this representation of $\lambda$
is not unique in terms of the measures $\mu_0,\mu_1,\dots,\mu_k$.

To obtain a unique representation,
consider first an $\R$-valued signed
Radon measure $\nu$ on $[a,b]$
and $f\in \mathscr{C}^1([a,b])$.
Since $f(x)=f(a)+\int_a^x f(t)dt$ for all $x\in [a,b]$,
it follows from Fubini's theorem that
\[
\int_{[a,b]} f(x)\nu(dx)
={}& \nu([a,b]) f(a)+\int_{[a,b]} \int_a^x f'(t)dt\nu(dx) \\
={}& \nu([a,b]) f(a)+\int_{[a,b]} \int_{[a,b]} \indic(t<x) f'(t)dt\nu(dx) \\
={}& \nu([a,b]) f(a)+\int_{[a,b]} \int_{[a,b]} \indic(t<x) f'(t)\nu(dx) dt \\
={}& \nu([a,b]) f(a)+\int_{[a,b]} \nu(]t,b]) f'(t) dt
\]
i.e.,
\[
\int_{[a,b]} f(x)\nu(dx)
=c f(a)+\int_{[a,b]} f'(t) g(t) dt
\]
with $c:=\nu([a,b])\in\R$
and $g(t):=\nu(]t,b])$, $t\in [a,b]$.
Notice that $g:[a,b]\to\R$ is left-continuous
and bounded (since $|g(t)|=|\nu(]t,b])|\leq |\nu|(]t,b])\leq |\nu|([a,b])<\infty$ for all $t\in [a,b]$).
Hence $g(t)dt$ is a (Borel and hence) Radon measure on $[a,b]$.

From the previous observation and induction in \eqref{eq:dual_of_C^k:1}, one deduces that
any $\lambda\in \mathscr{C}^k([a,b],V)^*$ can be represented
in the form (cf. \cite[Exercise IV.13.36]{dunford1988})
\begin{align}\label{eq:dual_of_C^k:2}
\lambda(f)=\sum_{j=1}^{k-1} \la c_j,f^{(j)}(a)\ra_{V^*,V}+\int_{[a,b]} f^{(k)} d\mu,
\quad f\in \mathscr{C}^k([a,b],V),
\end{align}
for unique constants $c_1,\dots,c_{k-1}\in V^*$
and a unique $V$-valued Radon measure $\mu$ on $[a,b]$.


\bibliographystyle{plain}
\bibliography{biblio.bib}

@inproceedings{Choset2009,
 author = {Hatton, Ross L. and Choset, Howie},
 title = {Generating gaits for snake robots by annealed chain fitting and keyframe wave extraction},
 booktitle = {Proc. IEEE/RSJ Int. Conf. Intell. Robots syst.},
 series = {IROS'09},
 year = {2009},
 isbn = {978-1-4244-3803-7},
 location = {St. Louis, MO, USA},
 pages = {840--845},
 numpages = {6},
}

@article{Gray1946,
author = {GRAY, J.},
title = {The Mechanism of Locomotion in Snakes},
journal = {J. Exp. Biol.},
volume = {23},
number = {2},
pages = {101-120},
year = {1946},
}

@ARTICLE{Boyer2011,
title={Recursive Inverse Dynamics of Mobile Multibody Systems With Joints and Wheels},
author={Boyer, F. and Ali, S.},
journal={IEEE Transactions on Robotics},
year={2011},
month={April},
volume={27},
number={2},
pages={215 - 228},
doi={10.1109/TRO.2010.2103450},
ISSN={1552-3098}, }

@INPROCEEDINGS{Ostrowski96gaitkinematics,
    author = {Jim Ostrowski and Joel Burdick},
    title = {Gait kinematics for a serpentine robot},
    booktitle = {Proc. IEEE Int. Conf. on Rob. and Autom.},
    year = {1996},
    pages = {1294--1299}
}

@article{Choseta,
author = {Howie Choset and Wade Henning},
    title = {A Follow-The-Leader Approach To Serpentine Robot Motion Planning},
    journal = {ASCE Journal of Aerospace Engineering},
    year = {1999},
    volume = {12}
}

@article{Liljeback,
 author = {Liljeback, Paal and Pettersen, Kristin Y. and Stavdahl, Oyvind and Gravdahl, Jan Tommy},
 title = {Hybrid modelling and control of obstacle-aided snake robot locomotion},
 journal = {Trans. Rob.},
 issue_date = {October 2010},
 volume = {26},
 issue = {5},
 month = {October},
 year = {2010},
 issn = {1552-3098},
 pages = {781--799},
  publisher = {IEEE Press},
 address = {Piscataway, NJ, USA},
}

@Article{BurgnerKahrs2015,
  author   = {Burgner-Kahrs, Jessica and Rucker, D. Caleb and Choset, Howie},
  journal  = {IEEE Transactions on Robotics},
  title    = {Continuum Robots for Medical Applications: A Survey},
  year     = {2015},
  number   = {6},
  pages    = {1261-1280},
  volume   = {31},
  doi      = {10.1109/TRO.2015.2489500},
}

@article{boyer2011macrocontinuous,
  title={Macrocontinuous dynamics for hyperredundant robots: application to kinematic locomotion bioinspired by elongated body animals},
  author={Boyer, Fr{\'e}d{\'e}ric and Ali, Shaukat and Porez, Mathieu},
  journal={IEEE Transactions on Robotics},
  volume={28},
  number={2},
  pages={303--317},
  year={2011},
  publisher={IEEE}
}

@article{boyer2010poincare,
  title={Poincar{\'e}--Cosserat equations for the lighthill three-dimensional large amplitude elongated body theory: Application to robotics},
  author={Boyer, Frederic and Porez, Mathieu and Leroyer, Alban},
  journal={Journal of Nonlinear Science},
  volume={20},
  pages={47--79},
  year={2010},
  publisher={Springer}
}

@ARTICLE{6072271,
  author={Boyer, Frédéric and Ali, Shaukat and Porez, Mathieu},
  journal={IEEE Transactions on Robotics},
  title={Macrocontinuous Dynamics for Hyperredundant Robots: Application to Kinematic Locomotion Bioinspired by Elongated Body Animals},
  year={2012},
  volume={28},
  number={2},
  pages={303-317},
  doi={10.1109/TRO.2011.2171616}}

@ARTICLE{7314984,
  author={Burgner-Kahrs, Jessica and Rucker, D. Caleb and Choset, Howie},
  journal={IEEE Transactions on Robotics},
  title={Continuum Robots for Medical Applications: A Survey},
  year={2015},
  volume={31},
  number={6},
  pages={1261-1280},
  doi={10.1109/TRO.2015.2489500}}

@book{marsden1999introduction,
  title={Introduction to mechanics and symmetry: a basic exposition of classical mechanical systems},
  author={Marsden, Jerrold E and Ratiu, Tudor S and Golubitsky, M},
  volume={17},
  year={1999},
  publisher={Springer}
}

@article{poincare1901forme,
  title={Sur une forme nouvelle des {\'e}quations de la m{\'e}canique},
  author={Poincar{\'e}, Henri and others},
  journal={CR Acad. Sci},
  volume={132},
  pages={369--371},
  year={1901}
}

@article{bloch2025infinite,
  title={Infinite-dimensional and field-theoretic nonholonomic mechanics},
  author={Bloch, Anthony M and Zenkov, Dmitry V},
  journal={Regular and Chaotic Dynamics},
  volume={30},
  number={4},
  pages={550--565},
  year={2025},
  publisher={Springer}
}

@article{rao2021model,
  title={How to model tendon-driven continuum robots and benchmark modelling performance},
  author={Rao, Priyanka and Peyron, Quentin and Lilge, Sven and Burgner-Kahrs, Jessica},
  journal={Frontiers in Robotics and AI},
  volume={7},
  pages={630245},
  year={2021},
  publisher={Frontiers Media SA}
}

@article{canali2022design,
  title={Design of a novel long-reach cable-driven hyper-redundant snake-like manipulator for inspection and maintenance},
  author={Canali, Carlo and Pistone, Alessandro and Ludovico, Daniele and Guardiani, Paolo and Gagliardi, Roberto and De Mari Casareto Dal Verme, Lorenzo and Sofia, Giuseppe and Caldwell, Darwin G},
  journal={Applied Sciences},
  volume={12},
  number={7},
  pages={3348},
  year={2022},
  publisher={MDPI}
}

@phdthesis{strohmeyer2018networks,
  title={Networks of nonlinear thin structures-theory and applications},
  author={Strohmeyer, Christoph},
  year={2018},
  school={Dissertation, Erlangen, Friedrich-Alexander-Universit{\"a}t Erlangen-N{\"u}rnberg~…}
}

@ARTICLE{10494907,
  author={Xun, Lingxiao and Zheng, Gang and Kruszewski, Alexandre},
  journal={IEEE Transactions on Robotics},
  title={Cosserat-Rod-Based Dynamic Modeling of Soft Slender Robot Interacting With Environment},
  year={2024},
  volume={40},
  number={},
  pages={2811-2830},
  doi={10.1109/TRO.2024.3386393}}

@book{dunford1988,
  title={Linear Operators, Part 1: General Theory},
  author={Dunford, N. and Schwartz, J. T.},
  year={1988},
  publisher={Wiley}
}

@book{hatcher:2001,
  title={Algebraic Topology},
  author={Hatcher, A.},
  year={2001},
  publisher={}
}

@misc{holm1998eulerpoincareequationssemidirectproducts,
      title={The Euler-Poincare Equations and Semidirect Products with Applications to Continuum Theories},
      author={D. D. Holm and J. E. Marsden and T. S. Ratiu},
      year={1998},
      eprint={chao-dyn/9801015},
      archivePrefix={arXiv},
      primaryClass={chao-dyn},
      url={https://arxiv.org/abs/chao-dyn/9801015},
}

@book{kirillov2008introduction,
  title={An introduction to Lie groups and Lie algebras},
  author={Kirillov, Alexander A},
  volume={113},
  year={2008},
  publisher={Cambridge University Press Cambridge}
}

@article{Cox_1970, title={The motion of long slender bodies in a viscous fluid Part 1. General theory}, volume={44}, DOI={10.1017/S002211207000215X}, number={4}, journal={Journal of Fluid Mechanics}, author={Cox, R. G.}, year={1970}, pages={791–810}}

@article{BoyerPrimault,
title = {The Poincaré-Chetayev equations and flexible multibody systems},
journal = {Journal of Applied Mathematics and Mechanics},
volume = {69},
number = {6},
pages = {925-942},
year = {2005},
issn = {0021-8928},
doi = {https://doi.org/10.1016/j.jappmathmech.2005.11.015},
url = {https://www.sciencedirect.com/science/article/pii/S0021892805001383},
author = {F. Boyer and D. Primault}
}

@BOOK{Bloch,
   author = {A. Bloch and P. Crouch and J. Baillieul and J. Marsden},
   title = {Nonholonomic Mechanics and Control},
   publisher = {Springer-Verlag},
   address={New York},
   year = {2007}
   }

@article{boyer2017poincare,
  title={Poincare’s equations for cosserat media: Application to shells},
  author={Boyer, Frederic and Renda, Federico},
  journal={Journal of Nonlinear Science},
  volume={27},
  pages={1--44},
  year={2017},
  publisher={Springer}
}

@article{boyer2022extended,
  title={Extended Hamilton’s principle applied to geometrically exact Kirchhoff sliding rods},
  author={Boyer, Fr{\'e}d{\'e}ric and Lebastard, Vincent and Candelier, Fabien and Renda, Federico},
  journal={Journal of Sound and Vibration},
  volume={516},
  pages={116511},
  year={2022},
  publisher={Elsevier}
}

@article{BoyerPorezMauny2018,
  author  = {Boyer, Fr{\'e}d{\'e}ric and Porez, Mathieu and Mauny, Johan},
  title   = {Reduced Dynamics of the Non-holonomic Whipple Bicycle},
  journal = {Journal of Nonlinear Science},
  volume  = {28},
  number  = {3},
  pages   = {943--983},
  year    = {2018},
  doi     = {10.1007/s00332-017-9434-x}
}

@Book{Cosserat1909,
  author    = {Cosserat, E. and Cosserat, F.},
  publisher = {A. Hermann et fils},
  title     = {Theorie des corps d{\'e}dormables},
  year      = {1909},
  url       = {https://books.google.fr/books?id=U3e4AAAAIAAJ},
}

@Article{Antman76,
title={Ordinary differential equations of nonlinear elasticity I: Foundations of the theories of
non-linearly elastic rods and shell},
author={S. S. Antman},
journal={Arch. Rat. Mech. Anal.},
year={1976},
volume={61},
number={4},
pages={307-351},
}

@Article{Simo1988,
  author    = {J.C. Simo and L. Vu-Quoc},
  journal   = {Computer Methods in Applied Mechanics and Engineering},
  title     = {On the dynamics in space of rods undergoing large motions {\textemdash} A geometrically exact approach},
  year      = {1988},
  month     = {feb},
  number    = {2},
  pages     = {125--161},
  volume    = {66},
  doi       = {10.1016/0045-7825(88)90073-4},
  publisher = {Elsevier {BV}},
}

@book{Truesdell1991,
  author    = {Truesdell, Clifford A.},
  title     = {A First Course in Rational Continuum Mechanics},
  volume    = {1},
  series    = {Pure and Applied Mathematics},
  number    = {71},
  edition   = {2},
  publisher = {Academic Press},
  address   = {Boston},
  year      = {1991}
}

@Book{Murray2017,
  author    = {Murray, Richard M and Li, Zexiang and Sastry, S Shankar},
  publisher = {CRC press},
  title     = {A mathematical introduction to robotic manipulation},
  year      = {2017},
}

@Book{MerodioRosato2012,
  author    = {Merodio, Jose and Rosato, Anthony D.},
  title     = {General Overview of Continuum Mechanics},
  booktitle = {Continuum Mechanics},
  volume    = {I},
  pages     = {1--52},
  publisher = {EOLSS Publishers},
  year      = {2012}
}

@BOOK{MarsdenHughes,
   author = {J. E. Marsden and T. J. R. Hughes},
   title= {Mathematical Foundations of Elasticity},
   publisher = {Dover edition},
   edition={First},
   year = {1994} }

@misc{rodriguez2020boundaryfeedbackstabilizationintrinsic,
      title={Boundary feedback stabilization for the intrinsic geometrically exact beam model},
      author={Charlotte Rodriguez and Günter Leugering},
      year={2020},
      eprint={1912.02543},
      archivePrefix={arXiv},
      primaryClass={math.AP},
      url={https://arxiv.org/abs/1912.02543},
}

@article{lewis1995variational,
  title={Variational principles for constrained systems: theory and experiment},
  author={Lewis, Andrew D and Murray, Richard M},
  journal={International Journal of Non-Linear Mechanics},
  volume={30},
  number={6},
  pages={793--815},
  year={1995},
  publisher={Elsevier}
}

@book{sharpe2000differential,
  title={Differential geometry: Cartan's generalization of Klein's Erlangen program},
  author={Sharpe, Richard W},
  volume={166},
  year={2000},
  publisher={Springer Science \& Business Media}
}

@book{bastin2016stability,
  title={Stability and boundary stabilization of 1-D hyperbolic systems},
  author={Bastin, Georges and Coron, Jean-Michel},
  volume={88},
  year={2016},
  publisher={Springer}
}

@article{Damping2013,
  author  = {Linn, J. and Lang, H. and Tuganov, A.},
  title   = {Geometrically exact Cosserat rods with Kelvin--Voigt type viscous damping},
  journal = {Mechanics Sciences},
  volume  = {4},
  number  = {1},
  pages   = {79--96},
  year    = {2013},
  doi     = {10.5194/ms-4-79-2013}
}

@incollection{Antman1984,
  author    = {Antman, Stuart S.},
  title     = {The Theory of Rods},
  booktitle = {Mechanics of Solids},
  volume    = {II},
  editor    = {Truesdell, Clifford},
  pages     = {641--703},
  publisher = {Springer-Verlag},
  address   = {Berlin, Heidelberg},
  year      = {1984}
}

@book{meirovitch2010methods,
  title={Methods of analytical dynamics},
  author={Meirovitch, Leonard},
  year={2010},
  publisher={Courier Corporation}
}

@article{ko2017modeling,
  title={Modeling polymorphic transformation of rotating bacterial flagella in a viscous fluid},
  author={Ko, William and Lim, Sookkyung and Lee, Wanho and Kim, Yongsam and Berg, Howard C and Peskin, Charles S},
  journal={Physical Review E},
  volume={95},
  number={6},
  pages={063106},
  year={2017},
  publisher={APS}
}

@book{jarchow2012locally,
  title={Locally convex spaces},
  author={Jarchow, Hans},
  year={2012},
  publisher={Springer Science \& Business Media}
}

@misc{edvardsson:2016,
title={The Noether theorem},
note={Analytical mechanics - FYGB08},
author={Edvardsson, Elisabet},
year={2016},
howpublished={\url{https://jfuchs.hotell.kau.se/kurs/amek/prst/15_nthm.pdf}}
}

@article{tummers2023cosserat,
  title={Cosserat rod modeling of continuum robots from newtonian and lagrangian perspectives},
  author={Tummers, Matthias and Lebastard, Vincent and Boyer, Fr{\'e}d{\'e}ric and Troccaz, Jocelyne and Rosa, Benoit and Chikhaoui, M Taha},
  journal={IEEE Transactions on Robotics},
  volume={39},
  number={3},
  pages={2360--2378},
  year={2023},
  publisher={IEEE}
}

@book{brogliato1996nonsmooth,
  title={Nonsmooth impact mechanics: models, dynamics and control},
  author={Brogliato, Bernard},
  year={1996},
  publisher={Springer}
}

@INPROCEEDINGS{1087247,
  author={Khatib, O.},
  booktitle={Proceedings. 1985 IEEE International Conference on Robotics and Automation},
  title={Real-time obstacle avoidance for manipulators and mobile robots},
  year={1985},
  volume={2},
  number={},
  pages={500-505},
  doi={10.1109/ROBOT.1985.1087247}}

\end{document}